\documentclass[pdflatex,sn-mathphys-ay]{sn-jnl}

\usepackage{amsthm,amsmath,amsfonts,amssymb}
\usepackage{graphicx}
\usepackage{subcaption}
\usepackage{multirow,multicol}
\usepackage{enumitem}
\usepackage{accents}
\usepackage{bbm,bm}
\usepackage{float}
\usepackage{scalerel}
\makeatletter
\providecommand{\cline}[1]{\@cline#1\@nil}
\makeatother

\theoremstyle{thmstyleone}      
\newtheorem{theorem}{Theorem}
\newtheorem{proposition}{Proposition}
\newtheorem{lemma}{Lemma}
\newtheorem{corollary}{Corollary}

\theoremstyle{thmstyletwo}      
\newtheorem{definition}{Definition}
\newtheorem{property}{Property}
\newtheorem{remark}{Remark}

\def \a {\alpha}
\def \b {\beta}
\def\Var{{\mathrm{Var}}}
\newcommand\ddfrac[2]{\frac{\displaystyle #1}{\displaystyle #2}}
\newcommand\wh[1]{\hstretch{2}{\hat{\hstretch{.5}{#1}}}}
\def \hp {\wh{p}}
\def \hpi {\wh{\pi}}
\def \Sz {\boldsymbol{\Omega}}
\def \Sr {\boldsymbol{\Gamma}}
\def \Sx {\boldsymbol{\Omega}_{\square}}
\def \nmin {n_{\scriptscriptstyle\blacktriangledown}}
\def \nmax {n_{\scriptscriptstyle\blacktriangle}}
\newcommand{\EO}{\textup{EO}}
\newcommand{\EOlong}{elliptically optimal}
\def \Lone {L_1}  
\def \Ltwo {L_2}  
\def \Lthree {L_3}
\def \Uone {U_1}  
\def \Utwo {U_2}  
\def \Uthree {U_3}
\def \Rone {R_1}  
\def \Rtwo {R_2}
\def \Lsone {L_1^*} 
\def \Lstwo {L_2^*}
\def \Usone {U_1^*} 
\def \Ustwo {U_2^*}

\begin{document}

\title[Elliptically Optimal Confidence Interval]{The Elliptically Optimal
Confidence Interval: A Bivariate Extension of Wilson's Score Method}

\author*{\fnm{Nawaf} \sur{Mohammed}}\email{nawaf.mohammed.ac@gmail.com}

\affil*{\orgname{Independent Researcher}, \state{Ontario}, \country{Canada}}

\abstract{Constructing a confidence interval for the difference between two
independent binomial proportions involves a nuisance direction that is not
identified by the estimand. The one-sample Wilson score interval inverts a
scalar score test, but has no direct bivariate analogue isolating the
difference: inverting the joint normal approximation yields an elliptical region
in the unit square, whereas the estimand \(p_1-p_2\) is one-dimensional. We
define the \EOlong{} (\EO) confidence interval as the range of \(p_1-p_2\) over
this region and solve the resulting optimization problem in closed form,
obtaining explicit bounds in six mutually exclusive and exhaustive cases. The
solution admits a compact characterization: the \EO{} interval is the score
interval obtained by maximizing over the nuisance variance rather than
estimating it. It is therefore the shortest interval obtained by projecting the
elliptical region, and inherits its coverage guarantee. We derive the exact
coverage excess, \(2[\Phi(z\mathcal{R})-\Phi(z)]\), where \(\mathcal{R}\) is the
ratio of the least-favourable to the true standard deviation. The excess
vanishes on an explicit line through the parameter space, is bounded by \(\a\),
and is invariant under proportional scaling of the sample sizes. Exact
enumeration of the binomial coverage shows that the Wald interval, whose
variance estimator is downward biased by a factor \(1-1/n\) under balanced
allocation, falls below nominal coverage almost everywhere. The \EO{} interval
never under-covers under the normal approximation and always yields admissible,
non-degenerate bounds. Its price is over-coverage when both proportions are
extreme, which we quantify exactly.
}



\maketitle

\section{Introduction}
\label{sec:intro}

Comparing two independent binomial proportions is among the most frequently
performed inferential tasks in the empirical sciences, and in most applications
the quantity of substantive interest is not a $p$-value but an interval estimate
for the difference between the two underlying success probabilities. Despite the
elementary appearance of the problem, the construction of such an interval is
subtle: the sampling variance of the estimated difference is itself a function of
the unknown parameters, so any interval procedure must decide how to handle a
nuisance quantity that cannot be observed. Almost all of the difficulty --- and
almost all of the difference between competing methods --- can be traced back to
that single decision.

The oldest resolution simply replaces the unknown parameters by their sample
counterparts inside the standard error, which yields the familiar Wald
interval.\footnote{Wald's method estimates the unknown standard error by
substituting the observed sample proportions for the parameters, and then centres
a symmetric interval at the point estimate. Its simplicity has made it a staple of
introductory courses, but the two layers of approximation on which it rests are
known to produce erratic coverage, and bounds that may fall outside the parameter
space, for small samples or extreme proportions \citep{Agresti2000,Newcombe1998}.} \citet{Wilson1927} declined to make that substitution. For a single
population he instead retained the unknown parameter inside the variance and
treated the standardised normal-approximation inequality as a relation to be
\emph{solved} for the parameter itself. Squaring the standardised deviation and
collecting terms converts the approximation into a quadratic inequality in the
unknown proportion whose leading coefficient is strictly positive; its solution
set is therefore the closed interval lying between the two roots of the
associated quadratic, and both of those roots can be shown to lie inside the unit
interval \citep{ONeill2021}. The resulting score interval is never degenerate,
never leaves the parameter space, and behaves far more reliably than Wald's
across the whole range of sample sizes and parameter values. Just as importantly
for what follows, its endpoints admit an equivalent \emph{variational}
description: they are the smallest and the largest values of the parameter that
are compatible with the quadratic constraint and with the parameter space. The
endpoints of Wilson's interval are thus not algebraic accidents but genuine
extremal bounds --- and it is precisely this characterisation, rather than the
closed-form root expressions, that generalises.

When the same programme is carried out for two independent populations, the
algebra proceeds identically but the geometry does not. Retaining both unknown
parameters inside the variance of the estimated difference and squaring produces
a quadratic inequality in \emph{two} variables; its discriminant is negative for
every admissible sample size and confidence level, so the associated conic is
always an ellipse, and the solution set is that ellipse together with its
interior, intersected with the unit square. This bivariate region --- rather than
an interval --- is the exact two-population analogue of Wilson's quadratic
constraint. Herein lies the essential obstruction. The object we wish to report
is one-dimensional, an interval for the difference of the two proportions,
whereas the object the normal approximation actually delivers is two-dimensional.
Because of this mismatch of dimensions there is no exact bivariate counterpart of
Wilson's inversion, and some form of reduction from the plane to the line is
unavoidable.

The literature contains a substantial number of alternatives to the Wald interval
for the difference of two proportions; see, among others,
\citep{Agresti2000,Beal1987,Fagerland2011,Mee1984,Miettinen1985,Pan2002}. Closest in spirit to the present work is
Method~10 of \citet{Newcombe1998}, which addresses the two-population
problem by combining the univariate Wilson bounds computed separately for each
proportion. That construction is simple, well studied and performs creditably in
practice, but it is a recombination of two one-dimensional solutions rather than a
genuine inversion of the two-dimensional constraint; the bivariate region itself
plays no direct role in it. What appears to be missing from the literature is a
method that takes Wilson's variational characterisation seriously and applies it,
unmodified, in the bivariate setting.

That is the approach adopted here. We keep Wilson's optimisation formulation
verbatim and simply enlarge its feasible set: the endpoints of our interval are
defined as the minimum and the maximum of the difference of the two proportions
taken over the elliptical region intersected with the unit square. This reduction
from the plane to the line is not arbitrary. The elliptical region is compact and
convex, being the intersection of the sublevel set of a positive-definite
quadratic form with the unit square, and it is never empty, since the observed
pair of sample proportions always belongs to it. Consequently its image under the
linear map that sends a pair of proportions to their difference is exactly the
closed interval between the two optimal values. Our interval is therefore the
\emph{shortest possible} interval with the property that membership of the
elliptical region implies membership of the interval: no shorter interval can
inherit the coverage guarantee carried by that region. It is in this precise, and
deliberately circumscribed, sense that we call the resulting procedure the
\emph{\EOlong} (\EO) confidence interval --- optimal \emph{given} the elliptical
constraint, rather than optimal among all conceivable procedures for this problem.

Two consequences of this construction should be flagged at the outset. First, the
set of parameter pairs whose difference falls inside the \EO{} interval contains
the elliptical region, so the interval is conservative: its coverage is
at least the nominal level, up to the error of the underlying normal
approximation. Conservatism is the price of the dimensional reduction, and it is
not a defect to be concealed but a quantity to be measured; accordingly, a
substantial part of this paper is devoted to characterising the discrepancy
between nominal and true coverage explicitly, and to doing so before, rather than
after, comparing the interval with its competitors. Second, the construction
turns out to be remarkably parsimonious: the resulting bounds depend on the two
observed proportions only through their difference, a structural feature we
establish and discuss in due course.

The remainder of the paper is organised as follows. Section~\ref{sec:quadratic}
derives the quadratic inequalities underlying the univariate and bivariate Wilson
score constructions, and introduces the elliptical region together with the
optimisation problem that defines the \EO{} interval. In
Section~\ref{sec:properties_of_omega_lambda} we establish the general geometric
properties of the unrestricted region $\Sz$ and of its restriction $\Sx$ to the
unit square, and illustrate their behaviour graphically. This leads to
Section~\ref{sec:the_EO_CI}, where we solve the optimisation problem in closed
form and record a range of properties and remarks. Section~\ref{sec:error_analysis} derives an exact closed-form expression for the
discrepancy between the nominal and the true coverage probability of the
interval. Armed with these results,
Section~\ref{sec:comparison} compares the \EO{} interval with the Wald and
Newcombe intervals on the basis of realised coverage rather than width alone.
Finally, Section~\ref{sec:conclusions} summarises the main findings and their
implications.

\section{Quadratic inequalities for the univariate and bivariate Wilson score intervals}
\label{sec:quadratic}

\subsection{The univariate case}
\label{subsec:univariate}

Consider a population with parameter $p\in[0,1]$, representing the proportion of
individuals possessing a certain characteristic. To construct a confidence
interval (CI) for $p$, we draw a sample of size $n>0$ and compute the sample
proportion $\hp\in[0,1]$. For a chosen confidence level $(1-\a)100\%$, the normal
approximation to the sampling distribution of $\hp$ gives
\begin{equation}
\label{eq:mainprob}
\mathbb{P}\left(-z\le \ddfrac{\hp-p}{\sqrt{\Var[\hp]}}\le z\right)\approx 1-\a,
\end{equation}
where $z=z_{\a/2}>0$ is the $z$-score for a significance level $\a\in(0,1)$ and
$\Var[\hp]=\ddfrac{p(1-p)}{n}$ is the variance of $\hp$.

The decisive feature of \eqref{eq:mainprob} is that the unknown parameter $p$
appears both in the numerator and, through the variance, in the denominator of the
standardised quantity. Wilson's insight \citep{Wilson1927} was that this is an
asset rather than an obstacle: instead of estimating the denominator, one may
retain it and solve the resulting relation for $p$. Since the two inequalities in
\eqref{eq:mainprob} are equivalent to a single inequality on the absolute
deviation,
\begin{equation}
\label{eq:mainprob2}
\mathbb{P}\left(\ddfrac{|\hp-p|}{\sqrt{\Var[\hp]}}\le z\right)\approx 1-\a,
\end{equation}
and both sides of the inner inequality are non-negative, squaring is a reversible
operation and yields
\begin{equation}
\label{eq:mainprob3}
\mathbb{P}\left(\ddfrac{(\hp-p)^2}{{\Var[\hp]}}\le z^2\right)\approx 1-\a.
\end{equation}
Clearing the denominator and collecting like powers of $p$, the event in
\eqref{eq:mainprob3} becomes a quadratic inequality in the unknown parameter:
\begin{equation}
\label{eq:parabolaprob}
\mathbb{P}\left(ap^2+bp+c\le 0\right)\approx 1-\a,
\end{equation}
with $a=1+\ddfrac{z^2}{n}$, $b=-2\hp-\ddfrac{z^2}{n}$ and $c=\hp^2$. The leading
coefficient is strictly positive, so the associated parabola is convex and the
solution set of the inequality is precisely the closed interval bounded by the two
roots of the quadratic. Solving explicitly gives the Wilson score interval,
\begin{equation}
\label{eq:WilsonCIOnePop}
\frac{n\hp+z^2/2}{n+z^2}-\frac{z\sqrt{z^2+4 n \hp(1-\hp)}}{2 \left(n+z^2\right)}\le p\le \frac{n\hp+z^2/2}{n+z^2}+\frac{z\sqrt{z^2+4 n \hp(1-\hp)}}{2 \left(n+z^2\right)}.
\end{equation}

Both roots in \eqref{eq:WilsonCIOnePop} are real and lie within the unit interval,
so that Wilson's interval is always well defined and admissible, irrespective of
the sample size or of how extreme the observed proportion may be; a detailed
account of this and of many further properties of the interval is given by
\citet{ONeill2021}.\footnote{The argument is elementary and worth recording.
Writing $q(p)=ap^{2}+bp+c$, one has the pleasing identities $q(0)=\hp^{2}\ge0$ and
$q(1)=a+b+c=(1-\hp)^{2}\ge0$, while the discriminant
$b^{2}-4ac=z^{2}\bigl(z^{2}+4n\hp(1-\hp)\bigr)/n^{2}\ge0$ guarantees that the roots
are real, and the vertex $-b/(2a)=(2n\hp+z^{2})/\bigl(2(n+z^{2})\bigr)$ lies in
$[0,1]$. A convex parabola that is non-negative at both endpoints of $[0,1]$ and
attains its minimum inside $[0,1]$ must have both of its real roots in $[0,1]$.}

Beyond their closed form, the bounds in \eqref{eq:WilsonCIOnePop} admit an
equivalent variational description that will be the point of departure for the
bivariate construction. Writing
$\mathcal{F}=\bigl\{q\in[0,1]:aq^{2}+bq+c\le0\bigr\}$ for the feasible set induced
by \eqref{eq:parabolaprob} together with the parameter space, the Wilson bounds
solve
\begin{equation}
\label{eq:WilsonproblemCI}
\min_{q\in\mathcal{F}}\,q\;\le\;p\;\le\;\max_{q\in\mathcal{F}}\,q .
\end{equation}
Thus the endpoints of Wilson's interval are not merely algebraic artefacts but
arise as the optimal values of a constrained minimisation and maximisation problem
over the set of parameter values compatible with the data. This characterisation
establishes them as true extremal bounds, and, unlike the explicit root formulae,
it transfers verbatim to higher dimensions.

\subsection{The bivariate case}
\label{subsec:bivariate}

Now suppose we have two independent populations with unknown parameters $p_1$ and
$p_2$, each representing the proportion of individuals possessing a certain
characteristic. Both parameters lie in the interval $[0,1]$ --- although this
restriction is relaxed when discussing more general theoretical properties later
on. Our goal is to construct a CI for their difference, $p_1-p_2$. To do so, we
collect two samples of sizes $n_1>0$ and $n_2>0$, yielding the respective sample
proportions $\hp_1$ and $\hp_2$, each lying in $[0,1]$.

For notational convenience, from here onward we set $p=p_1-p_2$ and
$\hp=\hp_1-\hp_2$, with $p,\hp\in[-1,1]$, and define
\begin{equation*}
\nmin\equiv\min(n_1,n_2),
\qquad
\nmax\equiv\max(n_1,n_2),
\end{equation*}
the downward- and upward-pointing markers serving as a mnemonic for the smaller
and the larger of the two sample sizes; as before, $z=z_{\a/2}>0$ for a
significance level $\a\in(0,1)$. Under this notation the classical normal
approximation takes the same form as \eqref{eq:mainprob}, with the variance of the
difference given by
$\Var[\hp]=\ddfrac{p_1(1-p_1)}{n_1}+\ddfrac{p_2(1-p_2)}{n_2}$.

Retaining both unknown parameters inside this variance and repeating the steps of
Section~\ref{subsec:univariate} --- rewriting \eqref{eq:mainprob} as
\eqref{eq:mainprob2}, squaring as in \eqref{eq:mainprob3}, and grouping like terms
--- the event of interest takes the equivalent form
\begin{equation}
\label{eq:ellipseprob}
\mathbb{P}\left(Ap_1^2+Bp_1p_2+Cp_2^2+Dp_1+Ep_2+F\le0\right)\approx 1-\a,
\end{equation}
where $A=1+\ddfrac{z^2}{n_1}$, $B=-2$, $C=1+\ddfrac{z^2}{n_2}$,
$D=-2\hp-\ddfrac{z^2}{n_1}$, $E=2\hp-\ddfrac{z^2}{n_2}$, and $F=\hp^2$. The
inequality in \eqref{eq:ellipseprob} defines a quadratic curve in two variables
together with its interior --- that is, a conic section and its enclosed region.
To determine the nature of that conic we examine the discriminant,
\begin{gather*}
B^2-4AC=-\frac{4 z^2 \left(n_1+n_2+z^2\right)}{n_1n_2}<0,
\end{gather*}
which is strictly negative for every admissible pair of sample sizes and every
confidence level, so that the curve is always an ellipse; equivalently, the
quadratic form is positive definite, and the region it bounds is compact and
convex. Let us denote the (unrestricted) ellipse together with its interior by
\begin{equation*}
\Sz=\left\{(p_1,p_2)\in\mathbb{R}^2:Ap_1^2+Bp_1p_2+Cp_2^2+Dp_1+Ep_2+F\le0\right\},
\end{equation*}
and define its restriction to the unit square, namely the subset
\begin{equation*}
\Sx=\left\{(p_1,p_2)\in[0,1]^2:Ap_1^2+Bp_1p_2+Cp_2^2+Dp_1+Ep_2+F\le0\right\},
\end{equation*}
which is the intersection of $\Sz$ with the parameter space, $\Sx=\Sz\cap[0,1]^2$.
The mnemonic subscript is intended to remind the reader that $\Sx$ differs from
$\Sz$ only in being confined to the square. With this notation, \eqref{eq:ellipseprob}
may be written compactly as
\begin{equation}
\label{eq:ellipseprob_new}
\mathbb{P}((p_1,p_2)\in\Sx)\approx 1-\a.
\end{equation}
We note in passing that $\Sx$ is never empty: substituting $(p_1,p_2)=(\hp_1,\hp_2)$
annihilates the squared deviation and leaves $-z^{2}\widehat{\Var}[\hp]\le0$, so the
observed pair of sample proportions always belongs to $\Sx$. Being a non-empty,
bounded, closed and convex set, $\Sx$ is exactly the kind of object
over which the extremal programme \eqref{eq:WilsonproblemCI} is well posed.

Our aim is to establish a Wilson-type CI for $p$ that preserves the structure of
the set $\Sx$ in \eqref{eq:ellipseprob_new}. In the single-population case,
\eqref{eq:parabolaprob} defines a subset of the real line, which matches exactly
the dimensionality of the object we wish to report. In the two-population case,
however, the main difficulty arises from a mismatch between dimensions: our
objective is a one-dimensional solution --- an interval for the difference $p$ ---
whereas the underlying formulation is inherently two-dimensional, represented by
the set $\Sx$. Consequently, unlike in the univariate case, an exact CI for $p$ is
not attainable, and a reduction from the plane to the line must be made explicit
and defended.

To perform that reduction we adopt the optimisation-based formulation of
\eqref{eq:WilsonproblemCI}, applied without modification to the bivariate feasible
set. Our proposed interval is therefore defined by the principle
\begin{equation}
\label{eq:mainproblemCI}
\min_{(q_1,q_2)\in\Sx}\,(q_1-q_2)\;\le\;p\;\le\;\max_{(q_1,q_2)\in\Sx}\,(q_1-q_2),
\end{equation}
which extends Wilson's method to the bivariate setting directly and intuitively.
The interval \eqref{eq:mainproblemCI} induces a set of parameter pairs
$(p_1,p_2)\in[0,1]^2$ whose difference lies between the two extreme values,
\begin{equation*}
\Sr=\Bigl\{(p_1,p_2)\in[0,1]^{2}:\ \min_{(q_1,q_2)\in\Sx}(q_1-q_2)\ \le\ p_1-p_2\ \le\ \max_{(q_1,q_2)\in\Sx}(q_1-q_2)\Bigr\},
\end{equation*}
and by construction $\Sx\subseteq\Sr$. The interval is therefore conservative: the
event that $p$ falls inside it is implied by, but not equivalent to, the event that
$(p_1,p_2)$ falls inside the elliptical region, so its coverage is at least the
nominal level up to the error of the normal approximation. We shall refer to this
novel interval as the \emph{\EOlong} (\EO) confidence interval.

The qualifier deserves an immediate and precise justification. Since $\Sx$ is
non-empty, compact and convex, and since $(q_1,q_2)\mapsto q_1-q_2$ is linear, the
image of $\Sx$ under this map is exactly the closed interval appearing in
\eqref{eq:mainproblemCI}. Every point of that interval is attained by some
parameter pair compatible with the data, and no proper subinterval retains the
containment $\Sx\subseteq\Sr$. The \EO{} interval is thus the shortest interval
that inherits the coverage guarantee carried by the elliptical region: it is
optimal \emph{relative to} the elliptical constraint, which is the only sense in
which the term is used in this paper. Conservatism and this form of optimality are
not in tension --- the excess coverage is attributable entirely to the inclusion
$\Sx\subseteq\Sr$ forced by the reduction in dimension, and is not an artefact of
a slack choice of endpoints. Quantifying that excess exactly is the object of
Section~\ref{sec:error_analysis}, and it is on the basis of those results, rather
than on interval width alone, that Section~\ref{sec:comparison} draws its
comparisons with competing procedures.

\section{Properties of the ellipse \texorpdfstring{$\Sz$}{Omega} and its unit-square subset \texorpdfstring{$\Sx$}{Omega-square}}
\label{sec:properties_of_omega_lambda}

The reduction carried out in Section~\ref{sec:quadratic} replaced a
one-dimensional constraint by a planar one, and in doing so it transferred the
whole inferential problem onto the geometry of a single convex region. Before the
extremal programme \eqref{eq:mainproblemCI} can be solved --- and before the
conservatism it entails can be quantified --- we must understand what that region
looks like: whether it is convex, whether it is ever empty, how much of it escapes
the parameter space, where it sits, how elongated it is and in which direction it
leans. This section answers each question in turn. The properties established
here are not merely descriptive. Every one of them is used either in the
derivation of the interval in Section~\ref{sec:the_EO_CI} or in the coverage
analysis of Section~\ref{sec:error_analysis}, and one object introduced below ---
the line $\ell$ of \eqref{eq:centreline} --- turns out to organise the entire
remainder of the paper.

We begin with an observation that governs everything which follows, and which is
best made before any picture is drawn.

\begin{remark}
\label{rem:difference_only}
Each of the six coefficients of the quadratic form in \eqref{eq:ellipseprob}
depends on the observed data only through the difference $\hp=\hp_1-\hp_2$;
neither sample proportion enters separately. Consequently the sets $\Sz$ and
$\Sx$ --- and hence every quantity derived from them, including the bounds of the
interval defined by \eqref{eq:mainproblemCI} --- are functions of
$(n_1,n_2,z,\hp)$ alone. Two experiments with markedly different individual
proportions but a common difference generate exactly the same elliptical region.
\end{remark}

Throughout this section it will be convenient to write
\begin{equation}
\label{eq:Vfunction}
V(p_1,p_2)=\ddfrac{p_1(1-p_1)}{n_1}+\ddfrac{p_2(1-p_2)}{n_2},
\end{equation}
so that, directly from the derivation preceding \eqref{eq:ellipseprob},
\begin{equation}
\label{eq:Omega_variance_form}
\Sz=\Bigl\{(p_1,p_2)\in\mathbb{R}^2:\ \bigl(\hp-p_1+p_2\bigr)^2\le z^2V(p_1,p_2)\Bigr\},
\end{equation}
with $\Sx=\Sz\cap[0,1]^2$. Form \eqref{eq:Omega_variance_form} is equivalent to
the coefficient form of \eqref{eq:ellipseprob} but is frequently more transparent,
since it separates the discrepancy between the observed and the hypothesised
difference from the variance that calibrates it. We shall move between the two
forms freely.

Two symmetries of \eqref{eq:Omega_variance_form} will be used repeatedly, and both
are immediate from that display. The first is a reflection across the main
diagonal accompanied by an interchange of the two populations,
\begin{equation}
\label{eq:mirror}
(p_1,p_2)\in\Sz(\hp;n_1,n_2)
\iff
(p_2,p_1)\in\Sz(-\hp;n_2,n_1),
\end{equation}
which allows every statement below to be proved for $\hp\ge0$ without loss of
generality. The second is a point reflection through the centre of the unit
square, with the sample sizes left untouched,
\begin{equation}
\label{eq:pointreflection}
(p_1,p_2)\in\Sz(\hp;n_1,n_2)
\iff
(1-p_1,1-p_2)\in\Sz(-\hp;n_1,n_2),
\end{equation}
which holds because $V(1-p_1,1-p_2)=V(p_1,p_2)$ while the squared discrepancy is
unchanged. Symmetry \eqref{eq:pointreflection} is the source of the central
symmetry of the coverage-error surface recorded in
Corollary~\ref{cor:ErrorProperties}$(c)$.

Before proceeding to the analysis, we visualise both sets for selected input
values.
\begin{figure}[htbp]
\centering
    \begin{subfigure}[t]{0.48\textwidth}
        \centering
        \includegraphics[width=\linewidth]{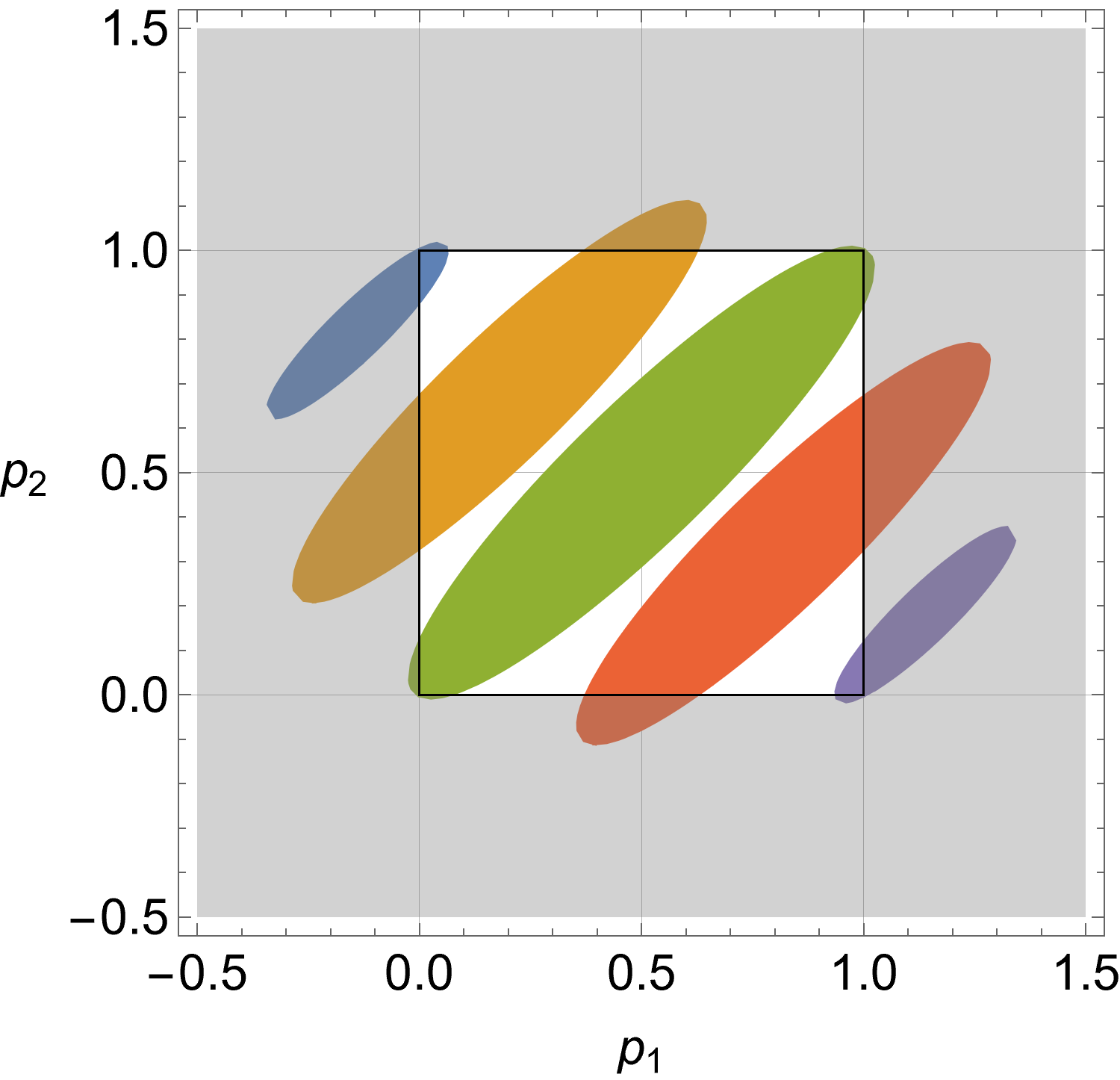}
        \caption{$n_1>n_2$ $(n_1=60,\ n_2=30)$}
        \label{subfig:setLF-ngm}\label{subfig:setLS-ngm}
    \end{subfigure}
    \hfill
    \begin{subfigure}[t]{0.48\textwidth}
        \centering
        \includegraphics[width=\linewidth]{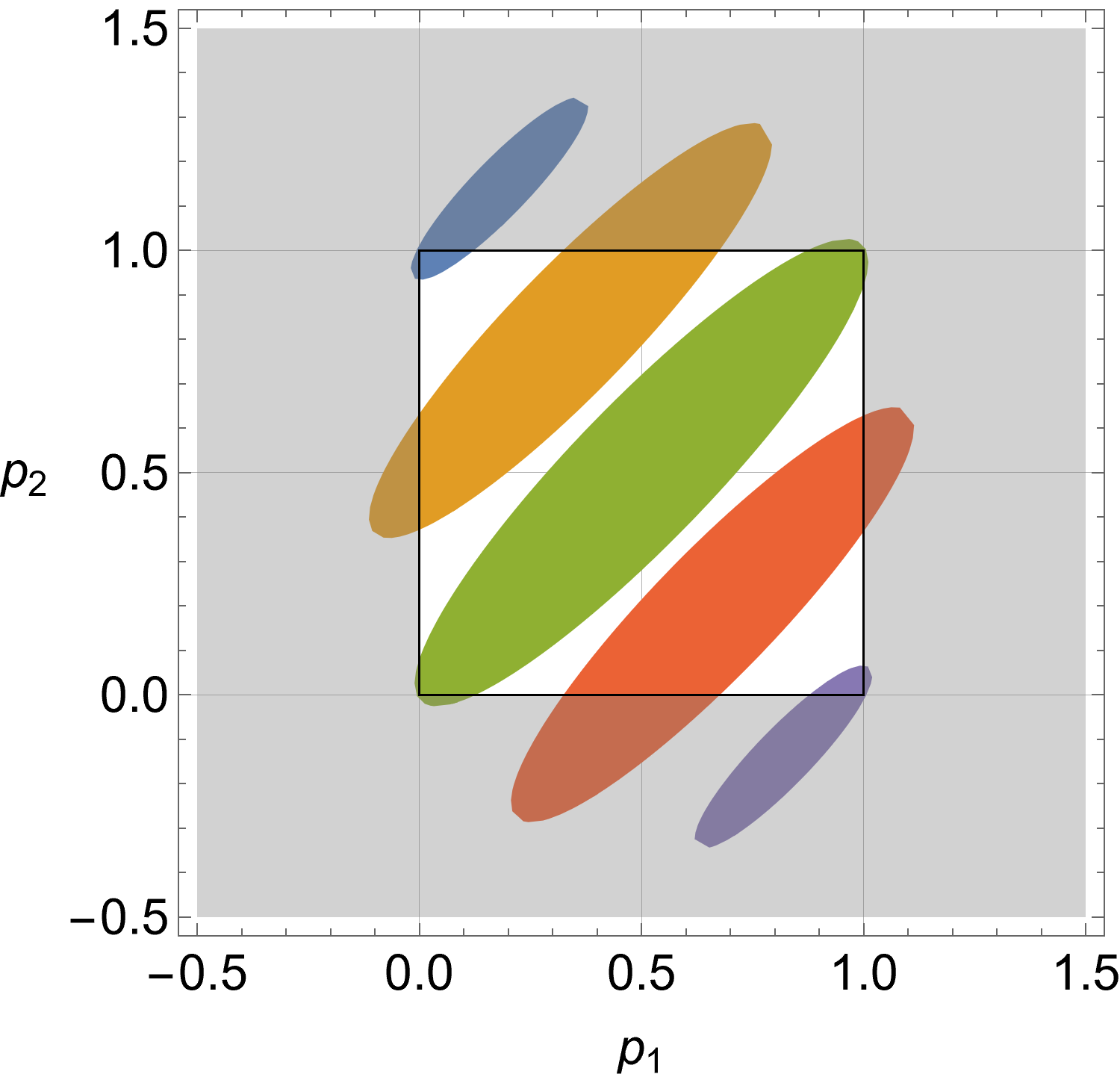}
        \caption{$n_1<n_2$ $(n_1=30,\ n_2=60)$}
        \label{subfig:setLF-mgn}\label{subfig:setLS-mgn}
    \end{subfigure}

    \vspace{0.8\baselineskip}

    \begin{subfigure}[t]{0.48\textwidth}
        \centering
        \includegraphics[width=\linewidth]{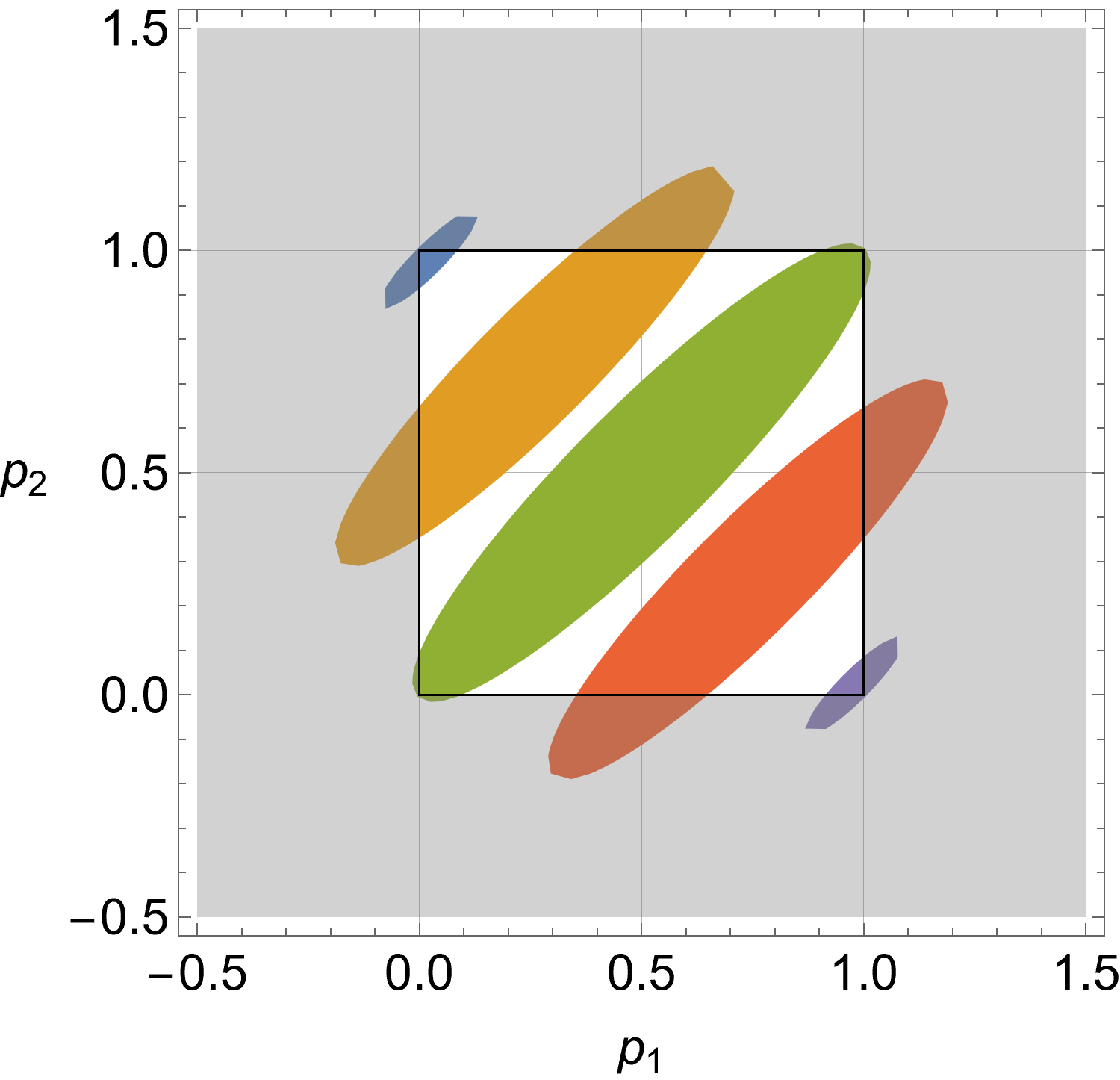}
        \caption{$n_1=n_2$ $(n_1=45,\ n_2=45)$}
        \label{subfig:setLF-nem}\label{subfig:setLS-nem}
    \end{subfigure}

    \vspace{0.8\baselineskip}

    \includegraphics[width=0.45\textwidth]{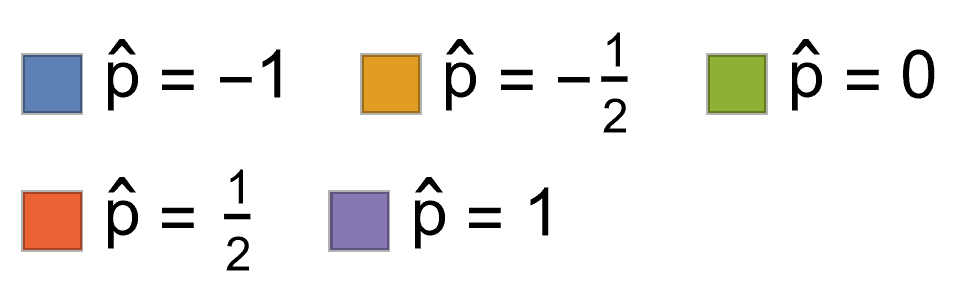}

\caption[The sets $\Sz$ and $\Sx$ for various sample sizes and proportions]{The
set $\Sz$ for three orderings of the sample sizes and the five values
$\hp\in\left\{-1,-\tfrac12,0,\tfrac12,1\right\}$, with $\a$ fixed at $5\%$. In
each panel the region outside the unit square is masked in grey: the unshaded
portion of each ellipse is therefore the corresponding subset
$\Sx=\Sz\cap[0,1]^2$. The colour key applies to all three panels.}
\label{fig:setLF}\label{fig:setLS}
\end{figure}

Figure~\ref{fig:setLF} displays, for each ordering of the two sample sizes, five
instances of $\Sz$, corresponding to the values
$\hp\in\left\{-1,-\tfrac12,0,\tfrac12,1\right\}$; by
Remark~\ref{rem:difference_only} no further indexing is needed, since $\hp$ alone
determines the set once $n_1$, $n_2$ and $z$ are fixed. Three features are
apparent, and each is confirmed analytically below. The ellipses are markedly
elongated, their major axes leaning counter-clockwise from the horizontal by an
angle strictly between $0$ and $\pi/2$; they drift diagonally downward as $\hp$
increases; and each of them escapes the unit square, although --- as
Section~\ref{subsec:properness} makes precise --- the escape may be visually
imperceptible when $\hp$ is near zero and pronounced when $\left|\hp\right|$ is
large. The elongation has a simple statistical reading: the data pin down the
\emph{difference} of the two proportions far more sharply than their common level,
so the region is thin transverse to the lines of constant difference and long
along them.

A convenient way to begin the analysis is to intersect $\Sz$ with the two
diagonals of the unit square. Remarkably, along either diagonal the bivariate
constraint collapses to a univariate Wilson problem, governed by a single
effective sample size.

\begin{lemma}
\label{lem:diagonal_sections}
Define the \emph{effective sample size}
\begin{equation}
\label{eq:nstar}
n_{\ast}=\frac{4n_1n_2}{n_1+n_2},
\end{equation}
that is, twice the harmonic mean of $n_1$ and $n_2$; it satisfies
$n_{\ast}\le n_1+n_2$, with equality if and only if $n_1=n_2$. Then:
\begin{itemize}
\item[$(a)$] On the opposite diagonal $\{(p_1,p_2)\in\mathbb{R}^2:p_1+p_2=1\}$,
membership of $\Sz$ is equivalent to
\begin{equation*}
\left(\frac{1+\hp}{2}-p_1\right)^{2}\le\frac{z^2\,p_1(1-p_1)}{n_{\ast}},
\end{equation*}
which is precisely the quadratic \eqref{eq:parabolaprob} of Wilson's univariate
method for an observed proportion $\hpi=(1+\hp)/2\in[0,1]$ and a sample of size
$n_{\ast}$. Consequently the section is a closed interval of strictly positive
length whose endpoints are given by \eqref{eq:WilsonCIOnePop} with $n$ replaced by
$n_{\ast}$ and $\hp$ by $\hpi$, and both endpoints lie in $[0,1]$. Its relative
interior is a non-empty open subinterval of $(0,1)$.
\item[$(b)$] On the main diagonal $\{(q,q):q\in\mathbb{R}\}$, membership of $\Sz$
is equivalent to $q(1-q)\ge \hp^{2}n_{\ast}/(4z^{2})$. This section is non-empty
if and only if
\begin{equation*}
\left|\hp\right|\le\frac{z}{\sqrt{n_{\ast}}},
\end{equation*}
in which case it is the closed interval centred at $q=1/2$ of half-length
$\sqrt{1/4-\hp^{2}n_{\ast}/(4z^{2})}$, and it is contained in $[0,1]$.
\end{itemize}
\end{lemma}
\begin{proof}
The proof of Lemma~\ref{lem:diagonal_sections} is relegated to
Appendix~\ref{app:diagonal_sections}.
\end{proof}

Part $(a)$ is the exact sense in which our bivariate construction contains
Wilson's univariate one: along the anti-diagonal the two-population problem
\emph{is} a one-population Wilson problem with the two samples pooled
harmonically. Part $(b)$ is equally suggestive. The main diagonal is the locus
$p_1=p_2$, on which the difference vanishes; the condition
$\left|\hp\right|\le z/\sqrt{n_{\ast}}$ is exactly the acceptance region of a
score test of $H_0:p_1=p_2$ evaluated at the least favourable common value $1/2$.
The geometry therefore already encodes when a null difference remains tenable, a
point taken up again in Section~\ref{sec:the_EO_CI}. One special case is worth
recording: when $\hp=0$ the condition in $(b)$ reduces to $q\in[0,1]$ for every
$n_1$ and $n_2$, so the boundary ellipse passes exactly through the corners
$(0,0)$ and $(1,1)$ of the unit square, and the main-diagonal chord is the full
diagonal of that square.

\subsection{Convexity, compactness and properness}
\label{subsec:properness}

We first record the structural facts on which the whole of
Sections~\ref{sec:the_EO_CI}--\ref{sec:comparison} rests.

\begin{property}
\label{prty:convex_compact}
The set $\Sz$ is a compact, convex subset of $\mathbb{R}^2$ with non-empty
interior, and it is centrally symmetric about its centre. Consequently
$\Sx=\Sz\cap[0,1]^2$ is compact and convex, hence connected.
\end{property}
\begin{proof}
The proof of Property~\ref{prty:convex_compact} is relegated to
Appendix~\ref{app:convex_compact}.
\end{proof}

Compactness and convexity are exactly what make the extremal programme
\eqref{eq:mainproblemCI} well posed: they guarantee that both optima are attained,
that the optimisers may be located by a support-function calculation, and --- since
a convex set is connected and the objective is continuous --- that the image of
$\Sx$ under $(p_1,p_2)\mapsto p_1-p_2$ is the \emph{whole} closed interval between
the two optima rather than some proper subset of it. That last consequence is the
step on which the profiled-variance representation of Lemma~\ref{lem:profile}
turns, and hence on which the entire coverage analysis of
Section~\ref{sec:error_analysis} depends.

The next result records that the passage from $\Sz$ to $\Sx$ never destroys the
region, and never leaves it intact either.

\begin{property}
\label{prty:nonempty_proper}
The subset $\Sx$ is non-empty with non-empty interior, and it is always a proper
subset of $\Sz$. Precisely, for all admissible $n_1,n_2,z$ and $\hp$:
\begin{itemize}
\item[$(a)$] $\Sx\neq\emptyset$;
\item[$(b)$] $\Sx$ has non-empty interior; indeed it contains an open segment of
the opposite diagonal of the unit square;
\item[$(c)$] $\Sx\subset\Sz$, the inclusion being strict.
\end{itemize}
\end{property}
\begin{proof}
The proof of Property~\ref{prty:nonempty_proper} is relegated to
Appendix~\ref{app:Sx_non-empty_propersubset}.
\end{proof}

All three parts carry consequences elsewhere. Part $(a)$ guarantees that the
programme \eqref{eq:mainproblemCI} is always feasible, so that the interval we
construct is never vacuous, however extreme the data. Part $(b)$ is what supplies
Slater's condition in the first-order analysis of
Appendix~\ref{app:mainCI}; it is not a triviality, since the natural witness for
part $(a)$ --- the observed pair $(\hp_1,\hp_2)$ --- satisfies the defining
inequality with \emph{zero} slack whenever both samples are homogeneous, that is
whenever $\hp_1,\hp_2\in\{0,1\}$, and is then a boundary point rather than an
interior one. Part $(c)$ guarantees that the restriction to the parameter space is
always binding, so that the unit square can never be dispensed with in the
optimisation. The argument in Appendix~\ref{app:Sx_non-empty_propersubset} yields
somewhat more than the statement itself, and it is worth isolating what it
delivers.

\begin{corollary}
\label{cor:Sz_geometry}
In addition to Property~\ref{prty:nonempty_proper}, the following
hold for all admissible $n_1,n_2,z$ and $\hp$:
\begin{itemize}
    \item[$\bullet$] The subset $\Sx$ always intersects the opposite diagonal of
    the unit square, that is,
    \begin{equation*}
        \Sx \cap \left\{ (p_1, p_2) \in [0,1]^2 : p_1 + p_2 = 1 \right\} \neq \emptyset.
    \end{equation*}

    \item[$\bullet$] The set $\Sz$ always intersects the region located either to
    the left of or below the unit square. More precisely,
    \begin{equation*}
        \Sz \cap \Big( (0,1) \times (-\infty,0) \;\cup\; (-\infty,0) \times (0,1) \Big) \neq \emptyset,
    \end{equation*}
    the first alternative occurring when $\hp\ge0$ and the second when $\hp\le0$.

    \item[$\bullet$] The set $\Sz$ also intersects the region located either to
    the right of or above the unit square. Specifically,
    \begin{equation*}
        \Sz \cap \Big( (0,1) \times (1,\infty) \;\cup\; (1,\infty) \times (0,1) \Big) \neq \emptyset,
    \end{equation*}
    the second alternative occurring when $\hp\ge0$ and the first when $\hp\le0$.

    \item[$\bullet$] Consequently, the admissible domain of $\Sz$ is contained
    within
    \begin{equation*}
        \Sz \subseteq [0,1]^2 \;\cup\;
        \big( (0,1)\times(-\infty,0) \big) \;\cup\;
        \big( (-\infty,0)\times(0,1) \big) \;\cup\;
        \big( (0,1)\times(1,\infty) \big) \;\cup\;
        \big( (1,\infty)\times(0,1) \big).
    \end{equation*}
\end{itemize}
\end{corollary}
\begin{proof}
The proof of Corollary~\ref{cor:Sz_geometry} is relegated to
Appendix~\ref{app:Sz_geometry}.
\end{proof}

The last bullet deserves emphasis, since it is the strongest of the four: the
ellipse may escape the unit square across any one of its four edges, but it can
never reach a region in which \emph{both} coordinates are inadmissible. The reason
is transparent in the form \eqref{eq:Omega_variance_form}, where membership forces
$V(p_1,p_2)\ge0$, and $V$ is strictly negative as soon as both arguments leave
$[0,1]$. It should be added that the escape guaranteed by the two middle bullets
may be very small. At $n_1=n_2=45$, $\a=5\%$ and $\hp=0$, for example, the
boundary ellipse touches the corners $(0,0)$ and $(1,1)$ exactly and protrudes
below the bottom edge only for $p_1$ between $0$ and $z^2/(\nmax+z^2)\approx0.079$;
the protrusion is real, as Property~\ref{prty:nonempty_proper}$(c)$
requires, but it is not visible at the resolution of
Figure~\ref{fig:setLF}. The escape grows with $\left|\hp\right|$, and for
$\left|\hp\right|=1$ a substantial part of the region lies outside the square.
The admissible domain described by the fourth bullet is illustrated in
Figure~\ref{fig:Sz_domain}.

\begin{figure}[htbp]
    \centering
    \includegraphics[width=0.60\textwidth]{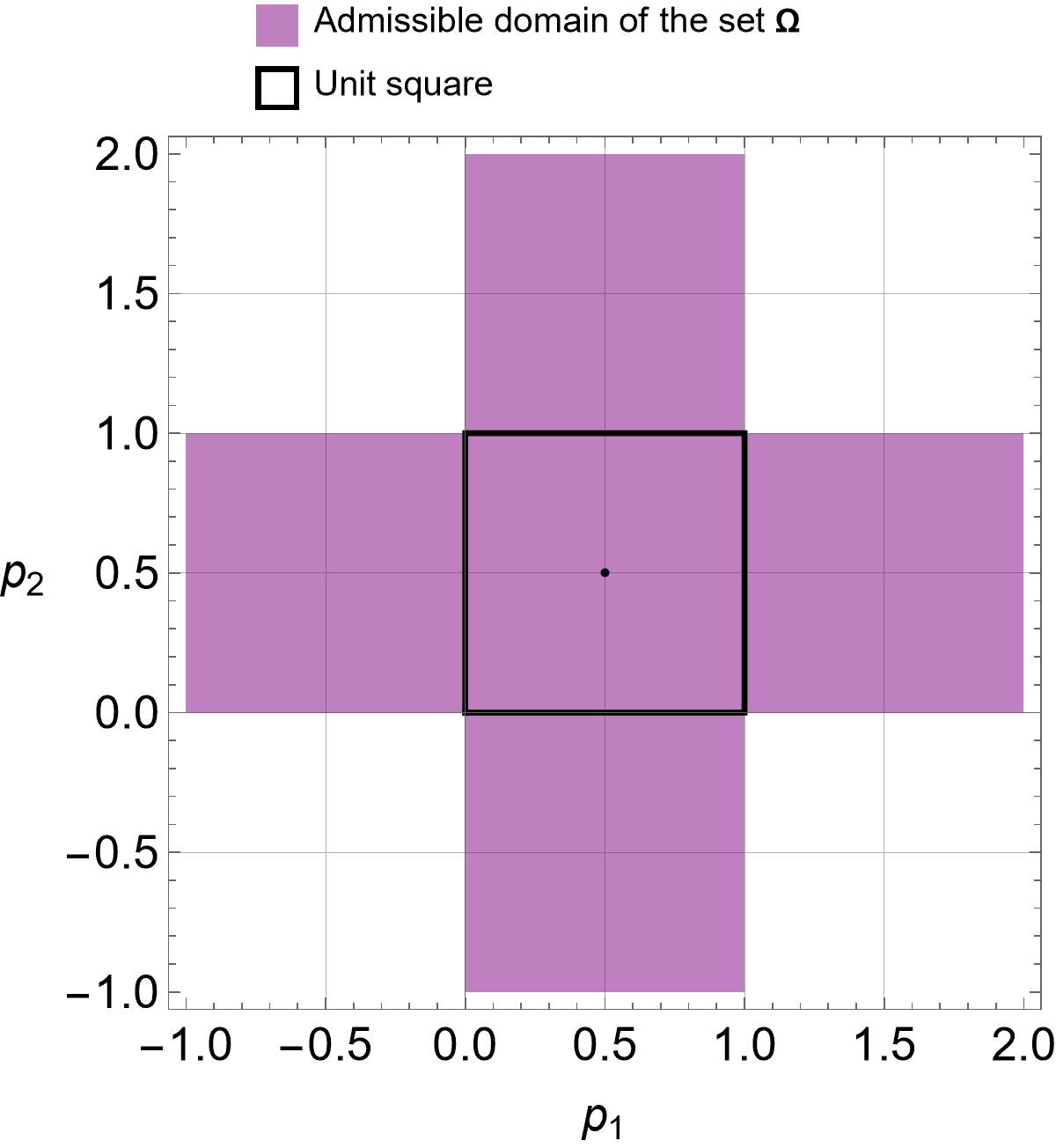} 
    \caption[Admissible domain of the unrestricted ellipse]{Visualisation of the
    admissible domain of $\Sz$, as delineated in Corollary~\ref{cor:Sz_geometry},
    over the region $[-1,2]^2$.}
    \label{fig:Sz_domain}
\end{figure}

Having settled that $\Sz$ is genuinely two-dimensional relative to the square, we
confirm that it never degenerates as a conic.

\begin{property}
\label{prty:determinant}
The determinant $\Delta$ of the matrix of the conic associated with
\eqref{eq:ellipseprob} is given by
\begin{equation}
\label{eq:Delta}
\Delta=-\ddfrac{z^4 \left((n_1+n_2)(n_1+n_2+z^2)-4 n_1n_2 \hp^2\right)}{4 n_1^2 n_2^2},
\end{equation}
and satisfies
\begin{equation}
\label{eq:Delta_bound}
\Delta\le -\ddfrac{z^4 \left((n_1-n_2)^2+(n_1+n_2)z^2\right)}{4 n_1^2 n_2^2}<0 .
\end{equation}
Together with the strictly negative discriminant $B^2-4AC$ computed in
Section~\ref{subsec:bivariate}, this implies that $\Sz$ is always a \emph{real},
non-degenerate ellipse: it never collapses to a single point or to a pair of
lines, and it is never the empty (imaginary) conic.
\end{property}
\begin{proof}
The proof of Property~\ref{prty:determinant} is relegated to
Appendix~\ref{app:determinant}.
\end{proof}

The bound \eqref{eq:Delta_bound} is obtained by setting $\hp^2=1$ and remains
strictly negative even in that most extreme case; degeneracy is therefore not
merely avoided for typical data but excluded uniformly over the entire range of
the observable. The bracket appearing in \eqref{eq:Delta} will reappear as the
quantity $\Xi$ of \eqref{eq:Xi}, where its positivity is what guarantees that the
bounds $\Lthree$ and $\Uthree$ of Definition~\ref{def:CIboundsLU} are real for
every possible sample.

\subsection{Location of the region}
\label{subsec:location}

\begin{property}
\label{prty:centre}
The centre of $\Sz$ is given by
\begin{equation}
\label{eq:centre}
(p_1^c,\ p_2^c)=
\left(\ddfrac{1}{2}+\ddfrac{n_1 \hp}{n_1+n_2+z^2},\ \ddfrac{1}{2}-\ddfrac{n_2 \hp}{n_1+n_2+z^2}\right).
\end{equation}
\end{property}
\begin{proof}
The proof of Property~\ref{prty:centre} is relegated to
Appendix~\ref{app:conic_reductions}.
\end{proof}

Expression \eqref{eq:centre} makes clear how the inputs govern the position of the
region, and it repays closer inspection than its simplicity might suggest. Three
observations stand out.

First, and most notably, the centre depends on the two sample proportions
\emph{only through their difference} $\hp$ --- a manifestation, at the level of
the most informative single summary of the ellipse, of the data reduction recorded
in Remark~\ref{rem:difference_only}. The location of the region is therefore
invariant under any change of the individual proportions that preserves their
difference, a structural feature that will reappear, in a much stronger form, in
the bounds of the interval themselves.

Second, the dependence is a shrinkage. Subtracting the two coordinates of
\eqref{eq:centre} gives the identity
\begin{equation}
\label{eq:centre_difference}
p_1^c-p_2^c=\hp\cdot\frac{n_1+n_2}{n_1+n_2+z^2},
\end{equation}
so that the difference at the centre of the ellipse is the observed difference
pulled toward zero by the factor $(n_1+n_2)/(n_1+n_2+z^2)\in(0,1)$. This is the
exact bivariate counterpart of the familiar univariate phenomenon whereby the
centre $(n\hp+z^2/2)/(n+z^2)$ of Wilson's interval \eqref{eq:WilsonCIOnePop} is
the observed proportion pulled toward $1/2$: in one dimension the score method
shrinks toward the least informative proportion, in two dimensions it shrinks
toward the null of equal proportions. The strength of the shrinkage is governed by
$z^2$ relative to the total sample size and vanishes as $n_1+n_2\to\infty$, as one
would expect of a correction of order $O\bigl(n^{-1}\bigr)$.

Third, the motion of the centre is rectilinear, and the line it traces deserves a
name. As $\hp$ increases from $-1$ to $1$ the point \eqref{eq:centre} traverses a
straight segment through $\left(\tfrac12,\tfrac12\right)$ with direction vector
$(n_1,-n_2)$, so that $p_1^c$ increases while $p_2^c$ decreases along a line of
slope $-n_2/n_1$. Writing that line in implicit form,
\begin{equation}
\label{eq:centreline}
\ell=\left\{(p_1,p_2)\in\mathbb{R}^{2}:\ n_2p_1+n_1p_2=\frac{n_1+n_2}{2}\right\}
=\left\{\left(\tfrac12+n_1t,\ \tfrac12-n_2t\right):t\in\mathbb{R}\right\},
\end{equation}
we obtain the diagonally downward drift observed in Figure~\ref{fig:setLF}, and an
explanation of why that drift is steeper when $n_2$ exceeds $n_1$. The line $\ell$
is invariant under the point reflection \eqref{eq:pointreflection}, and it passes
through the centre of the unit square.

\begin{remark}
\label{rem:centreline}
The line $\ell$ of \eqref{eq:centreline} is the distinguished direction of this
entire paper, and it is worth flagging here that it will be met three further
times. In Section~\ref{sec:the_EO_CI} the two extreme points of $\Sz$ in the direction
$(1,-1)$ are shown to lie on $\ell$, because the support direction
$A_2^{-1}(1,-1)^{\top}$ is proportional to $(n_1,-n_2)$ and the centre lies on
$\ell$ by construction; these are the optimisers of \eqref{eq:mainproblemCI}
precisely when the unit square is inactive, that is in case $(6)$ of
Table~\ref{tab:CIresults}. In Section~\ref{sec:error_analysis} the same line arises
twice more: it is the set of least-favourable parameter pairs attaining the
profiled variance $V^{\ast}$ of Lemma~\ref{lem:profile}, and it is the zero-error
locus \eqref{eq:valleyline} on which the coverage of the \EO{} interval is exactly
nominal. Finally, parametrising $\ell$ as in \eqref{eq:centreline}, membership of
the unit square amounts to $\left|t\right|\le1/(2\nmax)$, so that the difference
$p_1-p_2=(n_1+n_2)t$ ranges over precisely
$\left[-\hp_{\square},\hp_{\square}\right]$ with
$\hp_{\square}=(\nmax+\nmin)/(2\nmax)$ as in \eqref{eq:pbox}. The threshold
governing the six-case split of Table~\ref{tab:CIresults} is thus nothing other
than the range of the difference along the portion of $\ell$ inside the parameter
space.
\end{remark}

The following proposition characterises the region swept out by the centre.

\begin{proposition}
\label{prop:SzCenterDomain}
Let $T$ denote the union of the two open triangles
\begin{equation*}
T=\left\{(x,y)\in\mathbb{R}^2:\ \left(x-\tfrac12\right)\left(y-\tfrac12\right)<0
\ \ \text{and}\ \ \left|x-y\right|<1\right\},
\end{equation*}
that is, the triangle with vertices
$\left(\tfrac12,\tfrac12\right),\left(\tfrac32,\tfrac12\right),\left(\tfrac12,-\tfrac12\right)$
together with its reflection across the main diagonal, having vertices
$\left(\tfrac12,\tfrac12\right),\left(-\tfrac12,\tfrac12\right),\left(\tfrac12,\tfrac32\right)$.
Then the centre of $\Sz$ satisfies
\begin{equation*}
(p_1^c,\ p_2^c)\in T\cup\left\{\left(\tfrac12,\tfrac12\right)\right\},
\end{equation*}
with $(p_1^c,p_2^c)=\left(\tfrac12,\tfrac12\right)$ if and only if $\hp=0$.
The characterisation is sharp: if $n_1$ and $n_2$ are allowed to range over the
positive reals, then \emph{every} point of $T\cup\{(\tfrac12,\tfrac12)\}$ is
attained for a suitable admissible choice of $(n_1,n_2,z,\hp)$. If they are
restricted to the positive integers, the set of attainable centres is a dense
subset of $T\cup\{(\tfrac12,\tfrac12)\}$.
\end{proposition}
\begin{proof}
The proof of Proposition~\ref{prop:SzCenterDomain} is relegated to
Appendix~\ref{app:SzCenterDomain}.
\end{proof}

The two conditions defining $T$ have direct interpretations. The first states that
the coordinates of the centre always straddle $1/2$: the ellipse leans to one side
of the ``equal proportions at one half'' point according to the sign of $\hp$. The
second is exactly the shrinkage identity \eqref{eq:centre_difference}, which
confines the centre to the open strip between the lines $p_1-p_2=\pm1$ on which a
difference of proportions can legitimately live. The admissible region is shaded
in Figure~\ref{fig:centerDomain}.
\begin{figure}[htbp]
\centering
\includegraphics[width=0.60\textwidth]{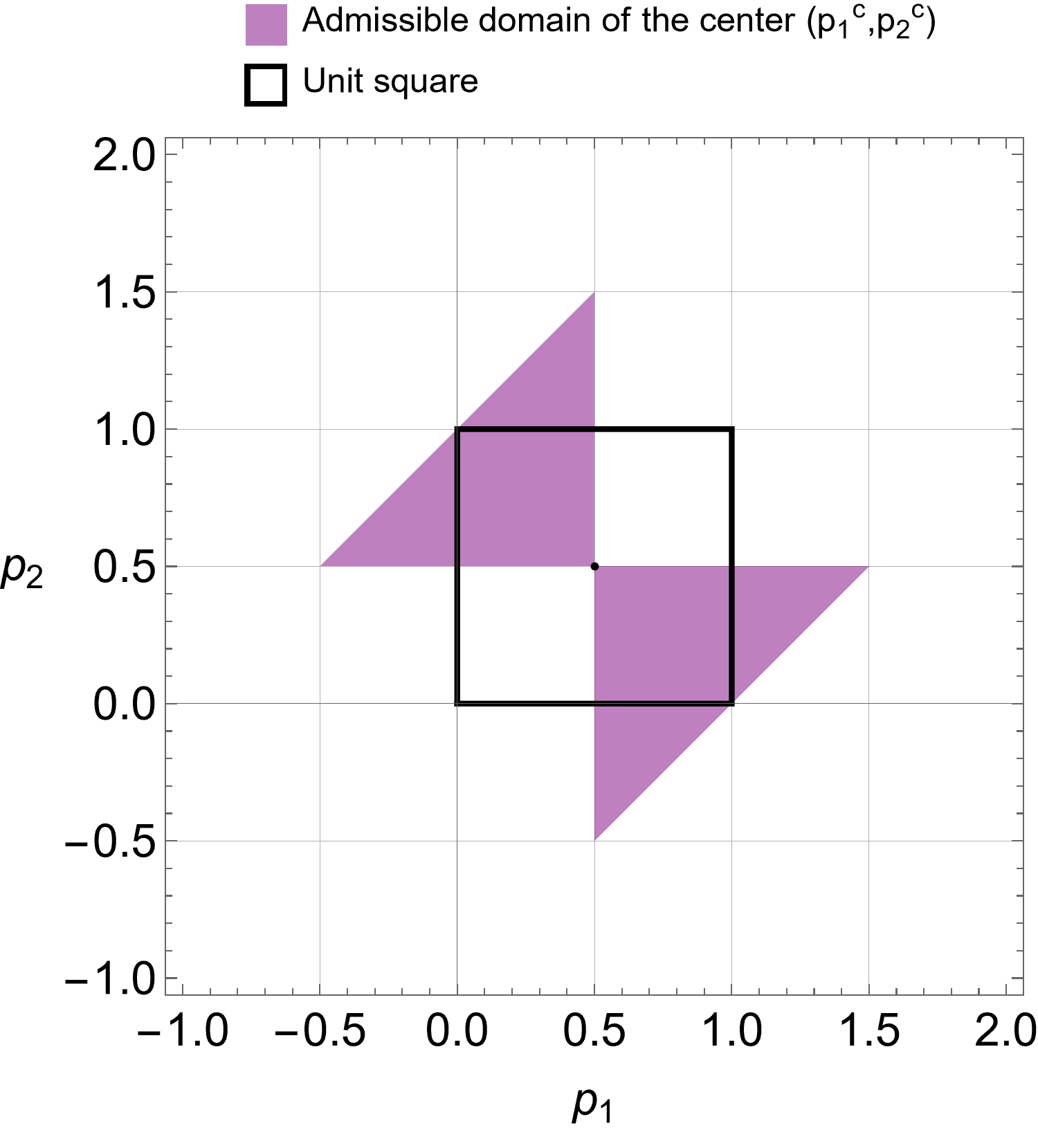}
\caption[Domain of the ellipse centre]{The domain of the centre of the set $\Sz$,
as delineated in Proposition~\ref{prop:SzCenterDomain}: two open triangles meeting
at $\left(\tfrac12,\tfrac12\right)$, shown together with the unit square.}
\label{fig:centerDomain}
\end{figure}

\begin{corollary}
\label{cor:centre_in_square}
The centre of $\Sz$ always belongs to $\Sz$; consequently it belongs to $\Sx$ if
and only if it belongs to the unit square, which occurs precisely when
\begin{equation*}
\left(p_1^c,p_2^c\right)\in
\left(\left(\left[0,\tfrac12\right)\times\left(\tfrac12,1\right]\right)\setminus\left\{(0,1)\right\}\right)
\;\cup\;\left\{\left(\tfrac12,\tfrac12\right)\right\}\;\cup\;
\left(\left(\left(\tfrac12,1\right]\times\left[0,\tfrac12\right)\right)\setminus\left\{(1,0)\right\}\right).
\end{equation*}
In particular the centre can never lie in the closed lower-left or upper-right
quadrant of the unit square other than at its midpoint.
\end{corollary}
\begin{proof}
The proof of Corollary~\ref{cor:centre_in_square} is relegated to
Appendix~\ref{app:SzCenterDomain}.
\end{proof}

\subsection{Size and orientation}
\label{subsec:shape}

Having located the region, we turn to its size and shape.

\begin{property}
\label{prty:axes}
The lengths of the major and minor axes of $\Sz$ are given, respectively, by
\begin{gather*}
\mathrm{Major}=\frac{\sqrt{\left((n_1+n_2)^2-4 n_1n_2 \hp^2+(n_1+n_2)z^2\right)\left((n_1+n_2)z^2 +2 n_1 n_2+\sqrt{4 n_1^2 n_2^2+(n_1-n_2)^2z^4}\right) }}{ \sqrt{2n_1n_2} \left(n_1+n_2+z^2\right)},
\\
\mathrm{Minor}=\frac{\sqrt{\left((n_1+n_2)^2-4 n_1n_2 \hp^2+(n_1+n_2)z^2\right)\left((n_1+n_2)z^2 +2 n_1 n_2-\sqrt{4 n_1^2 n_2^2+(n_1-n_2)^2z^4}\right) }}{ \sqrt{2n_1n_2} \left(n_1+n_2+z^2\right)}.
\end{gather*}
\end{property}
\begin{proof}
The proof of Property~\ref{prty:axes} is relegated to
Appendix~\ref{app:conic_reductions}.
\end{proof}

Although these expressions appear cumbersome, they simplify considerably in the
balanced design $n_2=n_1$:
\begin{gather*}
\mathrm{Major}=\sqrt{\frac{4n_1\left(1- \hp^2\right)+2z^2}{2 n_1+z^2}},
\qquad
\mathrm{Minor}=\frac{z\sqrt{4n_1\left(1- \hp^2\right)+2z^2}}{2 n_1+z^2},
\end{gather*}
and in this form the geometry is immediate. In the balanced design the major axis
cannot exceed $\sqrt{2}$, the bound being attained precisely at $\hp=0$; since
$\sqrt2$ is the length of the diagonal of the unit square, and since the ellipse
passes through $(0,0)$ and $(1,1)$ when $\hp=0$, the extremal configuration is the
one in which the major axis coincides exactly with the main diagonal of the
parameter space, in agreement with Lemma~\ref{lem:diagonal_sections}$(b)$. It
should be stressed that this bound is specific to the balanced case and does not
persist under unequal allocation: at $n_1=60$, $n_2=30$, $\hp=0$ and $\a=5\%$ the
major axis has length $1.4180$, and at $n_1=1000$, $n_2=10$ it has length
$1.4844$, both exceeding $\sqrt2$. This is not paradoxical, since by
Corollary~\ref{cor:Sz_geometry} the ellipse is not confined to the unit square.

The minor axis obeys the same bound in the balanced design but is typically far
smaller: their ratio is
\begin{equation*}
\frac{\mathrm{Minor}}{\mathrm{Major}}=\frac{z}{\sqrt{2n_1+z^2}},
\end{equation*}
which is $O\bigl(n_1^{-1/2}\bigr)$, so that for conventional confidence levels and
realistic sample sizes the region is a thin sliver rather than a rounded oval.
This is the quantitative counterpart of the elongation noted in
Figure~\ref{fig:setLF}, and it is the geometric source of the conservatism of our
interval: a thin sliver leaning across the square is poorly approximated by the
band of constant difference that contains it. The exact price of that
approximation is computed in Section~\ref{sec:error_analysis}.

Finally we determine the direction in which the sliver leans.

\begin{property}
\label{prty:rotation}
The rotation angle $\theta$ of the major axis of $\Sz$, measured
counter-clockwise from the horizontal axis, is given by
\begin{gather*}
\theta=
\begin{cases}
\ \dfrac{1}{2}\,\tan^{-1}\!\left(\dfrac{2n_1n_2}{(n_1-n_2)z^2}\right),
& \text{if } n_1>n_2,
\\[2.2ex]
\ \dfrac{1}{2}\left(\tan^{-1}\!\left(\dfrac{2n_1n_2}{(n_1-n_2)z^2}\right)+\pi\right),
& \text{if } n_2>n_1,
\\[2.2ex]
\ \dfrac{\pi}{4},
& \text{if } n_1=n_2,
\end{cases}
\end{gather*}
where $\tan^{-1}$ denotes the principal branch, with values in
$\left(-\pi/2,\pi/2\right)$.
\end{property}
\begin{proof}
The proof of Property~\ref{prty:rotation} is relegated to
Appendix~\ref{app:conic_reductions}.
\end{proof}

\begin{corollary}
\label{cor:rotation_range}
The rotation angle satisfies
\begin{equation*}
0<\theta<\frac{\pi}{4}\ \ \text{when }n_1>n_2,
\qquad
\theta=\frac{\pi}{4}\ \ \text{when }n_1=n_2,
\qquad
\frac{\pi}{4}<\theta<\frac{\pi}{2}\ \ \text{when }n_2>n_1 .
\end{equation*}
In particular $\theta\in(0,\pi/2)$ for all admissible inputs, so that $\Sz$ is
always rotated counter-clockwise from the horizontal axis by an angle strictly
between $0^{\circ}$ and $90^{\circ}$, confirming the behaviour observed in
Figure~\ref{fig:setLF}.
\end{corollary}
\begin{proof}
The proof of Corollary~\ref{cor:rotation_range} is relegated to
Appendix~\ref{app:conic_reductions}.
\end{proof}

Corollary~\ref{cor:rotation_range} completes the qualitative picture. In the
balanced design the major axis is parallel to the main diagonal, that is, to the
lines along which the difference $p_1-p_2$ is constant: the region is at its most
elongated in exactly the direction to which the parameter of interest is blind,
and at its narrowest transverse to it. Unequal sample sizes tilt the axis toward
the coordinate direction associated with the \emph{larger} sample, reflecting the
greater precision with which that proportion is resolved. It should be said that
the magnitude of the tilt is slight at conventional inputs: at the sample sizes of
Figure~\ref{fig:setLF} the angle $\theta$ takes the values $44.08^{\circ}$,
$45^{\circ}$ and $45.92^{\circ}$, so the three panels are practically
indistinguishable in orientation and differ visibly only in the placement of the
regions. The tilt becomes pronounced only for severely unbalanced designs, in
which $\tan(2\theta)=2n_1n_2/\bigl((n_1-n_2)z^2\bigr)$ approaches zero.

Taken together with the axis lengths of Property~\ref{prty:axes}, this explains
both why a Wilson-type bivariate region is so informative about $p_1-p_2$ and why
reducing it to an interval, as \eqref{eq:mainproblemCI} requires, must cost
something. Quantifying that cost exactly is the task of
Section~\ref{sec:error_analysis}; solving the programme itself is the task of the
section that follows.
\section{The elliptically optimal confidence interval}
\label{sec:the_EO_CI}

Section~\ref{sec:properties_of_omega_lambda} established everything we need about
the feasible set of the programme \eqref{eq:mainproblemCI}. The region
$\Sx$ is non-empty, compact and convex; it is always a proper subset of the
ellipse $\Sz$, so the unit square is always a binding constraint; its centre is
the observed difference shrunk toward zero; and it is a thin sliver whose major
axis leans across the square at an angle strictly between $0$ and $\pi/2$. In this
section we solve the programme.

Minimising and maximising a linear functional over a convex body is a
support-function calculation, and the geometry dictates in advance the form the
answer must take. Either the extreme point of the ellipse in the direction
$(\pm1,\mp1)$ falls inside the unit square, in which case the square is inactive
and the bound is a property of the ellipse alone; or that extreme point is cut off
by an edge, in which case the optimum migrates onto the edge, one proportion is
pinned at a known value, and the bivariate problem degenerates to a univariate
one. The three pairs of bounds introduced below are precisely these two
possibilities and their mirror image, and the content of the main theorem is the
determination of which possibility obtains for which data.

Before any of this, however, there is a reduction which makes the entire section
transparent, and which we therefore place first. It shows that the bivariate
programme \eqref{eq:mainproblemCI} is equivalent to a genuinely one-dimensional
inequality of exactly Wilson's form, in which the unknown nuisance direction has
been eliminated by maximisation.

\subsection{A profiled-variance representation}
\label{subsec:profile}

Recall from \eqref{eq:Vfunction} the variance function
$V(p_1,p_2)=p_1(1-p_1)/n_1+p_2(1-p_2)/n_2$, and from
\eqref{eq:Omega_variance_form} that $(p_1,p_2)\in\Sz$ if and only if
$\left(\hp-p_1+p_2\right)^{2}\le z^{2}V(p_1,p_2)$. Define the
\emph{least-favourable}, or \emph{profiled}, variance of the difference,
\begin{equation}
\label{eq:Vstardef}
V^{\ast}(p)=\max\Bigl\{V(p_1,p_2)\ :\ (p_1,p_2)\in[0,1]^{2},\ p_1-p_2=p\Bigr\},
\qquad p\in[-1,1],
\end{equation}
that is, the largest sampling variance that the observed difference could have
had, over all pairs of admissible proportions consistent with a hypothesised
difference $p$. Set also
\begin{equation}
\label{eq:pbox}
\hp_{\square}=\frac{\nmax+\nmin}{2\nmax}\in\left(\tfrac12,1\right],
\end{equation}
a threshold which will recur throughout the section and which equals unity exactly
in the balanced design $n_1=n_2$.

\begin{lemma}
\label{lem:profile}
The profiled variance \eqref{eq:Vstardef} is given in closed form by
\begin{equation}
\label{eq:Vstar}
V^{\ast}(p)=
\begin{cases}
\ \dfrac{\nmax+\nmin}{4\nmax\nmin}-\dfrac{p^{2}}{\nmax+\nmin},
& \text{if } |p|\le\hp_{\square},
\\[3ex]
\ \dfrac{|p|\left(1-|p|\right)}{\nmin},
& \text{if } |p|>\hp_{\square},
\end{cases}
\end{equation}
the maximum in \eqref{eq:Vstardef} being attained on the first branch at the
interior pair $\left(\tfrac12+\tfrac{n_1p}{n_1+n_2},\ \tfrac12-\tfrac{n_2p}{n_1+n_2}\right)$
and on the second branch at a pair lying on the edge of the unit square associated
with the larger sample size. The function $V^{\ast}$ is even, concave and
continuously differentiable on $[-1,1]$, it is strictly positive on $(-1,1)$, and
it vanishes at $p=\pm1$. Moreover the \EO{} interval defined by
\eqref{eq:mainproblemCI} admits the representation
\begin{equation}
\label{eq:EOprofile}
\Bigl[\min_{(p_1,p_2)\in\Sx}(p_1-p_2),\ \max_{(p_1,p_2)\in\Sx}(p_1-p_2)\Bigr]
=\Bigl\{p\in[-1,1]:\ \left(\hp-p\right)^{2}\le z^{2}V^{\ast}(p)\Bigr\}.
\end{equation}
\end{lemma}
\begin{proof}
The proof of Lemma~\ref{lem:profile} is relegated to Appendix~\ref{app:profile}.
\end{proof}

Representation \eqref{eq:EOprofile} is the conceptual centre of this paper, and it
deserves to be read slowly. Wilson's univariate interval is, by
\eqref{eq:parabolaprob}, the set of parameter values $p$ satisfying
$(\hp-p)^{2}\le z^{2}\Var_p[\hp]$, in which the variance is evaluated at the
hypothesised value rather than at the estimate. In two populations the hypothesised
difference $p$ does not determine the variance, because it does not determine the
pair $(p_1,p_2)$; a nuisance direction survives, and it must be dealt with somehow.
The classical resolutions are to estimate it, as Wald does, or to replace it by a
constrained maximum-likelihood estimate, as in the score intervals of \cite{Miettinen1985}. The programme
\eqref{eq:mainproblemCI} instead eliminates it by \emph{maximisation}: the
resulting interval is precisely Wilson's inequality with the variance replaced by
the largest value compatible with the hypothesis. In other words, the \EO{}
interval is the exact inversion of the least-favourable score test.

Three points follow immediately, and we shall not need to argue them again.
First, the procedure is exactly as conservative as the gap
$V^{\ast}(p)-V(p_1,p_2)$ is wide, and no more: conservatism is not a vague defect
of the construction but a quantity with a formula, which
Section~\ref{sec:error_analysis} will exploit. Second, the qualifier
\emph{optimal} in the name of the interval acquires a precise and limited
meaning: among all intervals whose defining inequality uses a variance surrogate
that is valid for every admissible pair consistent with $p$, the choice
$V^{\ast}$ is the smallest such surrogate, so \eqref{eq:EOprofile} is the shortest
interval retaining the guarantee. It is not optimal in any larger sense, and a
reader who prefers to read \EO{} as \emph{elliptically extended} will lose nothing
by doing so. Third, since $V^{\ast}$ is piecewise quadratic with three pieces, the
inequality in \eqref{eq:EOprofile} is a piecewise quadratic inequality with three
branches --- and the three branches are exactly the three pairs of bounds that the
manuscript's original derivation produced by other means.

\subsection{The three pairs of bounds}
\label{subsec:bounds}

\begin{definition}
\label{def:CIboundsLU}
Let $\Lone,\ \Uone,\ \Ltwo,\ \Utwo,\ \Lthree$ and $\Uthree$ be the scalars
determined by the inputs $n_1,n_2,\hp$ and $z$ through
\begin{gather*}
\Lone=\ddfrac{\nmin \hp+z^2/2}{\nmin+z^2}-\ddfrac{z\sqrt{z^2+4 \nmin\hp \left(1-\hp\right)}}{2\left(\nmin+z^2\right)},
\qquad
\Uone=\ddfrac{\nmin \hp+z^2/2}{\nmin+z^2}+\ddfrac{z\sqrt{z^2+4 \nmin\hp \left(1-\hp\right)}}{2\left(\nmin+z^2\right)},
\\[1.4ex]
\Ltwo=\ddfrac{\nmin \hp-z^2/2}{\nmin+z^2}-\ddfrac{z\sqrt{z^2-4 \nmin\hp \left(1+\hp\right)}}{2\left(\nmin+z^2\right)},
\qquad
\Utwo=\ddfrac{\nmin \hp-z^2/2}{\nmin+z^2}+\ddfrac{z\sqrt{z^2-4 \nmin\hp \left(1+\hp\right)}}{2\left(\nmin+z^2\right)},
\\[1.4ex]
\Lthree=\ddfrac{\left(\nmax+\nmin\right)\hp}{\nmax+\nmin+z^2}
-\ddfrac{z}{2\left(\nmax+\nmin+z^2\right)}\sqrt{\ddfrac{\left(\nmax+\nmin\right)\Xi}{\nmax\nmin}},
\\[1.4ex]
\Uthree=\ddfrac{\left(\nmax+\nmin\right)\hp}{\nmax+\nmin+z^2}
+\ddfrac{z}{2\left(\nmax+\nmin+z^2\right)}\sqrt{\ddfrac{\left(\nmax+\nmin\right)\Xi}{\nmax\nmin}},
\end{gather*}
where
\begin{equation}
\label{eq:Xi}
\Xi=\left(\nmax+\nmin\right)\left(\nmax+\nmin+z^2\right)-4\nmax\nmin\hp^{2}.
\end{equation}
\end{definition}

Two observations on this definition should be made before it is used.

The quantity $\Xi$ is not new. It is exactly the bracket appearing in the
determinant \eqref{eq:Delta} of Property~\ref{prty:determinant}, where the bound
\eqref{eq:Delta_bound} showed that
$\Xi\ge\left(\nmax-\nmin\right)^{2}+\left(\nmax+\nmin\right)z^{2}>0$ uniformly in
the data. Hence $\Lthree$ and $\Uthree$ are always real. The remaining four
quantities are not: $\Lone$ and $\Uone$ are real precisely when
$z^{2}+4\nmin\hp\left(1-\hp\right)\ge0$, which can fail only for $\hp<0$, and
$\Ltwo$ and $\Utwo$ precisely when $z^{2}-4\nmin\hp\left(1+\hp\right)\ge0$, which
can fail only for $\hp>0$. In particular the two conditions are never violated
simultaneously, so at least one of the two outer pairs is always available; we
shall see that this is exactly what the theorem requires. Every statement below
involving these four quantities is to be read as conditional on the relevant
radicand being non-negative, and we verify in each instance that this holds
wherever the quantity is actually used.

The second observation identifies all six quantities with objects that are already
familiar, and explains the otherwise mysterious appearance of the smaller sample
size $\nmin$ in four of them.

\begin{proposition}
\label{prop:WilsonIdentification}
Write
$W(\hpi,\nu)=\left\{q\in\mathbb{R}:\left(\hpi-q\right)^{2}\le z^{2}q(1-q)/\nu\right\}$
for the univariate Wilson region of Section~\ref{subsec:univariate}, and let $-W$
denote its reflection through the origin. Then, whenever the sets concerned are
non-empty,
\begin{equation*}
\left[\Lone,\Uone\right]=W\!\left(\hp,\nmin\right)\subseteq[0,1],
\qquad
\left[\Ltwo,\Utwo\right]=-W\!\left(-\hp,\nmin\right)\subseteq[-1,0],
\end{equation*}
and $\left[\Lthree,\Uthree\right]$ is the range of $p_1-p_2$ over the whole
ellipse $\Sz$, with midpoint equal to the shrunken difference $p_1^c-p_2^c$ of
\eqref{eq:centre_difference}. Equivalently, in the notation of
Lemma~\ref{lem:profile}, $\left[\Lthree,\Uthree\right]$ is the solution set of
$\left(\hp-p\right)^{2}\le z^{2}V^{\ast}(p)$ when $V^{\ast}$ is replaced
throughout by its first branch, and $\left[\Lone,\Uone\right]$ and
$\left[\Ltwo,\Utwo\right]$ are the solution sets when $V^{\ast}$ is replaced
throughout by its second branch on $p>0$ and on $p<0$ respectively. Furthermore
$\left[\Lone,\Uone\right]$ and $\left[\Ltwo,\Utwo\right]$ are the ranges of
$p_1-p_2$ over the sections of $\Sz$ by the edges $\{p_1=1\},\{p_2=0\}$ and by the
edges $\{p_1=0\},\{p_2=1\}$ of the unit square, respectively.
\end{proposition}
\begin{proof}
The proof of Proposition~\ref{prop:WilsonIdentification} is relegated to
Appendix~\ref{app:WilsonIdentification}.
\end{proof}

Proposition~\ref{prop:WilsonIdentification} is the precise sense in which the
\EO{} interval extends Wilson's. On any edge of the unit square one proportion is
pinned at a known value, the bivariate constraint degenerates, and what emerges is
\emph{literally} the score interval \eqref{eq:WilsonCIOnePop} evaluated at the
observed difference $\hp$. The smaller sample size appears because
$W(\hpi,\nu)$ widens as $\nu$ decreases, so that of the two edges available to
the optimiser on each side it is always the one governed by $\nmin$ that binds.
The remaining pair $\Lthree,\Uthree$ is genuinely bivariate: it records the extent
of the ellipse itself in the direction $(1,-1)$, and, as
Appendix~\ref{app:WilsonIdentification} shows, the extremising points lie along
the direction $(n_1,-n_2)$ --- the same direction along which the centre of the
ellipse travels as $\hp$ varies, by the third observation following
\eqref{eq:centre}, and the same direction along which the profiled variance
\eqref{eq:Vstar} attains its maximum. That three independent calculations produce
the same direction is not an accident; it reflects the fact that all three are
governed by the metric induced by $A_2$.

The ordering of the six quantities now becomes almost self-evident.

\begin{proposition}
\label{prop:orderLUCI}
Whenever the radicands in Definition~\ref{def:CIboundsLU} are non-negative, so
that all six quantities are real, they obey
\begin{equation*}
\Lthree\le\Ltwo\le\Utwo\le0\le\Lone\le\Uone\le\Uthree
\end{equation*}
for all $n_1,n_2,z>0$ and $\hp\in[-1,1]$. If only one of the two outer pairs is
real, the corresponding sub-chain continues to hold.
\end{proposition}
\begin{proof}
See Appendix~\ref{app:orderLUCI} for the proof of
Proposition~\ref{prop:orderLUCI}.
\end{proof}

\subsection{The main theorem}
\label{subsec:maintheorem}

By Property~\ref{prty:nonempty_proper} the feasible set is non-empty,
and it is compact, so the programme \eqref{eq:mainproblemCI} attains both optima
and the \EO{} interval is well defined. The representation \eqref{eq:EOprofile}
reduces its determination to a single question: into which branch of $V^{\ast}$
does each endpoint of the unconstrained solution fall? The following lemma answers
it, and is the engine of the theorem.

\begin{lemma}
\label{lem:switching}
Let $\Lthree$ and $\Uthree$ be as in Definition~\ref{def:CIboundsLU} and let
$\hp_{\square}$ be as in \eqref{eq:pbox}. Then the bounds of the \EO{} interval
are
\begin{equation}
\label{eq:switchingrule}
U=
\begin{cases}
\Uone, & \text{if } \Uthree>\hp_{\square},\\
\Uthree, & \text{if } \left|\Uthree\right|\le\hp_{\square},\\
\Utwo, & \text{if } \Uthree<-\hp_{\square},
\end{cases}
\qquad\qquad
L=
\begin{cases}
\Lone, & \text{if } \Lthree>\hp_{\square},\\
\Lthree, & \text{if } \left|\Lthree\right|\le\hp_{\square},\\
\Ltwo, & \text{if } \Lthree<-\hp_{\square},
\end{cases}
\end{equation}
and in each case the bound named on the right is real. The translation of the
three conditions in \eqref{eq:switchingrule} into explicit inequalities in
$(\hp,z)$ is carried out in Appendix~\ref{app:switching}; it involves both
thresholds \eqref{eq:R1} and \eqref{eq:R2}, whose roles are interchanged
according to the sign of $\hp$, and the resulting partition is recorded in
Table~\ref{tab:CIresults}.
\end{lemma}
\begin{proof}
The proof of Lemma~\ref{lem:switching} is relegated to
Appendix~\ref{app:switching}.
\end{proof}

Rule \eqref{eq:switchingrule} already settles a question that the original
formulation could only pose. Of the nine conceivable pairings of a lower with an
upper bound, the pairing $\Lone\le p\le\Utwo$ is excluded by
Proposition~\ref{prop:orderLUCI}, since it would define an empty interval. Of the
remaining eight, two more are excluded by \eqref{eq:switchingrule} on purely
logical grounds: the pairing $\left(\Lone,\Uthree\right)$ would require
$\Uthree\le\hp_{\square}<\Lthree\le\Uthree$, and the pairing
$\left(\Lthree,\Utwo\right)$ would require $\Lthree\le\Uthree<-\hp_{\square}$
together with $\left|\Lthree\right|\le\hp_{\square}$. Both are self-contradictory.
Exactly six pairings survive, and they are the six of the theorem.

\begin{theorem}
\label{thm:mainCI}
The optimisation problem \eqref{eq:mainproblemCI} admits the solution recorded in
Table~\ref{tab:CIresults}, with the thresholds $\Rone$ and $\Rtwo$ defined by
\eqref{eq:R1} and \eqref{eq:R2}. The six condition sets are pairwise disjoint and
exhaust the domain $\left\{(\hp,z):\hp\in[-1,1],\ z>0\right\}$, so that the
interval is uniquely determined by the data; and every bound appearing in the
table is real throughout the region in which it is invoked.
\end{theorem}

\renewcommand{\arraystretch}{2.1}
\begin{table}[htbp]
\centering
\caption[CI bounds for $p$]{Bounds of the \EO{} CI for $p$, with the corresponding
conditions on $\hp$ and $z$. The inequality signs are aligned columnwise, and
$\hp_{\square}=\left(\nmax+\nmin\right)/\left(2\nmax\right)$.}
\label{tab:CIresults}
\begin{tabular}{|c|c@{\;}c@{\;}c|r@{\;}c@{\;}l|r@{\;}c@{\;}l|}
\hline
& \multicolumn{3}{c|}{\rm{CI for $p$}} & \multicolumn{6}{c|}{\rm{Condition}}\\
\cline{2-10}
& \multicolumn{3}{c|}{} & \multicolumn{3}{c|}{$\hp$} & \multicolumn{3}{c|}{$z$}\\
\hline
\textsc{(1)} & $\Lone$ & $\le p\le$ & $\Uone$
& $\hp_{\square}<$ & $\hp$ & $\le 1$
& $0<$ & $z$ & $\le\Rone$\\
\hline
\textsc{(2)} & $\Ltwo$ & $\le p\le$ & $\Uone$
& $-1\le$ & $\hp$ & $\le 1$
& & $z$ & $\ge\Rtwo$\\
\hline
\textsc{(3)} & $\Ltwo$ & $\le p\le$ & $\Utwo$
& $-1\le$ & $\hp$ & $<-\hp_{\square}$
& $0<$ & $z$ & $\le\Rone$\\
\hline
\multirow{2}{*}{\textsc{(4)}} & \multirow{2}{*}{$\Ltwo$} & \multirow{2}{*}{$\le p\le$} & \multirow{2}{*}{$\Uthree$}
& $-1\le$ & $\hp$ & $\le-\hp_{\square}$
& $\Rone<$ & $z$ & $<\Rtwo$\\
\cline{5-10}
& & & & $-\hp_{\square}<$ & $\hp$ & $<0$
& $\Rone\le$ & $z$ & $<\Rtwo$\\
\hline
\multirow{2}{*}{\textsc{(5)}} & \multirow{2}{*}{$\Lthree$} & \multirow{2}{*}{$\le p\le$} & \multirow{2}{*}{$\Uone$}
& $0<$ & $\hp$ & $<\hp_{\square}$
& $\Rone\le$ & $z$ & $<\Rtwo$\\
\cline{5-10}
& & & & $\hp_{\square}\le$ & $\hp$ & $\le 1$
& $\Rone<$ & $z$ & $<\Rtwo$\\
\hline
\textsc{(6)} & $\Lthree$ & $\le p\le$ & $\Uthree$
& $-\hp_{\square}<$ & $\hp$ & $<\hp_{\square}$
& $0<$ & $z$ & $<\Rone$\\
\hline
\end{tabular}
\end{table}
\renewcommand{\arraystretch}{1}

\begin{proof}
The proof of Theorem~\ref{thm:mainCI} is presented in Appendix~\ref{app:mainCI}.
\end{proof}

The thresholds appearing in Table~\ref{tab:CIresults} are
\begin{equation}
\label{eq:R1}
\Rone=\frac{\sqrt{\nmin}\ \Bigl|\,2\nmax\left|\hp\right|-\left(\nmax+\nmin\right)\Bigr|}
            {\sqrt{\nmax^{2}-\nmin^{2}}},
\end{equation}
and
\begin{equation}
\label{eq:R2}
\Rtwo=\frac{\sqrt{\nmin}\ \Bigl(2\nmax\left|\hp\right|+\left(\nmax+\nmin\right)\Bigr)}
            {\sqrt{\nmax^{2}-\nmin^{2}}},
\end{equation}
subject to the conventions
\begin{equation}
\label{eq:Rconventions}
\frac{c}{0}=+\infty \ \ \text{for } c>0,
\qquad\text{and}\qquad
\frac{0}{0}=0 ,
\end{equation}
which are needed only in the balanced design $\nmax=\nmin$. Written out, these
conventions give $\Rtwo=\infty$ whenever $\nmax=\nmin$, and
\begin{equation*}
\Rone=
\begin{cases}
\ \infty, & \text{if } \nmax=\nmin \ \text{ and } \ \left|\hp\right|<1,\\[0.6ex]
\ 0, & \text{if } \nmax=\nmin \ \text{ and } \ \left|\hp\right|=1 .
\end{cases}
\end{equation*}
The second of these is not a cosmetic stipulation: it is forced. When
$\nmax=\nmin$ and $\hp=1$ one has $\hp_{\square}=1$, so cases $(1)$ and $(6)$ of
Table~\ref{tab:CIresults} are vacuous for that value of $\hp$ irrespective of
$\Rone$, and only the convention $\Rone=0$ places the configuration in case $(5)$,
which Appendix~\ref{app:mainCI} shows to be the correct assignment. That the
assignment matters can be seen directly: at $n_1=n_2=20$ and $\a=5\%$ one has
$\Lone=0.8389$ but $\Lthree=0.8248$, and it is $\Lthree$ that solves the
programme. Equivalently, \eqref{eq:Rconventions} is the unique convention under
which the \EO{} bounds are continuous functions of $(\hp,z)$, as recorded in
Corollary~\ref{cor:EOadmissible} below.

Both thresholds are even in $\hp$, as they must be. The reflection $\hp\mapsto-\hp$
interchanges $\Lone\leftrightarrow-\Utwo$, $\Uone\leftrightarrow-\Ltwo$ and
$\Lthree\leftrightarrow-\Uthree$, and correspondingly interchanges the cases
$(1)\leftrightarrow(3)$ and $(4)\leftrightarrow(5)$ while fixing $(2)$ and $(6)$;
the \EO{} interval therefore satisfies the equivariance
$\mathrm{CI}(-\hp)=-\,\mathrm{CI}(\hp)$, which is exactly the statement that
relabelling the two populations reverses the sign of the interval, as any sensible
procedure must. One further consequence of \eqref{eq:R1} and \eqref{eq:R2} closes
the apparent gap between cases $(4)$, $(5)$ and $(6)$ at $\hp=0$: since
$\left|2\nmax\left|\hp\right|-\left(\nmax+\nmin\right)\right|\le2\nmax\left|\hp\right|+\left(\nmax+\nmin\right)$,
one has $\Rone\le\Rtwo$ always, with equality if and only if $\hp=0$, so the band
$\Rone\le z<\Rtwo$ on which cases $(4)$ and $(5)$ live is empty precisely when the
observed difference vanishes.

\subsection{Basic properties of the interval}
\label{subsec:EOproperties}

\begin{corollary}
\label{cor:EOadmissible}
For all admissible inputs the \EO{} confidence interval $[L,U]$ satisfies the
following.
\begin{itemize}
\item[$(a)$] It always contains the point estimate: $L\le\hp\le U$.
\item[$(b)$] It is contained in the parameter space of the estimand,
$[L,U]\subseteq[-1,1]$, and it reaches an endpoint of that space only in the
degenerate case $\hp=\pm1$, in which $U=1$ or $L=-1$ respectively.
\item[$(c)$] Both endpoints are attained by admissible pairs
$(p_1,p_2)\in\Sx\subseteq[0,1]^{2}$, so every value in the interval is the
difference of a pair of genuine proportions compatible with the data.
\item[$(d)$] $L$ and $U$ are continuous functions of $(\hp,z)$ on
$[-1,1]\times(0,\infty)$, non-increasing and non-decreasing in $z$ respectively,
and satisfy $[L,U]\to\{\hp\}$ as $z\downarrow0$ and $[L,U]\to[-1,1]$ as
$z\to\infty$.
\end{itemize}
\end{corollary}
\begin{proof}
The proof of Corollary~\ref{cor:EOadmissible} is relegated to
Appendix~\ref{app:EOadmissible}.
\end{proof}

Part $(b)$ is worth isolating, because it is a genuine admissibility guarantee of
the kind that Wald's interval does not possess, and it is established here without
any appeal to simulation or to comparisons of width. Part $(d)$ shows in
particular that the strict and weak inequalities distinguishing adjacent rows of
Table~\ref{tab:CIresults} are immaterial: on the common boundary of two adjacent
regions the two candidate bounds coincide, and the table could be re-punctuated
without changing a single computed interval.

Figure~\ref{fig:setLT} illustrates several instances of the induced set $\Sr$,
computed directly from Table~\ref{tab:CIresults}, together with the corresponding
subsets $\Sx$.
\begin{figure}[htbp]
\centering
    \begin{subfigure}[t]{0.48\textwidth}
        \centering
        \includegraphics[width=\linewidth]{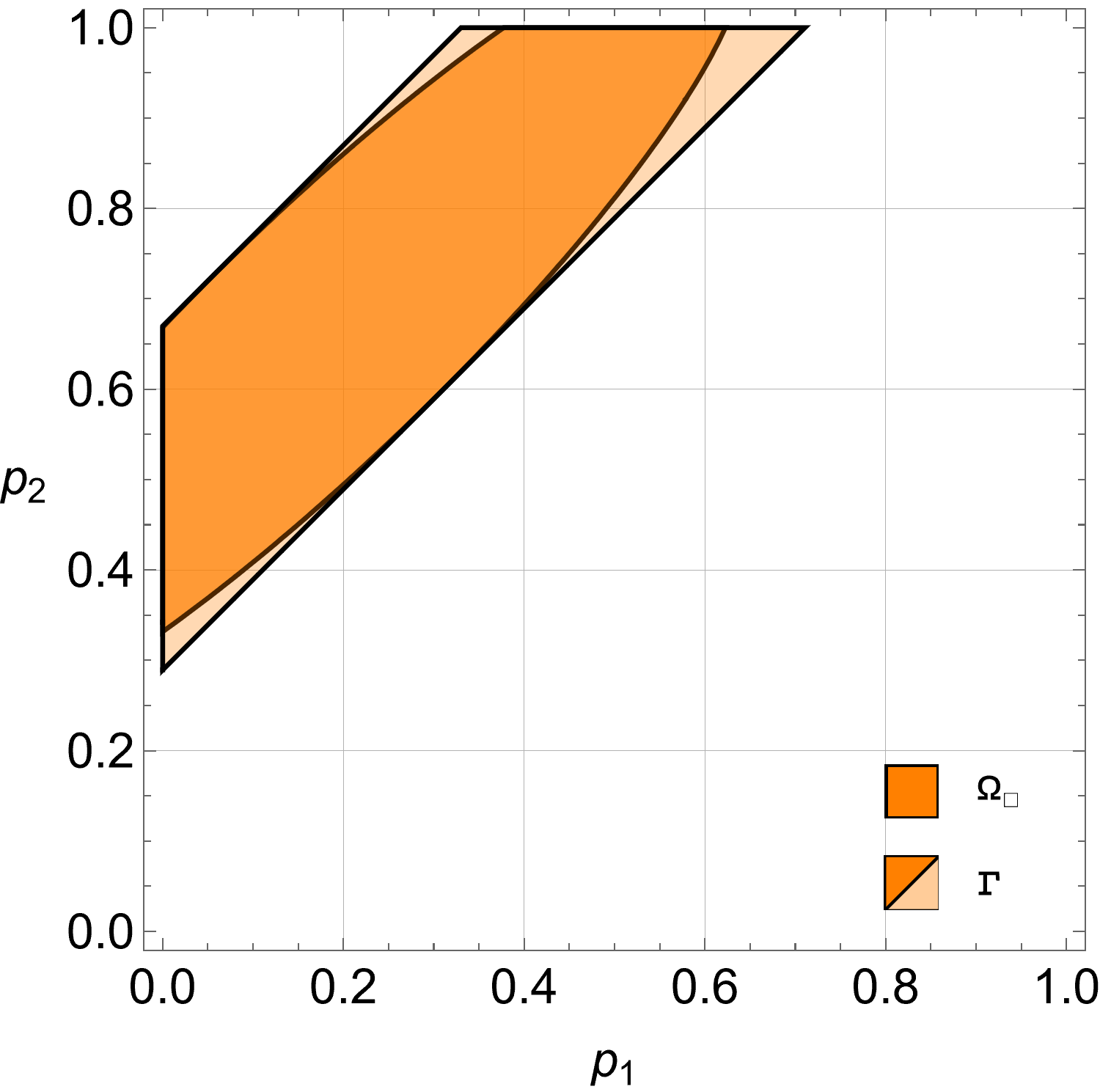}
        \caption{$\hp=-0.5$}
        \label{subfig:setLT-wn}
    \end{subfigure}
    \hfill
    \begin{subfigure}[t]{0.48\textwidth}
        \centering
        \includegraphics[width=\linewidth]{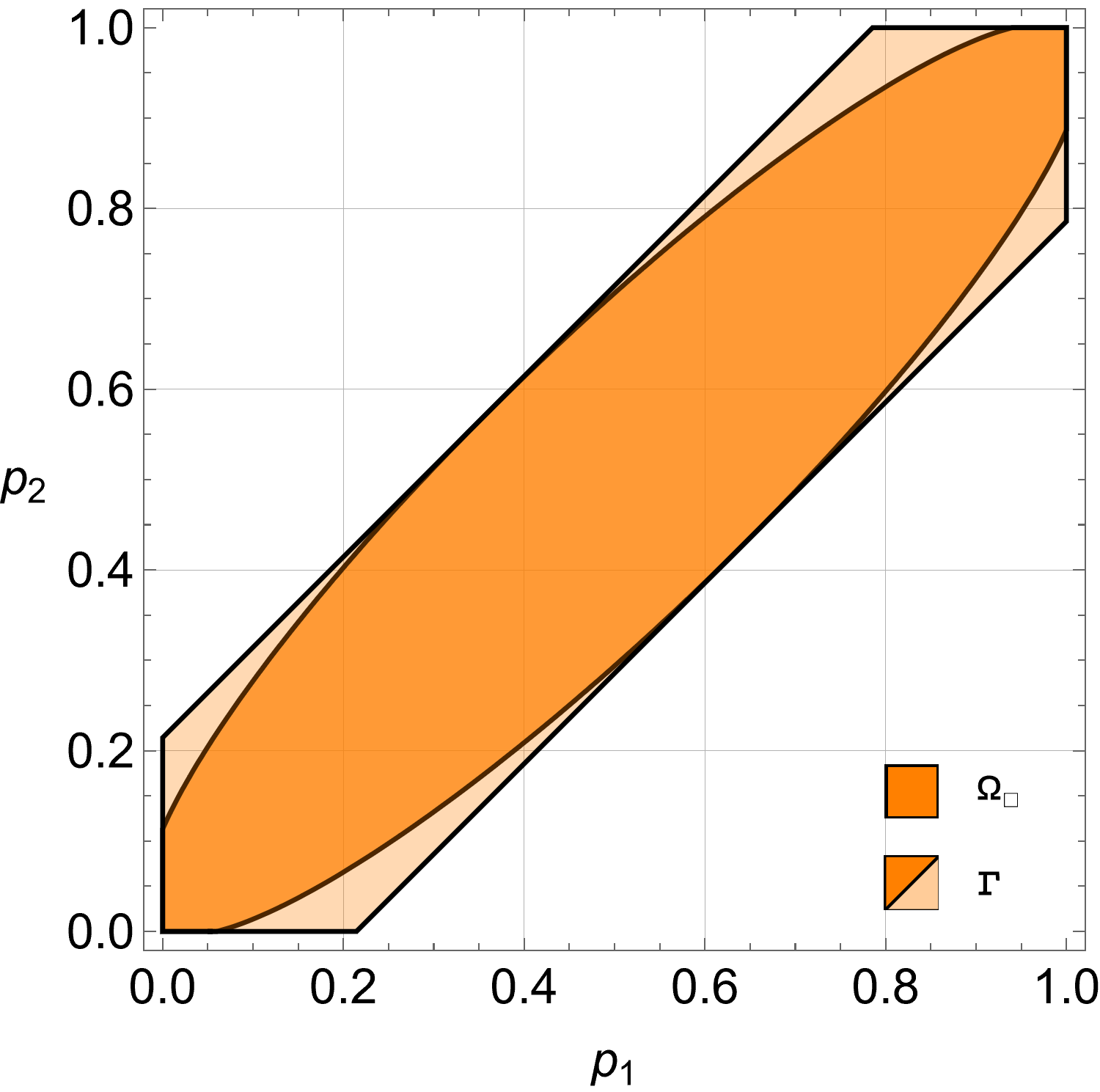}
        \caption{$\hp=0$}
        \label{subfig:setLT-wo}
    \end{subfigure}

    \vspace{0.8\baselineskip}

    \begin{subfigure}[t]{0.48\textwidth}
        \centering
        \includegraphics[width=\linewidth]{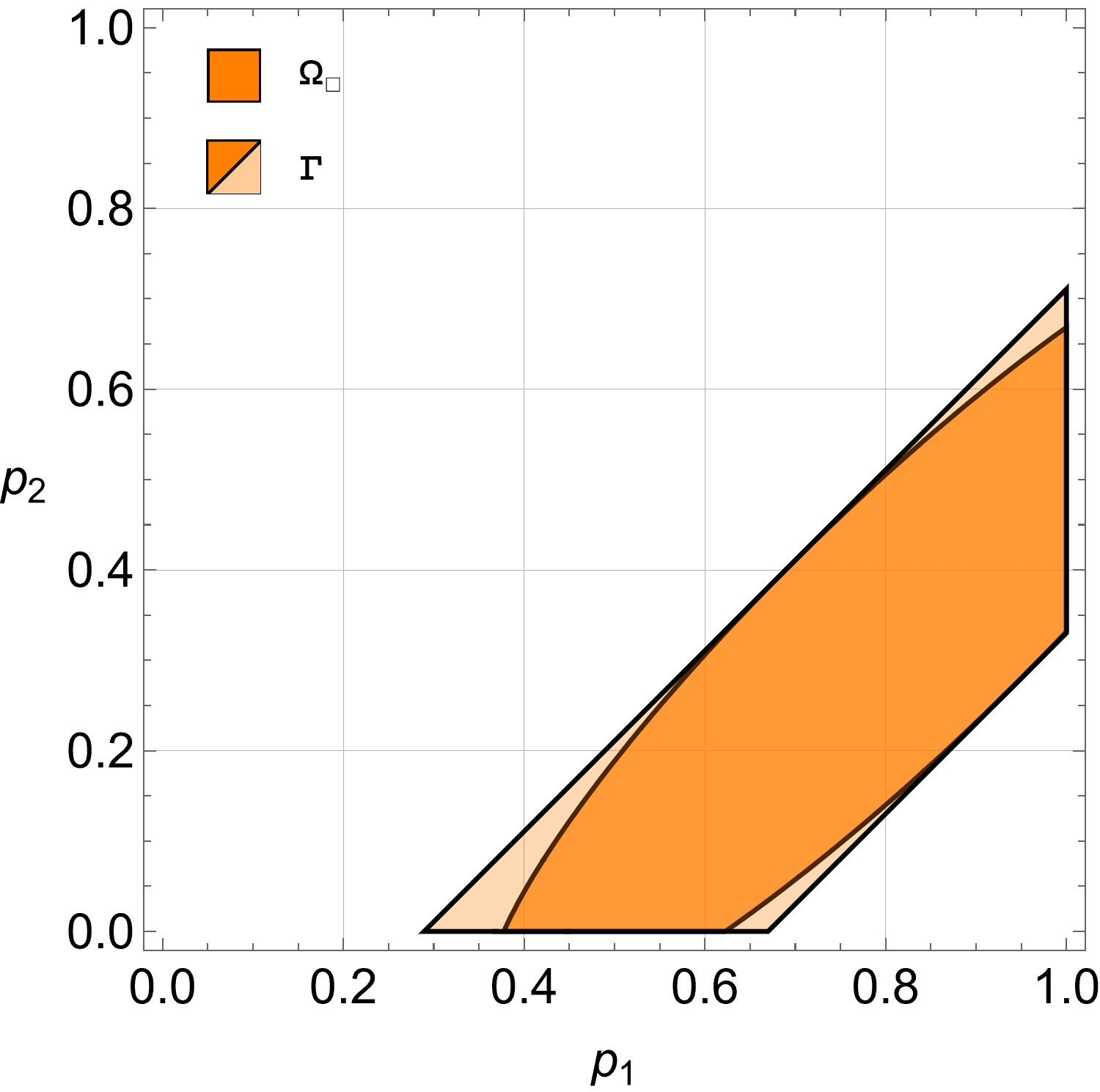}
        \caption{$\hp=0.5$}
        \label{subfig:setLT-wp}
    \end{subfigure}
\caption[The set $\Sx$ enclosed within the optimal values of $p$]{The set $\Sr$
induced by the \EO{} interval, shown together with the subset $\Sx$ that it
contains, for three values of the observed difference. The sample sizes and
significance level are fixed at $n_1=60$, $n_2=30$ and $\a=5\%$; by
Remark~\ref{rem:difference_only} no further inputs are needed.}
\label{fig:setLT}
\end{figure}

The geometry of Figure~\ref{fig:setLT} repays attention, because it displays both
the construction and its cost at a glance. The set $\Sr$ is the intersection of
the unit square with the band lying between the two lines $p_1-p_2=L$ and
$p_1-p_2=U$, and it is therefore a hexagon whose two slanted edges have unit
slope. All three panels fall under case $(6)$ of Table~\ref{tab:CIresults}, since
$\hp_{\square}=0.75$ while $z=1.96$, whereas $\Rone=9.487$ at $\hp=0$ and
$\Rone=3.162$ at $\hp=\pm0.5$; evaluating the
bounds gives $\left[-0.215,\,0.215\right]$ at $\hp=0$ and
$\left[0.289,\,0.670\right]$ at $\hp=0.5$, in agreement with the intercepts
visible in the plots. The two slanted edges of the band are tangent to the ellipse
at the points where the optima are attained, which is the geometric content of
\eqref{eq:mainproblemCI}, and by
Proposition~\ref{prop:WilsonIdentification} the chord joining those two points
runs in the direction $(n_1,-n_2)$. What separates $\Sr$ from $\Sx$ is the pair of
lens-shaped slivers between the ellipse and the band: parameter pairs whose
difference lies inside the interval although the pair itself lies outside the
elliptical region. These slivers are the geometric image of the gap
$V^{\ast}(p)-V(p_1,p_2)$ of Lemma~\ref{lem:profile}, and their area is a visual
proxy for the excess coverage that Section~\ref{sec:error_analysis} will quantify
exactly. They shrink as $\left|\hp\right|$ grows, because the band is then clipped
more aggressively by the corners of the square.

\begin{remark}
\label{rm:CIresultsremmarks}
We record some observations concerning Theorem~\ref{thm:mainCI} and
Table~\ref{tab:CIresults}.
\begin{itemize}
\item[$\bullet$] The \EO{} CI is conservative: the induced set $\Sr$ contains
$\Sx$, so the interval covers at least the nominal proportion of the parameter
pairs compatible with the data, up to the error of the normal approximation. By
Lemma~\ref{lem:profile} the source of the excess is identified exactly, namely the
replacement of $V$ by its least-favourable value $V^{\ast}$ along each line of
constant difference; its magnitude is derived in
Section~\ref{sec:error_analysis}.

\item[$\bullet$] The results employ three types of bound. The scalars
$\Lone,\Uone,\Ltwo$ and $\Utwo$ depend on the sample sizes only through $\nmin$,
whereas $\Lthree$ and $\Uthree$ depend on both;
Proposition~\ref{prop:WilsonIdentification} explains why, by identifying the first
four as univariate Wilson bounds arising on the edges of the unit square and the
last two as the extent of the ellipse itself. Nine pairings of a lower with an
upper bound are conceivable; Proposition~\ref{prop:orderLUCI} eliminates
$\Lone\le p\le\Utwo$, which would define an empty interval, and the switching rule
\eqref{eq:switchingrule} eliminates $\Lone\le p\le\Uthree$ and
$\Lthree\le p\le\Utwo$, leaving exactly the six of Table~\ref{tab:CIresults}.

\item[$\bullet$] Apart from the sample sizes $n_1$ and $n_2$, the \EO{} CI uses
the sample proportions only through their difference $\hp$. This is the
interval-level manifestation of Remark~\ref{rem:difference_only}, and it is best
understood as an invariance property rather than as a virtue of parsimony: the
procedure is invariant under any change of the individual proportions that
preserves their difference, so two studies reporting very different levels but the
same difference, at the same sample sizes, yield identical intervals. It should be
said plainly that this is not costless. The pair $(\hp_1,\hp_2)$ is minimal
sufficient, and the individual levels do carry information about which pairs
$(p_1,p_2)$ are plausible, hence about the variance of $\hp$; declining to use
that information is precisely what the maximisation in \eqref{eq:Vstardef} does,
and it is the reason the interval is conservative rather than exact.

\item[$\bullet$] When $\nmax=\nmin$, that is $n_1=n_2$, one has
$\hp_{\square}=1$ and the only admissible cases are $(4)$, $(5)$ and $(6)$: case
$(4)$ applies when $\hp=-1$ with $z>0$, case $(5)$ when $\hp=1$ with $z>0$, and
case $(6)$ when $-1<\hp<1$ with $z>0$.

\item[$\bullet$] When $p$ is optimised over $\Sz$ rather than over $\Sx$ the only
optima are $\Lthree$ and $\Uthree$, the global minimum and maximum. Accordingly
the \EO{} lower bound coincides with the global minimum in cases $(5)$ and $(6)$,
the upper bound with the global maximum in cases $(4)$ and $(6)$, and in case
$(6)$ both coincide. It follows that when $\nmax=\nmin$ the \EO{} bounds always
coincide with the global optima; in the interior case $(6)$ this is immediate,
while in the two boundary cases the coincidence is exact but degenerate, since
$\Ltwo=\Lthree=-1$ when $\hp=-1$ and $\Uone=\Uthree=1$ when $\hp=1$. In the latter
configuration the optimiser sits exactly at the corner $(1,0)$ of the unit square.

\item[$\bullet$] Selecting the appropriate bounds amounts to identifying which
condition is satisfied. The $(\hp,z)$ domain is partitioned into six disjoint and
exhaustive regions, ensuring a unique choice;
Figure~\ref{fig:CIsregionshpz} depicts this partition.
\end{itemize}
\end{remark}

\begin{figure}[htbp]
\centering
\includegraphics[width=0.70\textwidth]{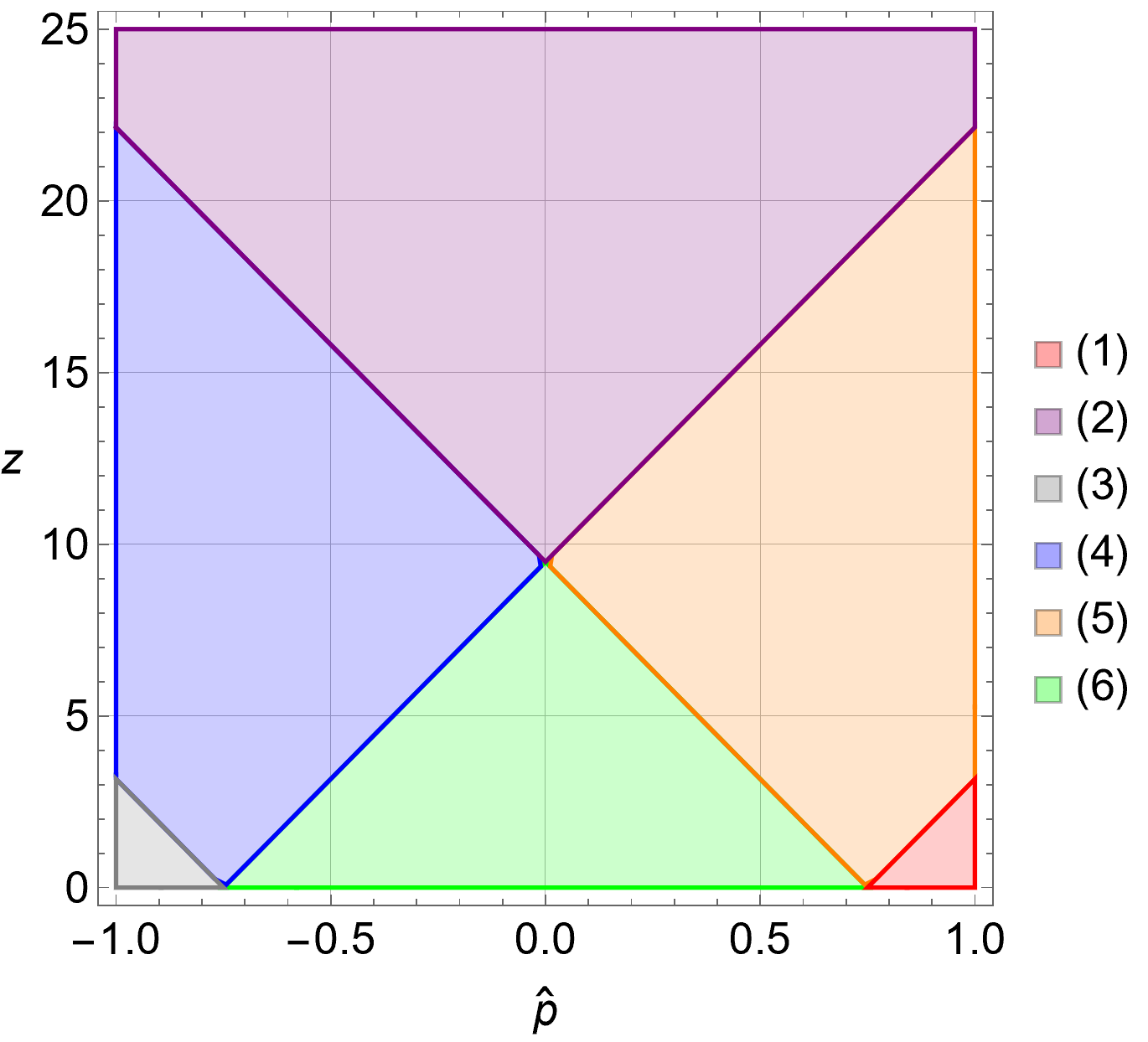}
\caption[$(\widehat{p},z)$ domain partition]{Partition of the $(\hp,z)$ domain
into the six cases of Table~\ref{tab:CIresults}, for $n_1=60$ and $n_2=30$.}
\label{fig:CIsregionshpz}
\end{figure}

Every feature of Figure~\ref{fig:CIsregionshpz} can be read off \eqref{eq:R1} and
\eqref{eq:R2}. The curve $z=\Rone(\hp)$ is piecewise linear in $\left|\hp\right|$
and vanishes exactly at $\hp=\pm\hp_{\square}=\pm0.75$, which is where the
interior region $(6)$ meets the horizontal axis and where the boundary regions
$(1)$ and $(3)$ begin. At $\hp=0$ the two thresholds coincide at
$\Rone=\Rtwo=\sqrt{\nmin\left(\nmax+\nmin\right)/\left(\nmax-\nmin\right)}=\sqrt{90}\approx9.49$,
the apex of the wedge separating region $(6)$ below from region $(2)$ above. At
$\left|\hp\right|=1$ one has
$\Rone=\sqrt{\nmin\left(\nmax-\nmin\right)/\left(\nmax+\nmin\right)}=\sqrt{10}\approx3.16$,
the height of the two small triangles $(1)$ and $(3)$, and $\Rtwo\approx22.14$,
the height at which regions $(4)$ and $(5)$ give way to region $(2)$. The
qualitative reading is that the interior solution $(6)$ governs the practically
relevant regime of moderate $z$ and moderate $\left|\hp\right|$, in which the
square never binds and the interval is a pure property of the ellipse, while the
edges of the square bind only when the confidence level is extreme or the observed
difference approaches its limits --- which is exactly the regime in which a naive
interval would be expected to misbehave.

Two questions remain. The first is how much the conservatism identified in
Lemma~\ref{lem:profile} actually costs in attained coverage; the second is how the
\EO{} interval compares with the established alternatives. These questions are not
independent, and the order in which they are taken matters. A comparison of
interval \emph{widths} carries no evidential weight by itself, since a narrower
interval may be either more accurate or merely overconfident, and only an analysis
of true coverage can distinguish the two; the same caution applies in the opposite
direction, since a conservative interval that is systematically too wide is not
thereby vindicated. We therefore turn first, in Section~\ref{sec:error_analysis},
to an exact characterisation of the discrepancy between nominal and attained
coverage, for which \eqref{eq:EOprofile} supplies the natural starting point, and
defer the comparison with the Wald and Newcombe intervals to
Section~\ref{sec:comparison}, where that characterisation can be brought to bear
on it.

\section{Coverage error of the elliptically optimal interval}
\label{sec:error_analysis}

The \EO{} interval was constructed in Section~\ref{sec:the_EO_CI} as a
conservative surrogate for the elliptical region $\Sx$, and
Lemma~\ref{lem:profile} identified the source of that conservatism exactly: the
unknown nuisance direction is eliminated by replacing the true variance
$V(p_1,p_2)$ with its least-favourable value $V^{\ast}(p)$ along the line of
constant difference. This section converts that identification into a formula. We
shall find that the excess coverage of the interval admits a closed-form
expression involving nothing beyond the standard normal distribution function, and
that this expression explains every qualitative feature of the error surfaces
plotted below --- where the error vanishes, where it is largest, why it is
symmetric, and, importantly, why it does not disappear as the samples grow.

\subsection{The error and its exact form}
\label{subsec:errordefinition}

Throughout this section $(p_1,p_2)$ denotes the true parameter pair, which is
fixed, and $\hp=\hp_1-\hp_2$ is the random quantity; the sets $\Sz$, $\Sx$ and
$\Sr$ are random, since they are determined by $\hp$. We write $p=p_1-p_2$ for the
true difference and
\begin{equation}
\label{eq:Vtrue}
V=V(p_1,p_2)=\ddfrac{p_1\left(1-p_1\right)}{n_1}+\ddfrac{p_2\left(1-p_2\right)}{n_2}
\end{equation}
for its variance, as in \eqref{eq:Vfunction}. Because $\Sx\subseteq\Sr$ by
construction, we have
$\mathbb{P}\left(\left(p_1,p_2\right)\in\Sr\right)\ge\mathbb{P}\left(\left(p_1,p_2\right)\in\Sx\right)$,
and the excess is what we wish to quantify:
\begin{equation}
\label{eq:errordef}
\mathrm{Error}\left(p_1,p_2\right)
=\mathbb{P}\left(\left(p_1,p_2\right)\in\Sr\right)-\mathbb{P}\left(\left(p_1,p_2\right)\in\Sx\right)\ \ge0 .
\end{equation}

Two remarks fix the meaning of \eqref{eq:errordef} precisely. First, the event
$\left(p_1,p_2\right)\in\Sr$ is, by the definition of $\Sr$, exactly the event
that the \EO{} interval covers $p$; the quantity
$\mathbb{P}\left(\left(p_1,p_2\right)\in\Sr\right)$ is therefore the attained
coverage probability of the interval. Second, and less obviously, the event
$\left(p_1,p_2\right)\in\Sx$ is exactly the event
$\left|\hp-p\right|\le z\sqrt{V}$, since membership of $\Sz$ is the inequality
\eqref{eq:Omega_variance_form} and the restriction to the unit square is automatic
for a genuine parameter pair. Under the normal approximation
$\hp\sim N\left(p,V\right)$ this event has probability $2\Phi(z)-1=1-\a$ exactly,
where $\Phi$ denotes the standard normal distribution function. The second term in
\eqref{eq:errordef} is thus the nominal level, and \eqref{eq:errordef} measures
purely the cost of the reduction from the plane to the line --- not the cost of the
normal approximation itself, which is a separate matter to which we return in
Remark~\ref{rm:approximationerror}.

Evaluating \eqref{eq:errordef} therefore reduces to evaluating the attained
coverage. The following lemma performs the inversion; it is immediate from
Lemma~\ref{lem:profile}, and it should be contrasted with the case-by-case
inversion of Table~\ref{tab:CIresults} that a direct approach would require.

\begin{lemma}
\label{lem:ErrorInversion}
Fix $n_1,n_2,z>0$ and $p\in[-1,1]$. The \EO{} interval computed from an observed
difference $\hp$ contains $p$ if and only if
\begin{equation}
\label{eq:acceptance}
\hp\in\Bigl[\,p-z\sqrt{V^{\ast}(p)},\ p+z\sqrt{V^{\ast}(p)}\,\Bigr],
\end{equation}
with $V^{\ast}$ the profiled variance \eqref{eq:Vstar}. In particular the
acceptance region is a single interval, symmetric about $p$, and its endpoints are
the quantities of Definition~\ref{def:ErrorBoundsLU} below: they equal
$\Lstwo$ and $\Ustwo$ when $\left|p\right|\le\hp_{\square}$ and $\Lsone$ and
$\Usone$ when $\left|p\right|>\hp_{\square}$.
\end{lemma}
\begin{proof}
The proof of Lemma~\ref{lem:ErrorInversion} is relegated to
Appendix~\ref{app:ErrorInversion}.
\end{proof}

The force of Lemma~\ref{lem:ErrorInversion} is that $V^{\ast}(p)$ depends on the
hypothesised difference and the sample sizes but not on the data. All six cases of
Table~\ref{tab:CIresults} invert simultaneously and collapse to the single
symmetric interval \eqref{eq:acceptance}; there is no need to invert the case
conditions separately from the bounds, and consequently no overlapping regions to
reconcile.

\begin{definition}
\label{def:ErrorBoundsLU}
For $n_1,n_2,z>0$ and $p\in[-1,1]$, define
\begin{gather*}
\Lsone=p-z\sqrt{\ddfrac{\left|p\right|\left(1-\left|p\right|\right)}{\nmin}},
\qquad\qquad
\Usone=p+z\sqrt{\ddfrac{\left|p\right|\left(1-\left|p\right|\right)}{\nmin}},
\\[1.4ex]
\Lstwo=p-\ddfrac{z}{2}\sqrt{\ddfrac{\left(\nmax+\nmin\right)^{2}-4\nmax\nmin p^{2}}{\nmax\nmin\left(\nmax+\nmin\right)}},
\qquad
\Ustwo=p+\ddfrac{z}{2}\sqrt{\ddfrac{\left(\nmax+\nmin\right)^{2}-4\nmax\nmin p^{2}}{\nmax\nmin\left(\nmax+\nmin\right)}} .
\end{gather*}
\end{definition}

These are precisely $p\mp z\sqrt{V^{\ast}(p)}$ evaluated on the two branches of
\eqref{eq:Vstar}: the first pair uses the edge branch
$V^{\ast}(p)=\left|p\right|\left(1-\left|p\right|\right)/\nmin$, the second the
interior branch $V^{\ast}(p)=\left(\nmax+\nmin\right)/\left(4\nmax\nmin\right)-p^{2}/\left(\nmax+\nmin\right)$.
Their ordering follows at once.

\begin{proposition}
\label{prop:ErrorBoundsOrder}
For all $n_1,n_2,z>0$ and $p\in[-1,1]$,
\begin{equation*}
\Lstwo\le\Lsone\le p\le\Usone\le\Ustwo ,
\end{equation*}
with equality throughout if and only if $\left|p\right|=\hp_{\square}$ or $z=0$.
\end{proposition}
\begin{proof}
The proof of Proposition~\ref{prop:ErrorBoundsOrder} is relegated to
Appendix~\ref{app:ErrorBoundsOrder}.
\end{proof}

We may now state the main result of this section. It replaces a case analysis by a
single expression in $\Phi$.

\begin{theorem}
\label{thm:mainError}
Let $(p_1,p_2)\in(0,1)^{2}$ be the true parameter pair, $p=p_1-p_2$, and let
$V=V(p_1,p_2)>0$ be as in \eqref{eq:Vtrue}. Define the \emph{variance inflation
factor}
\begin{equation}
\label{eq:inflation}
\mathcal{R}\left(p_1,p_2\right)=\sqrt{\frac{V^{\ast}(p)}{V\left(p_1,p_2\right)}}\ \ge\ 1 .
\end{equation}
Then, under the normal approximation $\hp\sim N\left(p,V\right)$, the attained
coverage of the \EO{} interval is
\begin{equation}
\label{eq:coverageEO}
\mathbb{P}\left(\left(p_1,p_2\right)\in\Sr\right)=2\Phi\!\left(z\,\mathcal{R}\left(p_1,p_2\right)\right)-1,
\end{equation}
and consequently
\begin{equation}
\label{eq:errorformula}
\mathrm{Error}\left(p_1,p_2\right)=2\Bigl[\Phi\!\left(z\,\mathcal{R}\left(p_1,p_2\right)\right)-\Phi\left(z\right)\Bigr].
\end{equation}
Moreover, writing
\begin{equation}
\label{eq:deltadef}
\delta\left(p_1,p_2\right)=n_2\left(p_1-\tfrac12\right)+n_1\left(p_2-\tfrac12\right),
\end{equation}
the inflation factor admits, on the interior branch
$\left|p\right|\le\hp_{\square}$, the exact representation
\begin{equation}
\label{eq:excessIdentity}
V^{\ast}(p)-V\left(p_1,p_2\right)=\frac{\delta\left(p_1,p_2\right)^{2}}{n_1n_2\left(n_1+n_2\right)},
\qquad\text{so that}\qquad
\mathcal{R}^{2}=1+\frac{\delta\left(p_1,p_2\right)^{2}}{n_1n_2\left(n_1+n_2\right)V\left(p_1,p_2\right)} .
\end{equation}
\end{theorem}
\begin{proof}
The proof of Theorem~\ref{thm:mainError} is presented in
Appendix~\ref{app:mainError}.
\end{proof}

The structure of \eqref{eq:errorformula} is worth dwelling on. The entire
dependence of the error on the parameters is channelled through the single scalar
$\mathcal{R}$, which measures how much larger the least-favourable variance is
than the true one. If the true pair happens to be the least-favourable pair on its
own difference line, then $\mathcal{R}=1$ and the interval is exact; otherwise the
interval behaves as though it had been computed at an inflated critical value
$z\mathcal{R}$ in place of $z$, and over-covers accordingly. The identity
\eqref{eq:excessIdentity} makes the geometry explicit: the excess variance is a
perfect square, vanishing exactly on a line and growing quadratically in the
signed deviation $\delta$ from it.

\begin{corollary}
\label{cor:ErrorProperties}
Under the hypotheses of Theorem~\ref{thm:mainError} the following hold.
\begin{itemize}
\item[$(a)$] \emph{Range.} $0\le\mathrm{Error}\left(p_1,p_2\right)<\a$ for every
$\left(p_1,p_2\right)\in(0,1)^{2}$, and the upper bound is approached but not
attained: $\mathrm{Error}\to\a$ as $\left(p_1,p_2\right)\to(0,0)$ or
$\left(p_1,p_2\right)\to(1,1)$.
\item[$(b)$] \emph{Zero-error locus.} On the interior branch,
$\mathrm{Error}\left(p_1,p_2\right)=0$ if and only if
$\delta\left(p_1,p_2\right)=0$, that is, if and only if
\begin{equation}
\label{eq:valleyline}
\b_1p_1+\b_2p_2=1,
\qquad
\b_1=\frac{2n_2}{n_1+n_2},
\qquad
\b_2=\frac{2n_1}{n_1+n_2},
\end{equation}
a line of slope $-n_2/n_1$ through the centre $\left(\tfrac12,\tfrac12\right)$ of
the unit square, with $\b_1+\b_2=2$. In the balanced design $n_1=n_2$ this is the
opposite diagonal $p_1+p_2=1$. On the outer branch
$\left|p\right|>\hp_{\square}$ the error vanishes only on the boundary of the unit
square, and is therefore strictly positive throughout the interior.
\item[$(c)$] \emph{Symmetry.} $\mathrm{Error}\left(1-p_1,1-p_2\right)=\mathrm{Error}\left(p_1,p_2\right)$;
that is, the error surface is invariant under the point reflection about
$\left(\tfrac12,\tfrac12\right)$. It is likewise invariant under the simultaneous
interchange of $\left(p_1,p_2\right)$ with $\left(p_2,p_1\right)$ and of $n_1$
with $n_2$.
\item[$(d)$] \emph{Scale invariance.} If $n_1=c_1N$ and $n_2=c_2N$ with
$c_1,c_2>0$ fixed, then $\mathcal{R}\left(p_1,p_2\right)$ does not depend on $N$.
Consequently the coverage error does \emph{not} tend to zero as the sample sizes
grow: for fixed $\left(p_1,p_2\right)$ off the line \eqref{eq:valleyline} it
converges to the strictly positive constant
$2\left[\Phi\left(z\mathcal{R}\right)-\Phi\left(z\right)\right]$. In the balanced
design with $p_1=p_2=q$ one has $\mathcal{R}^{2}=\left[4q(1-q)\right]^{-1}$ for
every $n$.
\end{itemize}
\end{corollary}
\begin{proof}
The proof of Corollary~\ref{cor:ErrorProperties} is relegated to
Appendix~\ref{app:ErrorProperties}.
\end{proof}

Part $(d)$ deserves to be stated without euphemism, since it delimits the claims
that can legitimately be made for the method. The \EO{} interval is conservative
by construction and its conservatism is a fixed function of the parameter pair and
of the ratio of sample sizes; it is not an $O\!\left(n^{-1}\right)$ correction
that washes out asymptotically. Enlarging the samples shrinks the interval at the
usual $O\!\left(n^{-1/2}\right)$ rate, but it shrinks the ideal interval at the
same rate, and the ratio of the two lengths tends to $\mathcal{R}>1$. The \EO{}
interval is therefore a valid conservative procedure, not an asymptotically
efficient one, and any comparison with competing methods must be conducted with
that in view.

\begin{remark}
\label{rm:approximationerror}
Two idealisations underlie \eqref{eq:errorformula} and should be recorded
explicitly. The first is the normal approximation itself: we have used
$\hp\sim N\left(p,V\right)$ both for the attained coverage and for the nominal
level, so that \eqref{eq:errorformula} isolates the \emph{reduction error} --- the
price of passing from $\Sx$ to $\Sr$ --- from the \emph{approximation error} of
the normal limit. To keep the latter small we assume throughout the usual
regularity conditions $n_1p_1,\ n_1(1-p_1),\ n_2p_2,\ n_2(1-p_2)\ge5$, a
requirement sometimes strengthened to $10$. It should be emphasised that these
conditions bound the approximation error only; they do not bound the reduction
error, which by Corollary~\ref{cor:ErrorProperties}$(d)$ is scale-invariant. At
$n_1=n_2=45$ and $\a=5\%$, for instance, the boundary of the admissible region
$p_i\in\left[\tfrac19,\tfrac89\right]$ still yields
$\mathrm{Error}=0.0482$, which is $96\%$ of the theoretical maximum $\a$. The
second idealisation is that we have not truncated the acceptance region
\eqref{eq:acceptance} to $[-1,1]$, although $\hp$ takes values there. Under the
stated regularity the normal law assigns negligible mass outside that range, so
the truncated and untruncated probabilities agree to within the approximation
error already accepted; the exact finite-sample treatment of
Section~\ref{subsec:exactcoverage} dispenses with both idealisations.
\end{remark}

\subsection{Behaviour of the error over the parameter space}
\label{subsec:errorplots}

We now examine the error surface numerically. The sample sizes are fixed at
$n_1=n_2=45$ in Figure~\ref{fig:Error-n1en2}; the complementary unbalanced designs
are deferred to Appendix~\ref{app:Error-n1gn2} and
Appendix~\ref{app:Error-n2gn1}. In each case the domain is the admissible region
determined by the regularity conditions of Remark~\ref{rm:approximationerror}.

\begin{figure}[htbp]
\centering
    \begin{subfigure}[t]{0.48\textwidth}
        \centering
        \includegraphics[width=\linewidth]{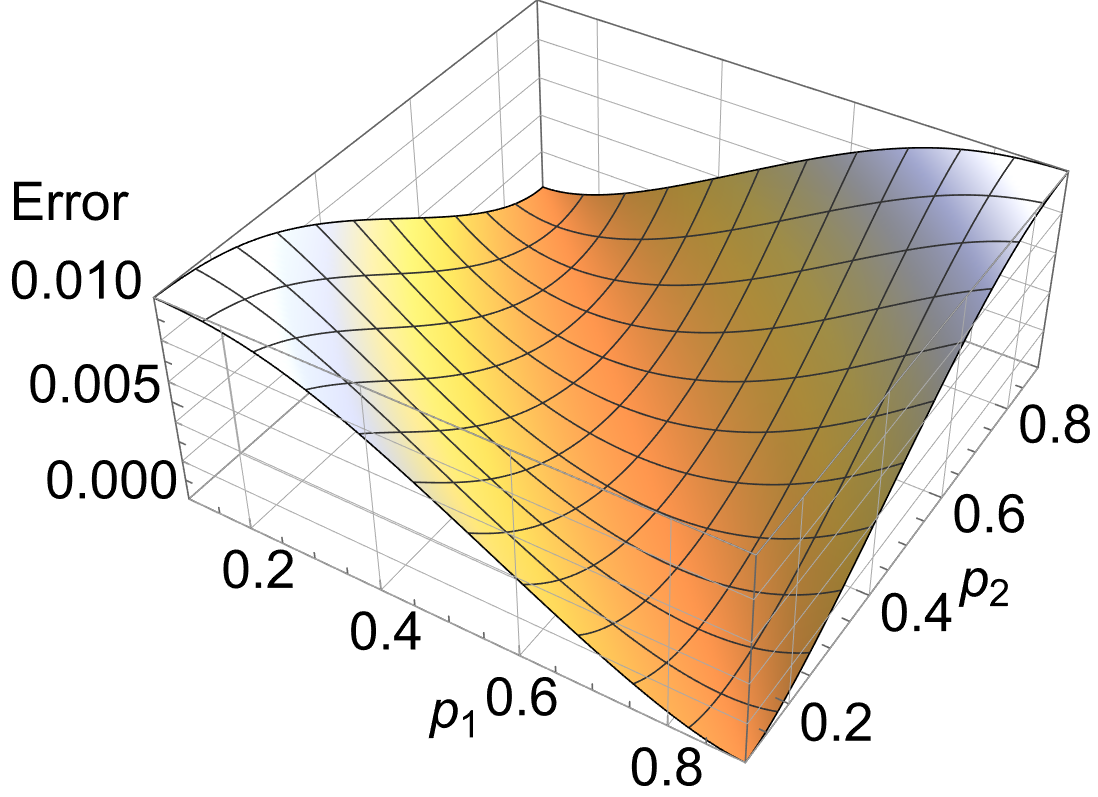}
        \caption{$\a=1\%$}
        \label{subfig:1_n1en2}
    \end{subfigure}
    \hfill
    \begin{subfigure}[t]{0.48\textwidth}
        \centering
        \includegraphics[width=\linewidth]{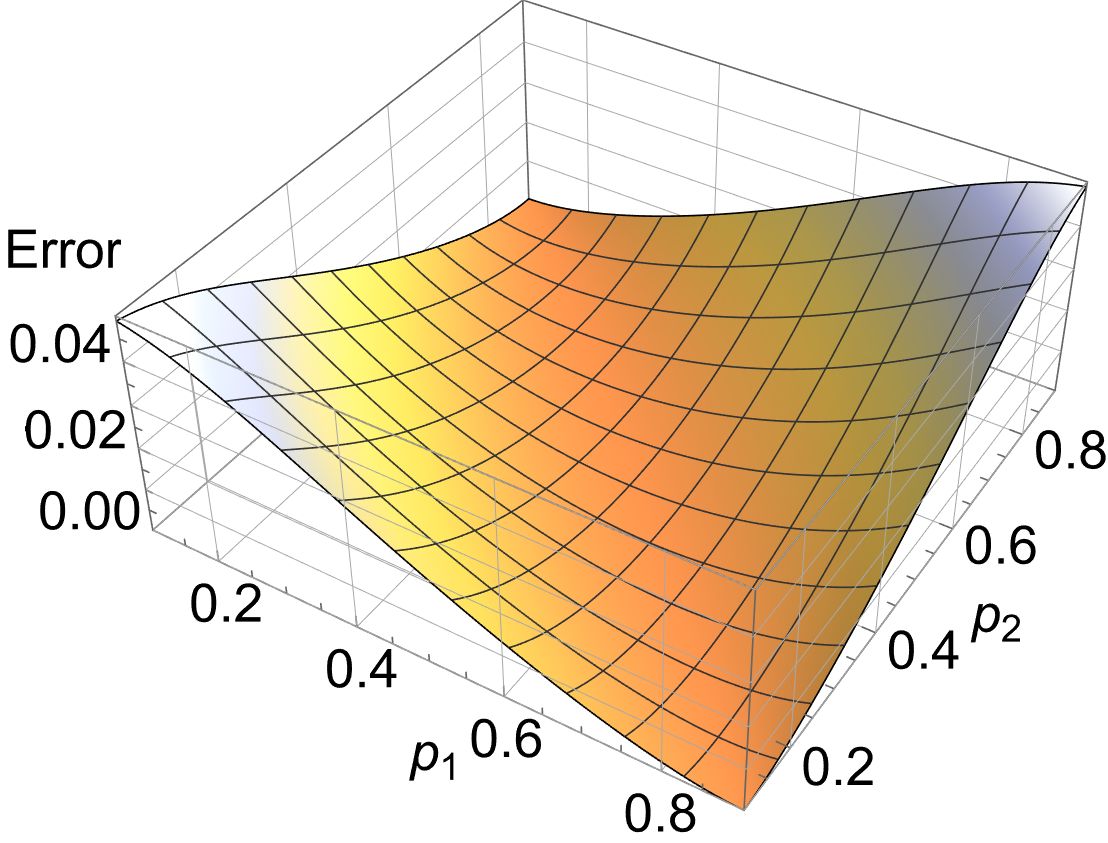}
        \caption{$\a=5\%$}
        \label{subfig:5_n1en2}
    \end{subfigure}

    \vspace{0.8\baselineskip}

    \begin{subfigure}[t]{0.48\textwidth}
        \centering
        \includegraphics[width=\linewidth]{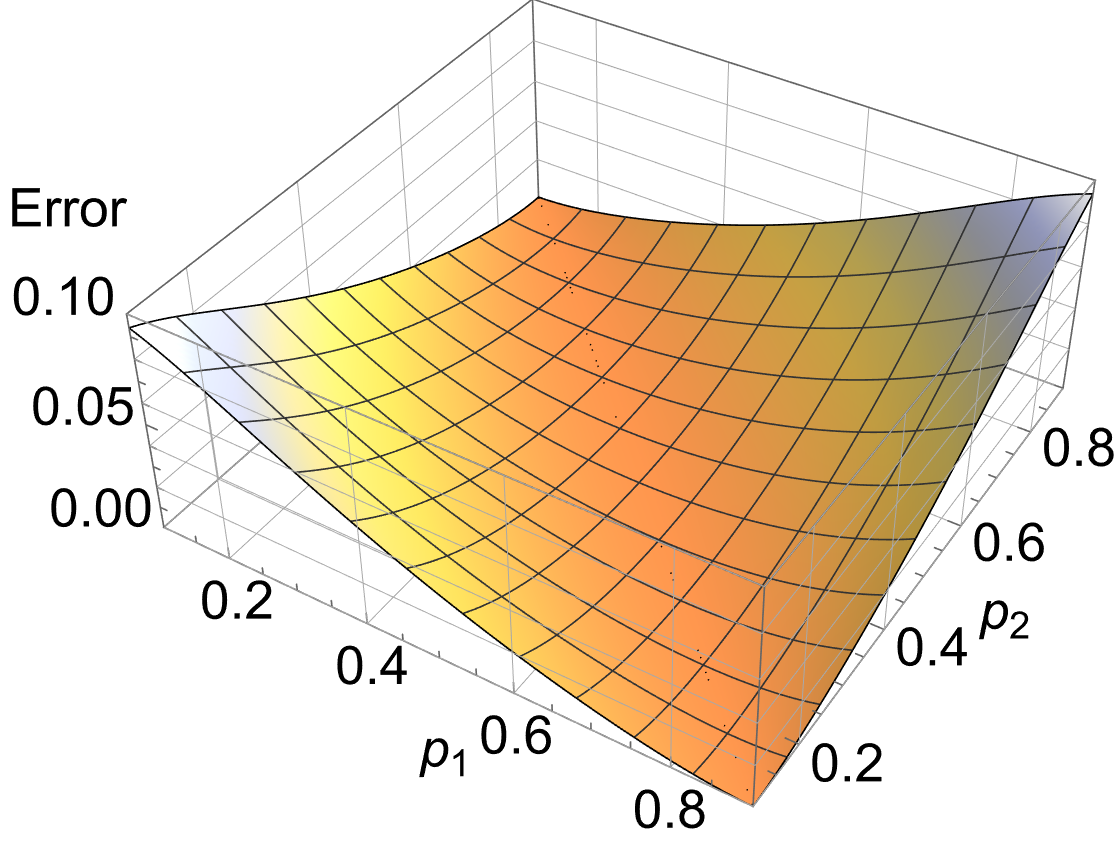}
        \caption{$\a=10\%$}
        \label{subfig:10_n1en2}
    \end{subfigure}
\caption[Coverage error over the parameter domain for $n_1=n_2$]{Coverage error
\eqref{eq:errorformula} of the \EO{} interval over the parameter domain
$\left(p_1,p_2\right)$, for three significance levels, with both sample sizes
fixed at $n_1=n_2=45$. The error vanishes along the opposite diagonal
$p_1+p_2=1$, which is the line \eqref{eq:valleyline} in the balanced case, and
rises toward the corners $(0,0)$ and $(1,1)$, where it approaches its supremum
$\a$. Note the vertical scales: the surfaces attain $0.010$, $0.048$ and $0.091$
respectively.}
\label{fig:Error-n1en2}
\end{figure}

Figures~\ref{fig:Error-n1en2}, \ref{fig:Error-n1gn2} and \ref{fig:Error-n2gn1}
exhibit exactly the structure predicted by Theorem~\ref{thm:mainError} and
Corollary~\ref{cor:ErrorProperties}, and it is worth checking the correspondence
point by point.

Each surface possesses a valley along which the error vanishes, and that valley is
the line \eqref{eq:valleyline}. For $n_1=n_2=45$ it is the opposite diagonal
$p_1+p_2=1$, which coincides with a diagonal of the plotted domain; accordingly
the surface in Figure~\ref{fig:Error-n1en2} descends to zero at two opposite
corners, and one may verify directly that at $\left(p_1,p_2\right)=(0.8,0.2)$ both
$V^{\ast}$ and $V$ equal $0.00711$, so that $\mathcal{R}=1$. For $n_1=60$ and
$n_2=30$ the valley is $p_1+2p_2=\tfrac32$, of slope $-\tfrac12$, and it runs
through the interior of the domain rather than corner to corner: in
Figure~\ref{fig:Error-n1gn2} the trough is visibly interior, and at
$\left(p_1,p_2\right)=(0.9,0.1)$ the error is $0.0171$ rather than zero. For
$n_1=30$ and $n_2=60$ the valley is $2p_1+p_2=\tfrac32$, of slope $-2$, and
Figure~\ref{fig:Error-n2gn1} is the mirror image of
Figure~\ref{fig:Error-n1gn2}, as the second symmetry of
Corollary~\ref{cor:ErrorProperties}$(c)$ requires.

Away from the valley the error rises in both directions, and the rise is governed
by \eqref{eq:excessIdentity}: the excess variance is exactly quadratic in the
signed deviation $\delta$, so the contours of the error are the level sets of
$\delta^{2}/V$. These are conics rather than pairs of parallel lines, because the
denominator $V$ itself varies over the square; the ascent is therefore steeper
toward the corners, where $V$ is small, than toward the edges. The maximum is
approached at $(0,0)$ and $(1,1)$, the two corners at which the true variance
degenerates while the least-favourable variance does not; at the other two corners
$V^{\ast}$ degenerates as well, and for $n_1=n_2$ the error there tends to zero,
consistently with those corners lying on the valley. Finally, every surface is
symmetric under reflection through the centre $\left(\tfrac12,\tfrac12\right)$, as
Corollary~\ref{cor:ErrorProperties}$(c)$ guarantees.

\begin{remark}
\label{rm:toowide}
The behaviour in the corners should be stated plainly, since it bears directly on
the comparisons of Section~\ref{sec:comparison}. When both true proportions are
small, or both large, the \EO{} interval is substantially too wide: at
$n_1=n_2=45$, $\a=5\%$ and $p_1=p_2=0.15$ its attained coverage is $99.4\%$, and
at $p_1=p_2=\tfrac19$ it is $99.8\%$. This is not a defect of the numerical
implementation but an exact consequence of \eqref{eq:errorformula}, and by
Corollary~\ref{cor:ErrorProperties}$(d)$ it does not improve with sample size.
Consequently the \EO{} interval cannot serve as a yardstick of correct width: an
interval that is narrower than the \EO{} interval in this region is not thereby
shown to be too narrow, and may well be closer to the nominal level. Establishing
which of two intervals is the more accurate requires their attained coverages to
be compared with each other and with the nominal level, which is the purpose of
the next subsection.
\end{remark}

\subsection{Exact coverage of the three intervals}
\label{subsec:exactcoverage}

The analysis above is exact in the reduction but approximate in the sampling
distribution. For the comparative assessment of
Section~\ref{sec:comparison} we require attained coverages that are exact in both
respects, and for all three procedures under consideration. These are available by
direct enumeration, since $\hp_1$ and $\hp_2$ are independent binomial
proportions on finite supports.

\begin{definition}
\label{def:exactcoverage}

For an interval procedure $M$ assigning to each observed pair
$\left(\hp_1,\hp_2\right)$ an interval
$\left[L_M\left(\hp_1,\hp_2\right),U_M\left(\hp_1,\hp_2\right)\right]$, the exact
coverage function, for given $n_1,n_2$ and $z$, is
\begin{multline}
\label{eq:exactcoverage}
C_M\left(p_1,p_2\right)
=\sum_{x_1=0}^{n_1}\sum_{x_2=0}^{n_2}
\binom{n_1}{x_1}p_1^{x_1}\left(1-p_1\right)^{n_1-x_1}
\binom{n_2}{x_2}p_2^{x_2}\left(1-p_2\right)^{n_2-x_2}\,
\\
\times\mathbf{1}\!\left\{L_M\!\left(\tfrac{x_1}{n_1},\tfrac{x_2}{n_2}\right)
\le p_1-p_2\le
U_M\!\left(\tfrac{x_1}{n_1},\tfrac{x_2}{n_2}\right)\right\},
\end{multline}
that is, the probability that the random interval produced by $M$ covers the true
difference. We write $C_{\EO}$, $C_{\mathrm{Wald}}$
and $C_{\mathrm{NC}}$ for the three procedures, where the \EO{} bounds are those
of Table~\ref{tab:CIresults}, the Wald bounds are
$\hp\pm z\sqrt{\widehat{\Var}\left[\hp\right]}$ with
$\widehat{\Var}\left[\hp\right]=\hp_1\left(1-\hp_1\right)/n_1+\hp_2\left(1-\hp_2\right)/n_2$,
and the Newcombe bounds are those of Method~10 of \cite{Newcombe1998}, namely
\begin{equation}
\label{eq:NewcombeCI}
L_{\mathrm{NC}}=\hp-\sqrt{\left(\hp_1-l_1\right)^{2}+\left(u_2-\hp_2\right)^{2}},
\qquad
U_{\mathrm{NC}}=\hp+\sqrt{\left(u_1-\hp_1\right)^{2}+\left(\hp_2-l_2\right)^{2}},
\end{equation}
with $\left[l_i,u_i\right]$ the univariate Wilson interval
\eqref{eq:WilsonCIOnePop} computed from $\hp_i$ and $n_i$.
\end{definition}

Expression \eqref{eq:exactcoverage} requires no approximation and is computable in
$O\!\left(n_1n_2\right)$ operations; the coverage \emph{error} of a procedure is
then $C_M\left(p_1,p_2\right)-(1-\a)$, which, unlike \eqref{eq:errordef}, may take
either sign. For the \EO{} interval the indicator has a simpler structure, because by Lemma~\ref{lem:ErrorInversion} it depends on $\left(x_1,x_2\right)$ only through the difference $x_1/n_1-x_2/n_2$; the exact coverage function, for given $n_1,n_2$ and $z$, is then the probability that this difference falls in the interval \eqref{eq:acceptance}, of which \eqref{eq:coverageEO} is the normal approximation. No
comparable simplification is available for the Wald and Newcombe intervals, whose
bounds depend on $\hp_1$ and $\hp_2$ separately.

For the Wald interval, however, one exact statement can be made without
enumeration, and it is the statement that the comparative discussion requires.

\begin{proposition}
\label{prop:WaldBias}
The Wald variance estimator satisfies, exactly and for all $n_1,n_2\ge1$,
\begin{equation}
\label{eq:Waldbias}
\mathbb{E}\left[\widehat{\Var}\left[\hp\right]\right]
=\Var\left[\hp\right]-\left(\frac{p_1\left(1-p_1\right)}{n_1^{2}}+\frac{p_2\left(1-p_2\right)}{n_2^{2}}\right)
\ <\ \Var\left[\hp\right],
\end{equation}
and in the balanced design $n_1=n_2=n$ this reduces to
$\mathbb{E}\left[\widehat{\Var}\left[\hp\right]\right]=\left(1-1/n\right)\Var\left[\hp\right]$.
\end{proposition}
\begin{proof}
The proof of Proposition~\ref{prop:WaldBias} is relegated to
Appendix~\ref{app:WaldBias}.
\end{proof}

Identity \eqref{eq:Waldbias} shows that the plug-in variance on which the Wald
interval rests is systematically too small, by a relative amount that is exactly
$1/n$ in the balanced case. The Wald interval is therefore, on average, shorter
than the interval $\hp\pm z\sqrt{\Var\left[\hp\right]}$ that would attain the
nominal level under the normal approximation, and its coverage is correspondingly
depressed. This is a quantified statement about a specific mechanism, and it is
the only sense in which the present paper asserts that the Wald interval is ``too
narrow''; the assertion is not, and cannot be, established by observing that the
Wald bounds lie inside the bounds of a conservative interval. With
\eqref{eq:errorformula} for the \EO{} interval, \eqref{eq:exactcoverage} for all
three, and \eqref{eq:Waldbias} for the Wald interval in particular, the
comparative analysis of Section~\ref{sec:comparison} can be conducted on the basis
of attained coverage rather than width.

\section{Comparison and Discussion}
\label{sec:comparison}

Section~\ref{sec:error_analysis} supplied what a comparison of interval procedures
requires: a closed-form expression \eqref{eq:errorformula} for the coverage excess
of the \EO{} interval, an exact finite-sample coverage function
\eqref{eq:exactcoverage} applicable to any procedure whatever, and an exact
identity \eqref{eq:Waldbias} for the bias of the variance estimator on which the
Wald interval rests. We now use these to compare the \EO{} interval with the two
established alternatives. The comparison is conducted on the basis of attained
coverage, and its conclusions are more qualified --- and, we believe, more useful
--- than a summary verdict would be.

\subsection{What a comparison can and cannot establish}
\label{subsec:basisofcomparison}

It is worth being explicit about the standard of evidence, because the natural
informal comparison is misleading in both directions.

The three procedures produce intervals of different lengths, and it is tempting to
read length as accuracy. It is not. A narrower interval is more accurate only if
its coverage remains at the nominal level; if the coverage has fallen, the
narrowness is exactly the symptom. Conversely, an interval that is wider than
another is not thereby shown to be the correct one: it may simply be conservative,
and by Corollary~\ref{cor:ErrorProperties} the \EO{} interval demonstrably is.
No one of the three can serve as a yardstick for the others. The only quantity that
adjudicates between them is the attained coverage probability, compared with the
nominal level, and computed for each procedure separately at each parameter pair.
Every claim in what follows is therefore either a structural fact that holds for
every sample, or a statement about $C_M\left(p_1,p_2\right)$ as defined in
\eqref{eq:exactcoverage}. In particular, the plots of interval bounds in
Figures~\ref{fig:CICq} and \ref{fig:ECICw} below are presented as description, not
as evidence; they show what the procedures do, not which of them is right.

For reference, the Wald interval for the difference of two proportions is
\begin{equation}
\label{eq:WaldCI}
\hp-z\sqrt{\widehat{\Var}\left[\hp\right]}\ \le\ p\ \le\ \hp+z\sqrt{\widehat{\Var}\left[\hp\right]},
\qquad
\widehat{\Var}\left[\hp\right]=\ddfrac{\hp_1\left(1-\hp_1\right)}{n_1}+\ddfrac{\hp_2\left(1-\hp_2\right)}{n_2},
\end{equation}
obtained by substituting the observed proportions for the unknown parameters
inside the variance, and the Newcombe interval is Method~10 of \cite{Newcombe1998},
given in \eqref{eq:NewcombeCI}, which combines the univariate Wilson limits
\eqref{eq:WilsonCIOnePop} of the two proportions by a square-and-add rule.

\subsection{Structural properties}
\label{subsec:structural}

Before turning to coverage we record the properties that can be settled once and
for all, without reference to any particular parameter value. These are genuine
points of difference, and they are established rather than observed.

\begin{proposition}
\label{prop:structural}
For all $n_1,n_2\ge1$ and all $z>0$:
\begin{itemize}
\item[$(a)$] The \EO{} interval always lies in $[-1,1]$ and always contains the
point estimate $\hp$, by Corollary~\ref{cor:EOadmissible}. The Wald interval
\eqref{eq:WaldCI} possesses neither property in general: at $n_1=n_2=20$ and $p_1=p_2=0.05$, a configuration outside the grid of Table~\ref{tab:coveragesummary}, the Wald interval attains $0.9883$.
\item[$(b)$] The \EO{} interval is never degenerate: its length is strictly
positive for every possible sample. This includes $\hp=\pm1$, where $V^{\ast}$ vanishes: since $V^{\ast}(p)\sim(1-|p|)/\nmin$ as $|p|\to1$ while $(\hp-p)^{2}$ vanishes quadratically, the interval at $\hp=1$ extends at least down to $1-z^{2}/\nmin$. The Wald interval collapses to a single
point whenever both samples are homogeneous, that is at the four outcomes
$\hp_1,\hp_2\in\{0,1\}$.
\item[$(c)$] The \EO{} interval is equivariant under relabelling of the two
populations, $\mathrm{CI}(-\hp)=-\,\mathrm{CI}(\hp)$, and for each fixed
hypothesised difference $p$ its acceptance region \eqref{eq:acceptance} is an
interval symmetric about $p$. This is the precise sense --- and the only sense ---
in which the procedure is symmetric.
\item[$(d)$] The \EO{} interval uses the data only through the difference $\hp$,
by Remark~\ref{rm:CIresultsremmarks}, whereas both alternatives depend on $\hp_1$
and $\hp_2$ separately.
\item[$(e)$] Under the normal approximation the \EO{} interval never under-covers:
$\mathcal{R}\ge1$ in \eqref{eq:inflation} gives
$C_{\EO}\ge1-\a$ for every parameter pair. No such guarantee is available for
either alternative, and the size of the excess is known in advance in closed form
from \eqref{eq:errorformula}.
\end{itemize}
\end{proposition}
\begin{proof}
Parts $(a)$--$(d)$ were established in Corollary~\ref{cor:EOadmissible}, Lemma~\ref{lem:profile}, Lemma~\ref{lem:ErrorInversion}, the discussion following \eqref{eq:R2} and Remark~\ref{rm:CIresultsremmarks}; part $(b)$ follows from the linear vanishing of $V^{\ast}$ at $\pm1$ noted in the statement. The counts for the Wald interval follow by
direct enumeration of the $\left(n_1+1\right)\left(n_2+1\right)$ sample outcomes.
Part $(e)$ is Theorem~\ref{thm:mainError} together with the inequality
$V^{\ast}(p)\ge V\left(p_1,p_2\right)$ of Lemma~\ref{lem:profile}.
\end{proof}

Part $(e)$ is the principal structural advantage of the method, and it is worth
separating from the empirical questions taken up below. The \EO{} interval offers
a coverage guarantee together with an exact accounting of its cost. That
combination is unusual: conservative procedures are common, but conservative
procedures whose conservatism is available in closed form, as a function of the
parameters, are not.

\subsection{Interval bounds across the parameter space}
\label{subsec:widths}

We first describe how the three sets of bounds relate to one another, deferring all
questions of accuracy to Section~\ref{subsec:coverage}. Both sample sizes are set
to $n_1=n_2=20$ and the significance level is fixed at $5\%$; the bounds are
plotted against $\hp_1$ with $\hp_2$ held fixed at each of
$\{0.01,\,0.25,\,0.75,\,0.99\}$. The plots obtained by interchanging the roles of $\hp_1$ and $\hp_2$ are given as Figure~\ref{fig:CICp} in Appendix~\ref{app:CICp}.
\begin{figure}[htbp]
\centering
    \begin{subfigure}[t]{0.48\textwidth}
        \centering
        \includegraphics[width=\linewidth]{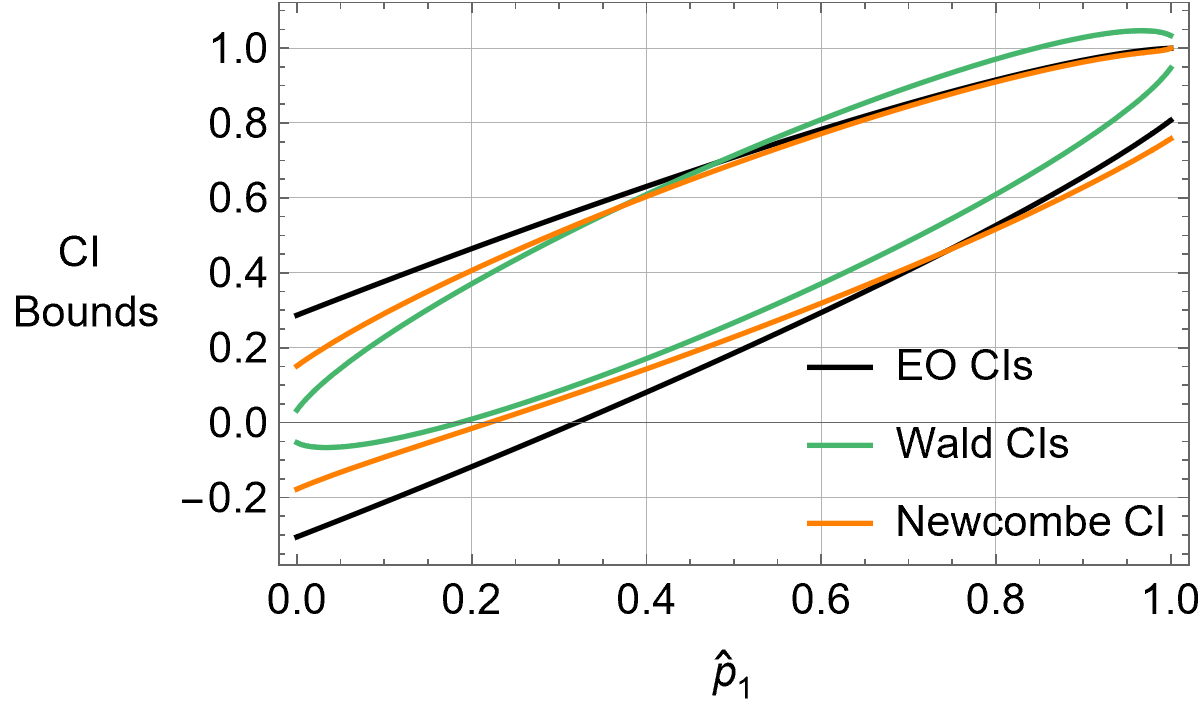}
        \caption{$\hp_2=0.01$}
        \label{subfig:CICq001}
    \end{subfigure}
    \hfill
    \begin{subfigure}[t]{0.48\textwidth}
        \centering
        \includegraphics[width=\linewidth]{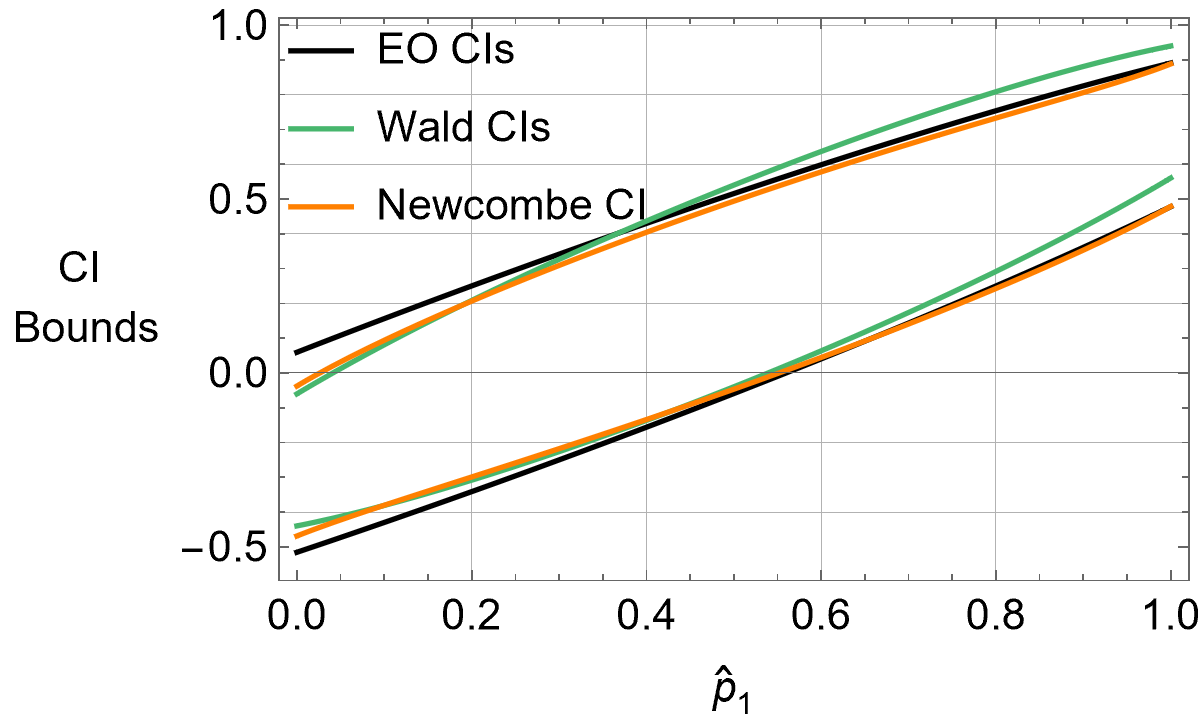}
        \caption{$\hp_2=0.25$}
        \label{subfig:CICq025}
    \end{subfigure}

    \vspace{0.8\baselineskip}

    \begin{subfigure}[t]{0.48\textwidth}
        \centering
        \includegraphics[width=\linewidth]{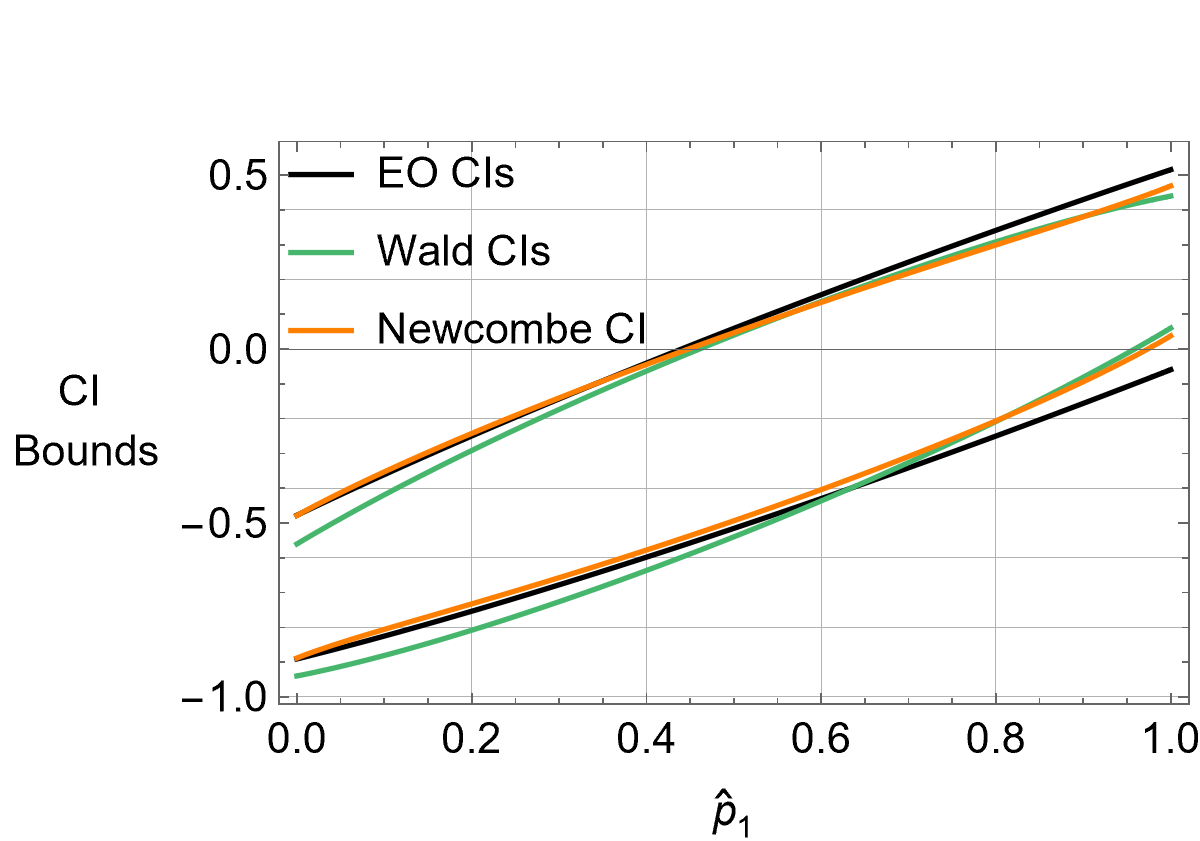}
        \caption{$\hp_2=0.75$}
        \label{subfig:CICq075}
    \end{subfigure}
    \hfill
    \begin{subfigure}[t]{0.48\textwidth}
        \centering
        \includegraphics[width=\linewidth]{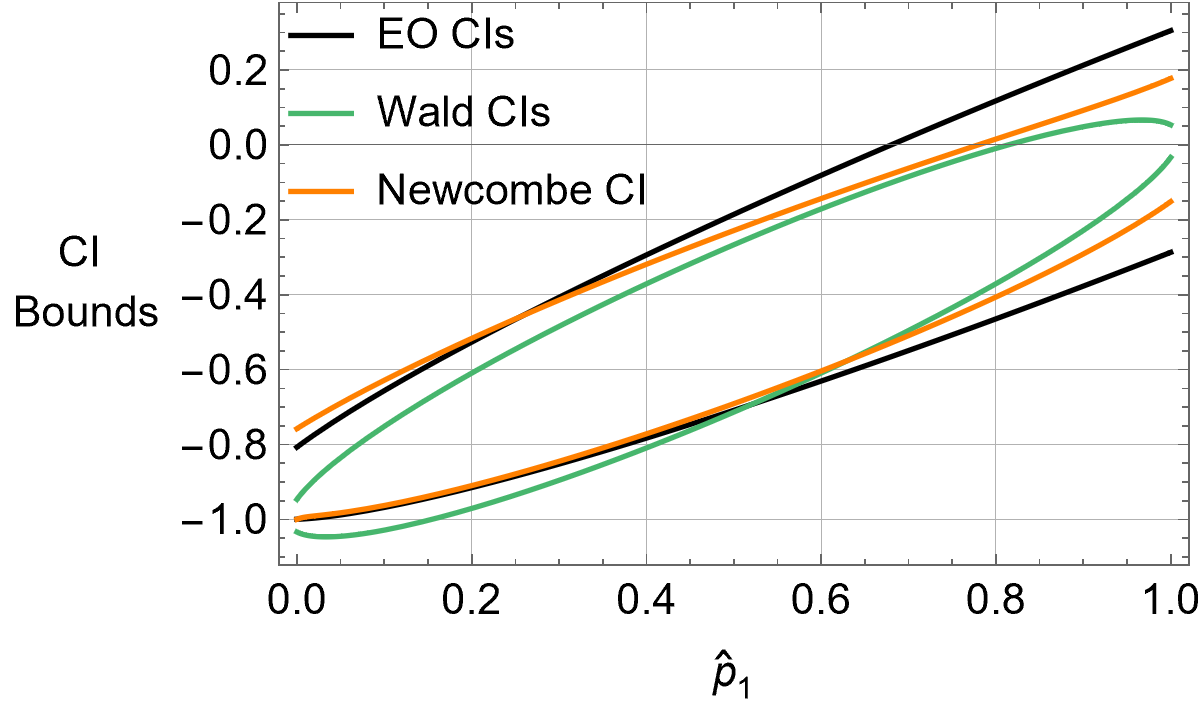}
        \caption{$\hp_2=0.99$}
        \label{subfig:CICq099}
    \end{subfigure}
\caption[Bounds of the \EO{}, Wald and Newcombe CIs against $\widehat{p}_1$]{Bounds of the
\EO{}, Wald and Newcombe intervals plotted against $\hp_1$ for four fixed values
of $\hp_2$, with $n_1=n_2=20$ and $\a=5\%$. The figure records the position of the
bounds only; the accuracy of the three procedures is assessed in
Table~\ref{tab:coveragepoints}.}
\label{fig:CICq}
\end{figure}

Two features of Figure~\ref{fig:CICq} are worth noting because they reappear in
the coverage analysis. First, the \EO{} bounds are smooth and monotone in $\hp_1$
across the whole range, including the two extreme panels, and they remain within
$[-1,1]$; the Wald bounds, by contrast, turn over near the ends of the range and
in the panel $\hp_2=0.01$ exceed unity. This is the behaviour predicted by
Proposition~\ref{prop:structural}$(a)$--$(b)$, and it is a statement about
admissibility, not about accuracy. Second, in the two central panels the three
sets of bounds are close together, which already suggests that the interesting
differences lie near the boundary of the parameter space rather than in the
interior.

The relationship between the bounds across the whole $\left(\hp_1,\hp_2\right)$
domain is displayed in Figure~\ref{fig:EOWaldNCorderingsdomain}, which partitions
the unit square according to the ordering of the endpoints.

\begin{figure}[htbp]
\centering
    \begin{subfigure}{\textwidth}
    \centering
    \includegraphics[width=0.86\linewidth]{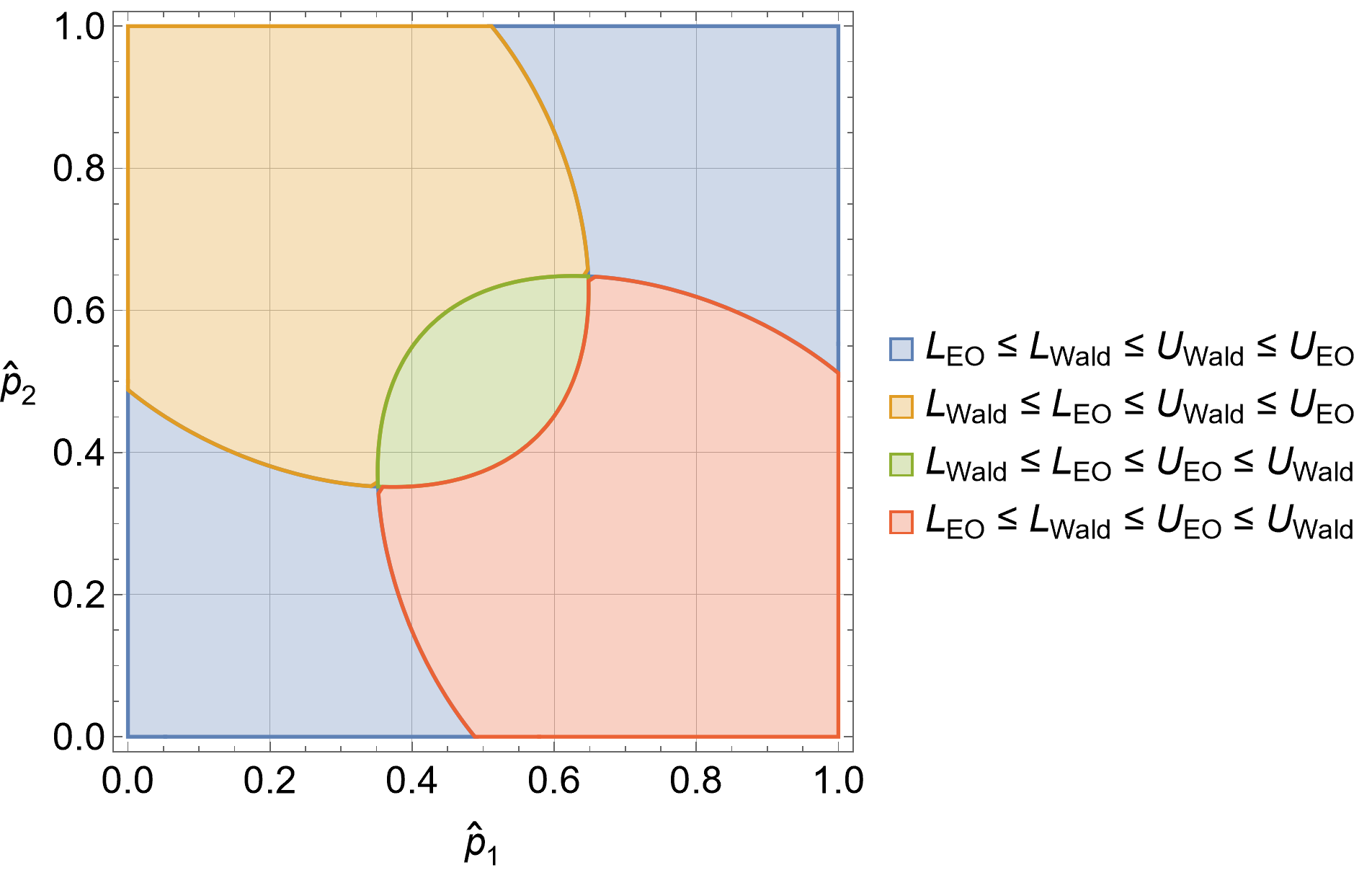}
    \caption{\EO{} against Wald}
    \label{subfig:EOW}
    \end{subfigure}

    \vspace{0.8\baselineskip}

    \begin{subfigure}{\textwidth}
    \centering
    \includegraphics[width=0.86\linewidth]{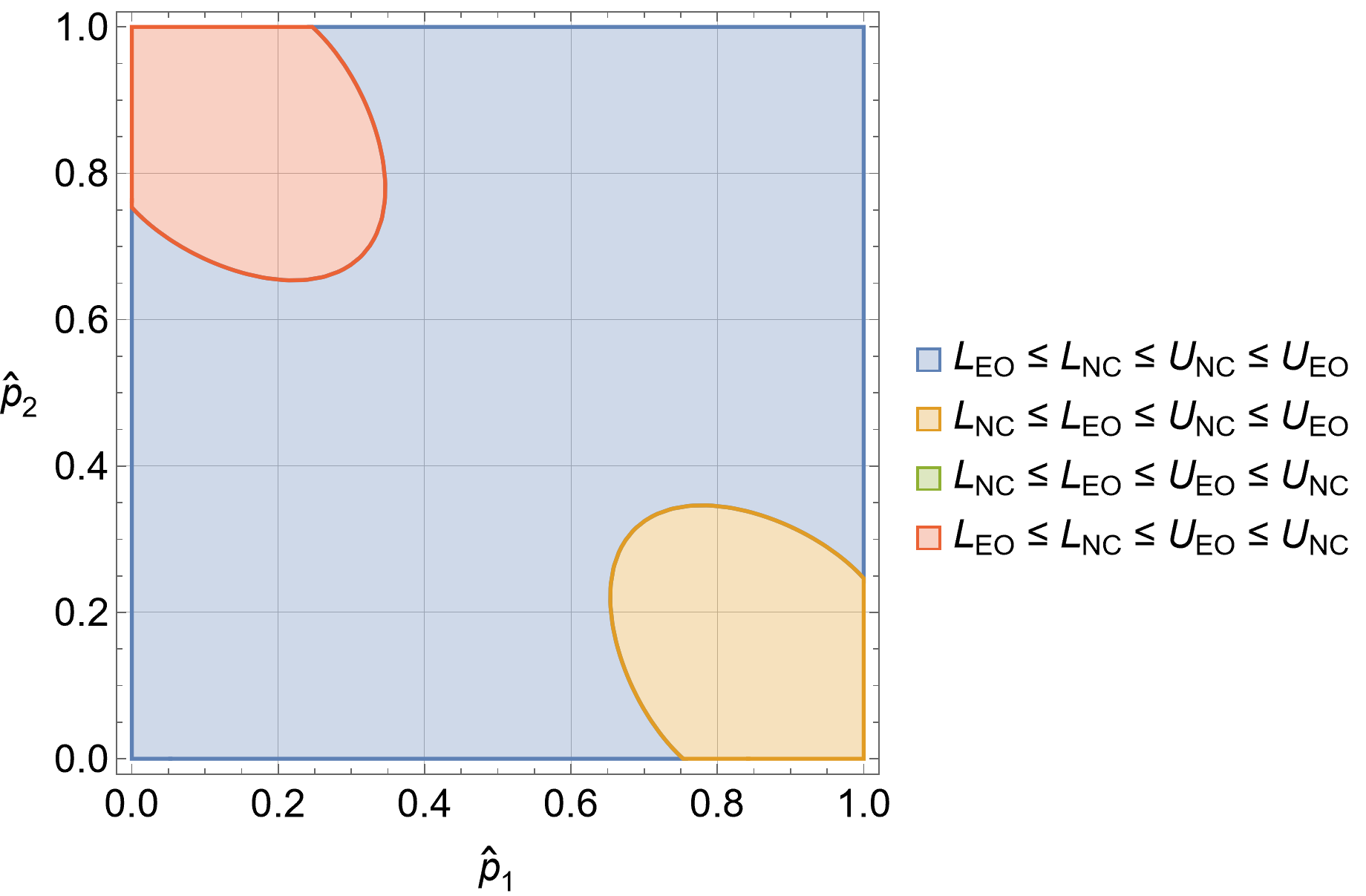}
    \caption{\EO{} against Newcombe}
    \label{subfig:EONC}
    \end{subfigure}
\caption[Orderings of the \EO{}, Wald and Newcombe bounds]{Orderings of the
endpoints of the \EO{} and Wald intervals, and of the \EO{} and Newcombe
intervals, over the $\left(\hp_1,\hp_2\right)$ domain, for $n_1=n_2=20$ and
$\a=5\%$. Here $L_{\mathrm{EO}},U_{\mathrm{EO}}$, $L_{\mathrm{Wald}},U_{\mathrm{Wald}}$
and $L_{\mathrm{NC}},U_{\mathrm{NC}}$ denote the lower and upper bounds of the
three procedures.}
\label{fig:EOWaldNCorderingsdomain}
\end{figure}

The two panels differ in character. Against Newcombe the picture is almost
one-directional: over the great majority of the square the Newcombe interval is
nested inside the \EO{} interval, and the two small lobes in which the nesting is
partial occur where one observed proportion is near zero and the other near one.
Against Wald the picture is genuinely four-fold: the Wald bounds may lie inside
the \EO{} bounds, outside them, or straddle them on either side, and all four
configurations occupy regions of positive area. This is the honest content of the
observation that the Wald interval behaves erratically --- its position relative to
a fixed reference interval is not a monotone function of the data --- and it is a
separate matter from whether it is too narrow, which the orderings alone cannot
settle.

Finally, Figure~\ref{fig:ECICw} superimposes the three intervals on the elliptical
region $\Sx$ itself, for four representative sample configurations.

\begin{figure}[htbp]
\centering
    \begin{subfigure}[t]{0.48\textwidth}
        \centering
        \includegraphics[width=\linewidth]{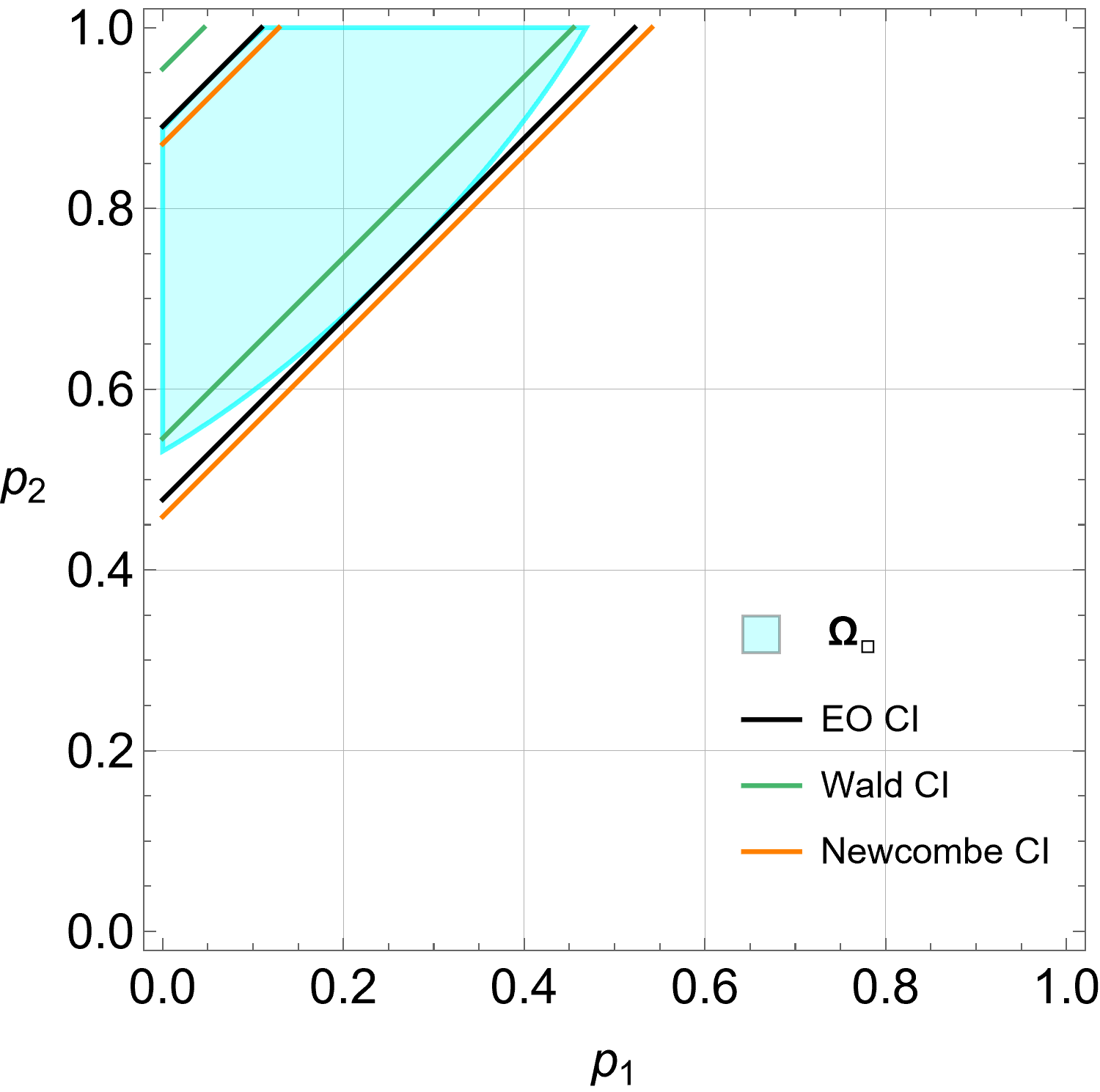}
        \caption{$\hp_1=0.15,\ \hp_2=0.90$ \ ($\hp=-0.75$)}
        \label{subfig:ECICwn075}
    \end{subfigure}
    \hfill
    \begin{subfigure}[t]{0.48\textwidth}
        \centering
        \includegraphics[width=\linewidth]{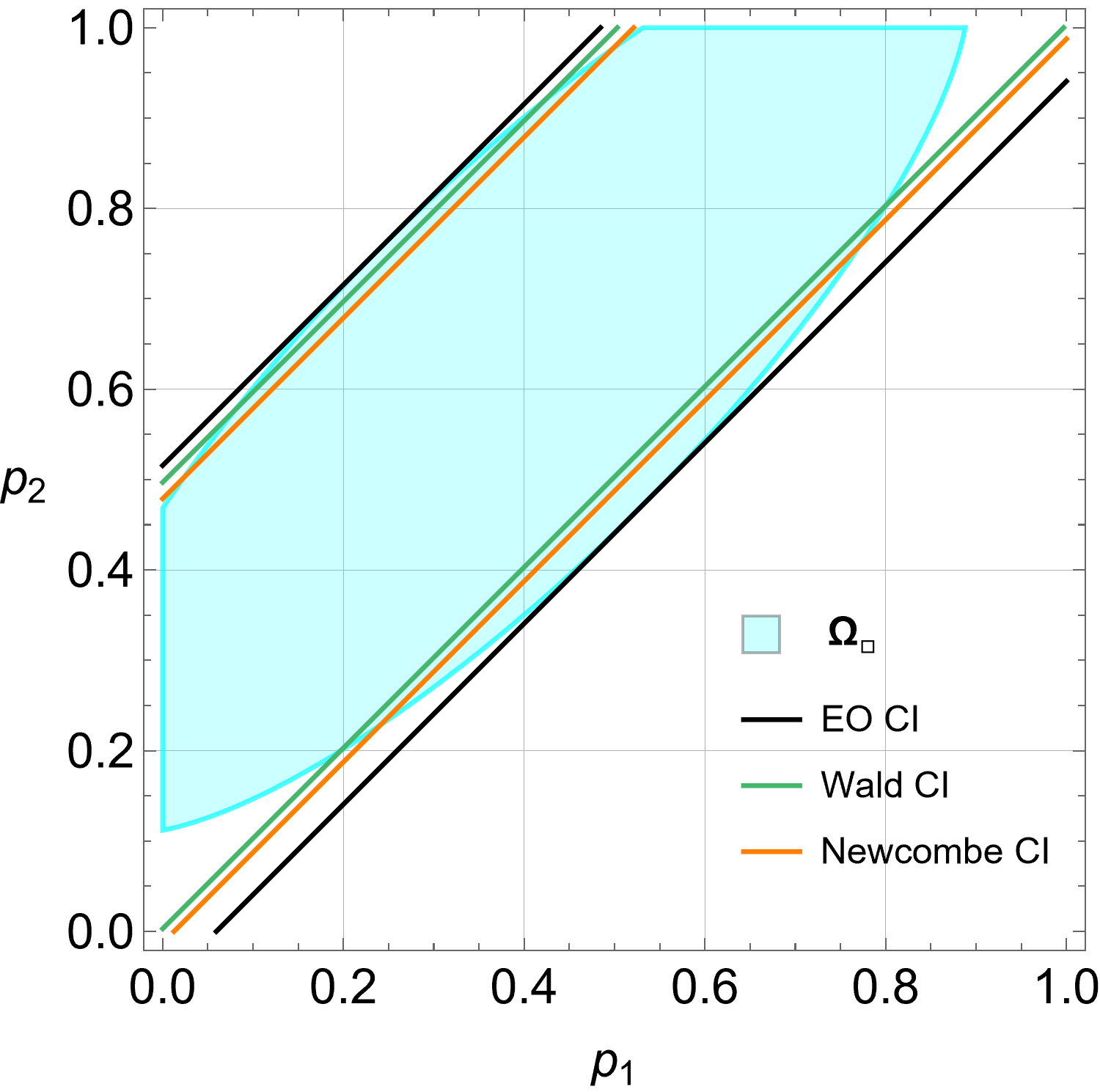}
        \caption{$\hp_1=0.65,\ \hp_2=0.90$ \ ($\hp=-0.25$)}
        \label{subfig:ECICwn025}
    \end{subfigure}

    \vspace{0.8\baselineskip}

    \begin{subfigure}[t]{0.48\textwidth}
        \centering
        \includegraphics[width=\linewidth]{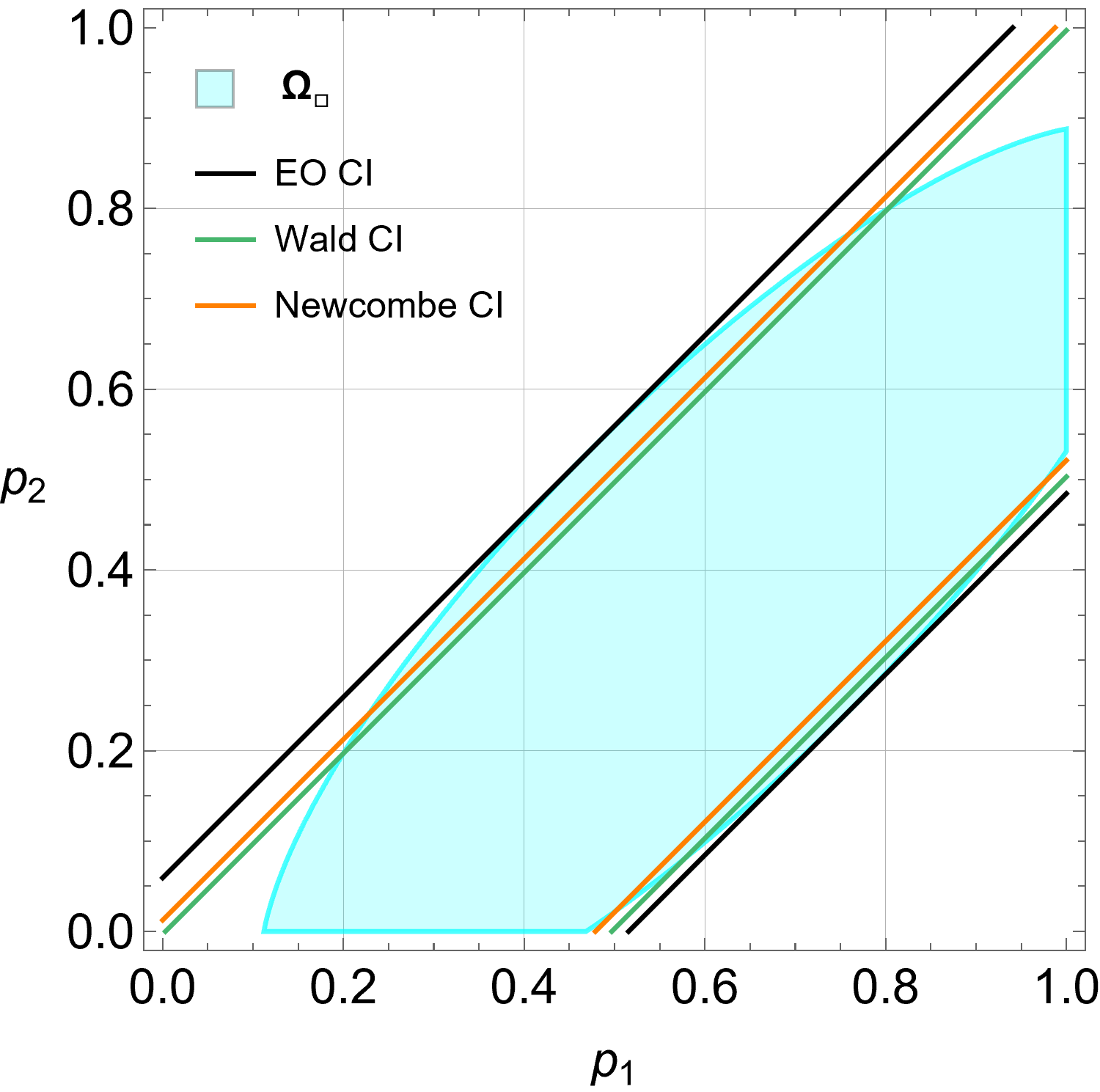}
        \caption{$\hp_1=0.90,\ \hp_2=0.65$ \ ($\hp=0.25$)}
        \label{subfig:ECICwp025}
    \end{subfigure}
    \hfill
    \begin{subfigure}[t]{0.48\textwidth}
        \centering
        \includegraphics[width=\linewidth]{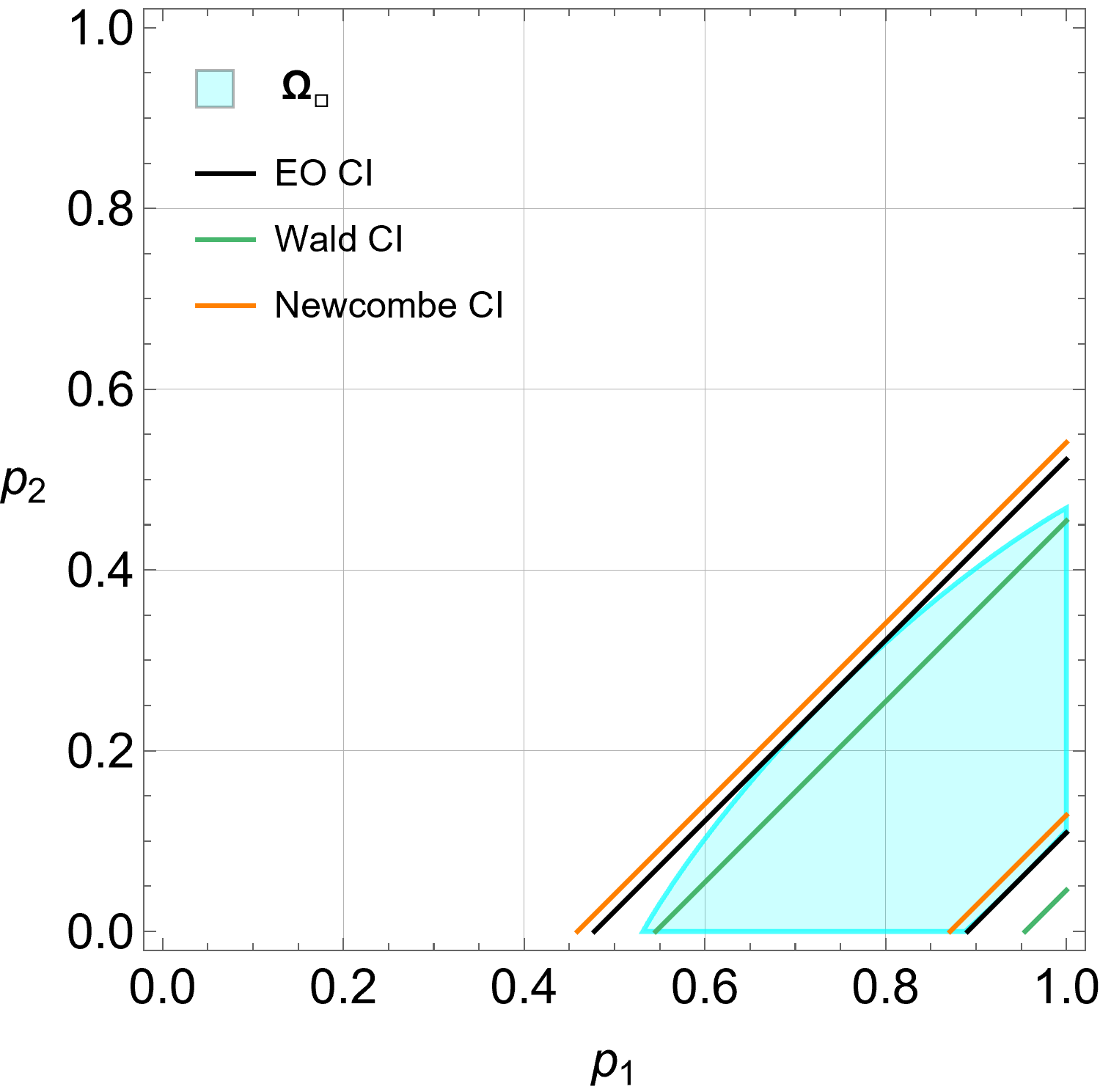}
        \caption{$\hp_1=0.90,\ \hp_2=0.15$ \ ($\hp=0.75$)}
        \label{subfig:ECICwp075}
    \end{subfigure}
\caption[The three intervals in relation to the set $\Sx$]{The \EO{}, Wald and
Newcombe bounds superimposed on the set $\Sx$, for four sample configurations with
$n_1=n_2=20$ and $\a=5\%$. The \EO{} bounds are the two lines that support
$\Sx$ exactly, by construction; the other two may cut across it. For
$\hp=0.25$ the \EO{} interval is $\left[-0.0594,\,0.5155\right]$ and for
$\hp=0.75$ it is $\left[0.4778,\,0.8908\right]$.}
\label{fig:ECICw}
\end{figure}

That the \EO{} bounds support $\Sx$ exactly, while the Wald and Newcombe bounds
may cut across it, is a restatement of the definition \eqref{eq:mainproblemCI} and
carries no evidential weight of its own: a line that cuts the ellipse excludes
some parameter pairs compatible with the data at the nominal level, but it may
nonetheless deliver coverage closer to the nominal level for the difference, which
is the quantity of interest. We now settle that question directly.

\subsection{Attained coverage}
\label{subsec:coverage}

We begin with the Wald interval, for which an exact analytic account is available.

By Proposition~\ref{prop:WaldBias} the plug-in variance underlying
\eqref{eq:WaldCI} is downward biased, exactly and for every parameter pair, with
$\mathbb{E}\bigl[\widehat{\Var}\left[\hp\right]\bigr]=\left(1-1/n\right)\Var\left[\hp\right]$
in the balanced design. Two consequences follow without approximation. First, by
Jensen's inequality applied to the concave square root,
\begin{equation}
\label{eq:Waldwidth}
\begin{gathered}
\mathbb{E}\left[z\sqrt{\widehat{\Var}\left[\hp\right]}\right]
\ \le\ z\sqrt{\mathbb{E}\bigl[\widehat{\Var}\left[\hp\right]\bigr]}
\ =\ z\sqrt{\Var\left[\hp\right]}\,\sqrt{1-b}
\ <\ z\sqrt{\Var\left[\hp\right]},
\\
b=\frac{\sum_{i}p_i\left(1-p_i\right)/n_i^{2}}{\Var\left[\hp\right]},
\end{gathered}
\end{equation}
so the Wald half-width is on average strictly smaller than the half-width
$z\sqrt{\Var\left[\hp\right]}$ that would attain the nominal level under the
normal approximation. Second, a first-order expansion of the coverage about that
ideal half-width gives the deficit $z\varphi(z)\,b$, where $\varphi$ is the
standard normal density; in the balanced case this is $z\varphi(z)/n$, equal to
$0.0057$ at $n=20$ and $0.0025$ at $n=45$ for $\a=5\%$. We record the expansion
only as an indication of the mechanism: it accounts for part but not all of the
deficit actually observed, the remainder arising from the variability of
$\widehat{\Var}$, from skewness, and from the lattice structure of the binomial.
The definitive statement is the enumerated coverage, to which we now turn.

Table~\ref{tab:coveragepoints} reports $C_M\left(p_1,p_2\right)$ from
\eqref{eq:exactcoverage} for the three procedures at $n_1=n_2=45$ and $\a=5\%$,
together with the inflation factor $\mathcal{R}$ of \eqref{eq:inflation} and the
normal-approximation prediction $2\Phi\left(z\mathcal{R}\right)-1$ for the \EO{}
interval.

\renewcommand{\arraystretch}{1.25}
\begin{table}[htbp]
\centering
\caption[Exact coverage of the three intervals]{Exact coverage
\eqref{eq:exactcoverage} of the \EO{}, Wald and Newcombe intervals at selected
parameter pairs, for $n_1=n_2=45$ and a nominal level of $95\%$. The column
$\mathcal{R}$ is the variance inflation factor \eqref{eq:inflation} and
``\EO{} (approx.)'' is the closed-form prediction \eqref{eq:coverageEO}.}
\label{tab:coveragepoints}
\begin{tabular}{|c|c|c|c|c|c|}
\hline
$\left(p_1,p_2\right)$ & $\mathcal{R}$ & \EO{} (approx.) & \EO{} (exact) & Wald & Newcombe\\
\hline
$\left(0.5,0.5\right)$ & $1.000$ & $0.9500$ & $0.9554$ & $0.9517$ & $0.9549$\\
$\left(0.7,0.3\right)$ & $1.000$ & $0.9500$ & $0.9503$ & $0.9433$ & $0.9498$\\
$\left(0.9,0.1\right)$ & $1.000$ & $0.9500$ & $0.9499$ & $0.9445$ & $0.9322$\\
\hline
$\left(0.4,0.4\right)$ & $1.021$ & $0.9545$ & $0.9596$ & $0.9424$ & $0.9514$\\
$\left(0.5,0.2\right)$ & $1.053$ & $0.9611$ & $0.9645$ & $0.9457$ & $0.9483$\\
$\left(0.3,0.3\right)$ & $1.091$ & $0.9675$ & $0.9716$ & $0.9452$ & $0.9495$\\
\hline
$\left(0.2,0.2\right)$ & $1.250$ & $0.9857$ & $0.9879$ & $0.9458$ & $0.9553$\\
$\left(0.8,0.8\right)$ & $1.250$ & $0.9857$ & $0.9879$ & $0.9458$ & $0.9553$\\
$\left(0.1,0.1\right)$ & $1.667$ & $0.9989$ & $0.9990$ & $0.9408$ & $0.9664$\\
\hline
\end{tabular}
\end{table}
\renewcommand{\arraystretch}{1}

The first block of Table~\ref{tab:coveragepoints} consists of pairs lying on the
zero-error line \eqref{eq:valleyline}, here $p_1+p_2=1$; the second of pairs
moderately off it; the third of pairs with both proportions extreme in the same
direction. The agreement between the third and fourth columns is close throughout,
confirming that \eqref{eq:coverageEO} captures the behaviour of the \EO{} interval
to within the lattice effects that the normal approximation necessarily omits.

Table~\ref{tab:coveragesummary} aggregates the same computation over a grid of
parameter pairs, which is the appropriate summary when no single configuration is
of special interest.

\renewcommand{\arraystretch}{1.25}
\begin{table}[htbp]
\centering
\caption[Summary of exact coverage over a parameter grid]{Summary of the exact
coverage \eqref{eq:exactcoverage} of the three procedures over a grid of parameter
pairs ---\ $p_1,p_2\in[0.25,0.75]$ in steps of $0.05$ for $n_1=n_2=20$, and $p_1,p_2\in[0.12,0.88]$ in steps of $0.04$ for $n_1=n_2=45$ --- at a nominal level of $95\%$. ``RMSE'' is the root mean squared deviation
of the attained coverage from the nominal level, and the last column gives the
proportion of grid points at which each procedure is the closest of the three to
nominal.}
\label{tab:coveragesummary}
\begin{tabular}{|l|l|c|c|c|c|c|}
\hline
Design & Method & Mean & Minimum & Maximum & RMSE & Closest\\
\hline
\multirow{3}{*}{$n_1=n_2=20$}
 & \EO{}      & $0.9627$ & $0.9335$ & $0.9832$ & $0.0181$ & $0.10$\\
 & Wald       & $0.9343$ & $0.9192$ & $0.9460$ & $0.0174$ & $0.00$\\
 & Newcombe   & $0.9464$ & $0.9295$ & $0.9575$ & $0.0055$ & $0.90$\\
\hline
\multirow{3}{*}{$n_1=n_2=45$}
 & \EO{}      & $0.9620$ & $0.9438$ & $0.9978$ & $0.0175$ & $0.20$\\
 & Wald       & $0.9431$ & $0.9161$ & $0.9512$ & $0.0080$ & $0.12$\\
 & Newcombe   & $0.9489$ & $0.9337$ & $0.9634$ & $0.0043$ & $0.68$\\
\hline
\end{tabular}
\end{table}
\renewcommand{\arraystretch}{1}

Three findings emerge, and we state them in the order of their reliability.

\emph{The Wald interval is too narrow, and this is now established rather than
asserted.} It falls below the nominal level at essentially every parameter pair
examined --- at $99\%$ of grid points for $n_1=n_2=45$ and at every grid point for
$n_1=n_2=20$ --- with a minimum attained coverage of $0.9161$ against a nominal
$0.95$. The finding is consistent with, and quantitatively anticipated by, the
exact variance identity \eqref{eq:Waldbias} and the width inequality
\eqref{eq:Waldwidth}. It should be added that the deficit is not uniform: at
$n_1=n_2=20$ and $p_1=p_2=0.05$ the Wald interval attains $0.9883$, over-covering
substantially, and this oscillation with the lattice is characteristic of the
method. The correct description is therefore that the Wald interval is
systematically but erratically too narrow, and that its erratic component is as
much an objection to it as its systematic one.

\emph{The \EO{} interval delivers its guarantee.} Its attained coverage is at or
above the nominal level at the great majority of parameter pairs, and the excess
behaves exactly as \eqref{eq:errorformula} predicts: negligible on the line
\eqref{eq:valleyline}, moderate nearby, and large when both proportions are
extreme in the same direction. It never produces an inadmissible or degenerate
interval. Where it does fall below nominal --- at $17\%$ of grid points for
$n_1=n_2=45$, the lowest value being $0.9438$ --- the shortfall is small and
attributable to the discreteness of the binomial rather than to the construction,
whose guarantee, by Proposition~\ref{prop:structural}$(e)$, is exact under the
normal approximation and not under the exact lattice distribution.

\emph{The Newcombe interval is the most accurate of the three in the interior.} It
has the smallest root mean squared deviation from the nominal level in every
design examined, and it is the closest of the three procedures to nominal at
between $68\%$ and $90\%$ of the grid points. This should be said plainly. The
\EO{} interval does not dominate it, and no claim of general superiority over
Newcombe's method is warranted by these computations.

\subsection{Where the \EO{} interval is inferior}
\label{subsec:limitations}

The candid corollary of Section~\ref{subsec:coverage} is that the conservatism
established in Section~\ref{sec:error_analysis} is a real cost, and that in a
substantial part of the parameter space it is the dominant one.

When both true proportions are small, or both large, the \EO{} interval is
markedly too wide. At $n_1=n_2=45$ and $\a=5\%$ its attained coverage is $0.988$
at $p_1=p_2=0.2$ and $0.999$ at $p_1=p_2=0.1$, against $0.946$ and $0.941$ for
Wald and $0.955$ and $0.966$ for Newcombe. In this region the \EO{} interval is
the least accurate of the three by the criterion of closeness to nominal, and it
is so by a wide margin. This is precisely the behaviour visible in the corners of
Figure~\ref{fig:Error-n1en2}, and it is not a numerical artefact: by
\eqref{eq:errorformula} it is the exact consequence of the inflation factor
$\mathcal{R}$ diverging as $V\left(p_1,p_2\right)\to0$ while $V^{\ast}$ does not.

Nor does the situation improve with more data. By
Corollary~\ref{cor:ErrorProperties}$(d)$ the factor $\mathcal{R}$ is invariant
under proportional scaling of the sample sizes, so the over-coverage persists in
the large-sample limit: for $p_1=p_2=q$ in a balanced design the \EO{} interval
converges to a procedure with coverage $2\Phi\bigl(z/\sqrt{4q(1-q)}\bigr)-1$,
which exceeds the nominal level for every $q\ne\tfrac12$ and does so by an amount
that no increase in $n$ will reduce. The \EO{} interval is a valid conservative
procedure; it is not an asymptotically efficient one. Any suggestion that it
dominates its competitors uniformly, or that its greater width relative to Wald or
Newcombe is by itself evidence that those intervals are too narrow, would be
inconsistent with the analysis of Section~\ref{sec:error_analysis} and is
disclaimed here.

\subsection{Practical guidance}
\label{subsec:guidance}

Taken together, the structural results of Section~\ref{subsec:structural} and the
coverage computations of Section~\ref{subsec:coverage} support a differentiated
recommendation rather than a ranking.

\begin{remark}
\label{rm:whentouse}
The \EO{} interval is the appropriate choice when a coverage guarantee is worth
more than length. This is the situation in confirmatory and regulatory settings,
in non-inferiority and equivalence testing where the consequence of an
under-covering interval is asymmetric, and wherever an interval must be reported
that cannot be permitted to exclude the truth more often than advertised. It is
also the natural choice when admissibility matters --- when bounds outside
$[-1,1]$ or a zero-width interval would be unacceptable --- and when the analyst
wishes to know in advance, from \eqref{eq:errorformula}, how much conservatism is
being purchased at the parameter values thought plausible.

The Newcombe interval remains preferable when the object is an interval whose
coverage tracks the nominal level as closely as possible and no guarantee is
required; the computations above show it to be the most accurate of the three in
the interior of the parameter space.

The Wald interval is not recommended for either purpose. Its variance estimator is
biased downward by the exact amount \eqref{eq:Waldbias}, its coverage falls below
the nominal level at essentially every parameter pair examined, and it can return
bounds outside the parameter space or of zero length.

Finally, the \EO{} interval should be used with caution when both proportions are
expected to be extreme in the same direction, since by
Section~\ref{subsec:limitations} its over-coverage is then severe and does not
diminish with sample size. In that regime the inflation factor $\mathcal{R}$ is
computable in advance from \eqref{eq:inflation} and provides a direct diagnostic:
values of $\mathcal{R}$ appreciably above unity signal that the interval will be
conservative by the amount \eqref{eq:errorformula}, and the analyst may decide
accordingly.
\end{remark}

The broader lesson of this comparison concerns the framework rather than the
verdict. Lemma~\ref{lem:profile} showed that the \EO{} interval is the score
interval with the nuisance direction eliminated by maximisation, whereas the
Newcombe and Miettinen--Nurminen constructions eliminate it by combination and by
constrained estimation respectively \citep{Newcombe1998,Miettinen1985}, and the
Wald interval by plug-in substitution. The three methods therefore differ in
exactly one design decision, and their comparative behaviour --- guaranteed
conservatism, close tracking of the nominal level, and systematic under-coverage
respectively --- is traceable to that decision alone. Seen in this light the
\EO{} interval is not a competitor to be ranked against the others but the extreme
point of a family, the member that trades length for an unconditional guarantee.
That framing, and the exact accounting of the trade that
Section~\ref{sec:error_analysis} makes possible, is what the present construction
contributes; we turn in Section~\ref{sec:conclusions} to what it suggests for
further work.

\section{Conclusions}
\label{sec:conclusions}

This paper extends Wilson's score interval to two independent populations by
retaining his variational characterisation and enlarging its feasible set. The normal approximation for the difference of two proportions defines an ellipse
in the plane, and the \EOlong{} (\EO) interval is the range of $p_1-p_2$ over the
portion of that ellipse lying in the unit square. Solving the programme in closed form gives the six-case description
of Theorem~\ref{thm:mainCI}, and the geometry proves unexpectedly economical: the locus traced by the centre of the ellipse, the line containing its two
extreme points in the direction $(1,-1)$, the set of least-favourable parameter
pairs, and the locus of exact coverage all lie on one and the same line
\eqref{eq:centreline}, and the threshold separating the six cases is the range of
the difference along the portion of that line inside the parameter space.

The conceptual centre of the construction is Lemma~\ref{lem:profile}: the \EO{}
interval is the set of hypothesised differences satisfying Wilson's inequality
with the variance replaced by its largest value consistent with the hypothesis.
The method therefore differs from the Wald and Newcombe intervals in exactly one
design decision --- how the nuisance direction is removed --- and its behaviour
follows from that decision alone. Two consequences are immediate. Under the
normal approximation its conservatism is guaranteed, since the least-favourable
variance dominates the true one; and it is measurable, the coverage excess being
given exactly by \eqref{eq:errorformula}. That excess vanishes on the line \eqref{eq:valleyline},
is bounded above by $\a$, and, by Corollary~\ref{cor:ErrorProperties}$(d)$, is
invariant under proportional scaling of the sample sizes.

The comparison of Section~\ref{sec:comparison} rests on attained coverage rather
than length, computed exactly from \eqref{eq:exactcoverage}. It establishes three
things. The Wald interval is too narrow: its variance estimator is biased downward
by the exact amount \eqref{eq:Waldbias}, and it falls below the nominal level at
essentially every parameter pair examined, with a minimum of $0.9161$ against a
nominal $0.95$. The \EO{} interval delivers its guarantee, never producing
inadmissible or degenerate bounds, and its excess coverage behaves precisely as
\eqref{eq:errorformula} predicts. And the Newcombe interval tracks the nominal
level most closely in the interior of the parameter space, so the \EO{} interval
does not dominate it. The \EO{} interval is best understood not as a competitor to
be ranked but as the guarantee-first member of a family: it purchases assurance
with length, and the price is known in advance. It is correspondingly suited to
confirmatory and regulatory settings, and correspondingly ill-suited to parameter
regions in which both proportions are extreme, where its over-coverage is severe
and does not diminish with sample size.

Three extensions suggest themselves. The profiling device is not specific to the
difference: any linear functional of $(p_1,p_2)$ may be handled by the same
support-function argument, and non-linear functionals such as the relative risk or
the odds ratio require only that the corresponding profiled variance be computed.
The construction generalises to $k>2$ populations, where the feasible set becomes
an ellipsoid and the profiled variance a maximisation over an affine slice of the
unit cube. Finally, the conservatism identified here is a consequence of
maximising over the nuisance direction; intermediate procedures, in which the
maximisation is restricted to a data-dependent subset of that direction, would
trade part of the guarantee for a shorter interval and merit study.

\backmatter

\bmhead{Acknowledgements}
The author thanks the referee for a careful and constructive report.

\section*{Statements and Declarations}
\textbf{Competing interests.} The author has no competing interests to declare
that are relevant to the content of this article.

\textbf{Funding.} No funding was received for conducting this study.

\textbf{Data availability.} No datasets were generated or analysed during the
current study. All numerical results reported in the paper are obtained by
deterministic evaluation of the closed-form expressions and by exact
enumeration of the binomial support; the code used is available from the author
on request.

\begin{appendices}
\section{Proofs}
\renewcommand{\thesubsection}{\thesection.\roman{subsection}}

Throughout this appendix we use the equivalent representation
\eqref{eq:Omega_variance_form} of the set $\Sz$, namely
\begin{equation*}
(p_1,p_2)\in\Sz
\iff
\bigl(\hp-p_1+p_2\bigr)^2\le z^2V(p_1,p_2),
\qquad
V(p_1,p_2)=\ddfrac{p_1(1-p_1)}{n_1}+\ddfrac{p_2(1-p_2)}{n_2},
\end{equation*}
together with the symmetries \eqref{eq:mirror} and \eqref{eq:pointreflection}, the
first of which permits us to assume $\hp\ge0$ whenever convenient. We write
$S=n_1+n_2$, $P=n_1n_2$ and $N=S+z^2$, and we abbreviate
$\nmin=\min(n_1,n_2)$ and $\nmax=\max(n_1,n_2)$.

\subsection{A univariate root-location lemma}
\label{app:rootlocation}

The following elementary fact is used repeatedly, here and in the appendices to
Sections~\ref{sec:the_EO_CI} and \ref{sec:error_analysis}. For $\hpi\in[0,1]$
it is due to \citet{ONeill2021}; the extension to negative $\hpi$ is
required because the bounds of Definition~\ref{def:CIboundsLU} are evaluated at
the observed difference $\hp$, which ranges over $[-1,1]$.

\begin{lemma}
\label{lem:rootlocation}
Let $\nu>0$, $z>0$ and let $\hpi\le1$ be real. Put
\begin{equation*}
q(t)=\left(1+\frac{z^{2}}{\nu}\right)t^{2}-\left(2\hpi+\frac{z^{2}}{\nu}\right)t+\hpi^{2},
\qquad
W(\hpi,\nu)=\left\{t\in\mathbb{R}:\ \left(\hpi-t\right)^{2}\le\ddfrac{z^{2}t(1-t)}{\nu}\right\}
=\left\{t:q(t)\le0\right\}.
\end{equation*}
Then:
\begin{itemize}
\item[$(a)$] $q$ has real roots if and only if
$z^{2}+4\nu\hpi\left(1-\hpi\right)\ge0$, and in that case both roots lie in
$[0,1]$, so that $W(\hpi,\nu)$ is a closed subinterval of $[0,1]$ whose
endpoints are obtained from \eqref{eq:WilsonCIOnePop} on replacing $n$ by $\nu$
and $\hp$ by $\hpi$.
\item[$(b)$] If moreover $\hpi\in[0,1]$, the discriminant is at least
$z^{4}/\nu^{2}>0$, so the roots are distinct and the relative interior of
$W(\hpi,\nu)$ is a non-empty open subinterval of $(0,1)$.
\item[$(c)$] If $0<\nu'\le\nu$ then $W(\hpi,\nu)\subseteq W(\hpi,\nu')$.
\end{itemize}
\end{lemma}
\begin{proof}
$(a)$ The discriminant of $q$ is
\begin{equation*}
\left(2\hpi+\frac{z^{2}}{\nu}\right)^{2}-4\left(1+\frac{z^{2}}{\nu}\right)\hpi^{2}
=\frac{4z^{2}\hpi\left(1-\hpi\right)}{\nu}+\frac{z^{4}}{\nu^{2}}
=\frac{z^{2}}{\nu^{2}}\Bigl(z^{2}+4\nu\hpi\left(1-\hpi\right)\Bigr),
\end{equation*}
which gives the stated criterion. Suppose the roots are real. The leading
coefficient is positive, and for every real $\hpi$ one has the two identities
\begin{equation*}
q(0)=\hpi^{2}\ge0,
\qquad
q(1)=1+\frac{z^{2}}{\nu}-2\hpi-\frac{z^{2}}{\nu}+\hpi^{2}=\left(1-\hpi\right)^{2}\ge0 .
\end{equation*}
The vertex is $v=\bigl(2\nu\hpi+z^{2}\bigr)\big/\bigl(2(\nu+z^{2})\bigr)$. If
$\hpi\ge0$ then $v\ge0$ trivially; if $\hpi<0$ then reality of the roots
forces $z^{2}\ge4\nu\left|\hpi\right|\left(1+\left|\hpi\right|\right)>2\nu\left|\hpi\right|$,
whence $v>0$. Also $v\le1$, since this is equivalent to
$2\nu\left(\hpi-1\right)\le z^{2}$, which holds because $\hpi\le1$. A convex
parabola that is non-negative at both endpoints of $[0,1]$ and attains its minimum
in $[0,1]$ has both of its real roots in $[0,1]$.

$(b)$ For $\hpi\in[0,1]$ we have $\hpi\left(1-\hpi\right)\ge0$, so the
discriminant is at least $z^{4}/\nu^{2}>0$ and the roots are distinct. The open
interval between them is contained in $(0,1)$: a root can equal $0$ only if
$q(0)=0$, that is $\hpi=0$, in which case the roots are $0$ and
$z^{2}/(\nu+z^{2})<1$; and a root can equal $1$ only if $q(1)=0$, that is
$\hpi=1$, in which case the roots are $\nu/(\nu+z^{2})>0$ and $1$.

$(c)$ Any $t\in W(\hpi,\nu)$ lies in $[0,1]$ by $(a)$, so $t(1-t)\ge0$ and
$z^{2}t(1-t)/\nu'\ge z^{2}t(1-t)/\nu\ge\left(\hpi-t\right)^{2}$.
\end{proof}

\subsection{Proof of Lemma \ref{lem:diagonal_sections}}
\label{app:diagonal_sections}
\begin{proof}
Both parts are exact algebraic reductions rather than approximations.

\emph{Part $(a)$.} Substituting $p_2=1-p_1$ into \eqref{eq:Omega_variance_form},
the left-hand side becomes
\begin{equation*}
\bigl(\hp-p_1+(1-p_1)\bigr)^{2}=\bigl(1+\hp-2p_1\bigr)^{2}
=4\left(\frac{1+\hp}{2}-p_1\right)^{2},
\end{equation*}
while the variance term becomes
\begin{equation*}
V\bigl(p_1,1-p_1\bigr)
=\frac{p_1(1-p_1)}{n_1}+\frac{(1-p_1)\bigl(1-(1-p_1)\bigr)}{n_2}
=p_1(1-p_1)\,\frac{n_1+n_2}{n_1n_2}.
\end{equation*}
Hence membership of $\Sz$ along this line is equivalent to
\begin{equation*}
4\left(\hpi-p_1\right)^{2}\le z^2p_1(1-p_1)\frac{n_1+n_2}{n_1n_2}
\iff
\left(\hpi-p_1\right)^{2}\le \frac{z^2p_1(1-p_1)}{n_{\ast}},
\end{equation*}
with $\hpi=(1+\hp)/2$ and $n_{\ast}=4n_1n_2/(n_1+n_2)$ as in \eqref{eq:nstar}.
Since $\hp\in[-1,1]$ we have $\hpi\in[0,1]$, so the section is exactly
$W(\hpi,n_{\ast})$ and Lemma~\ref{lem:rootlocation}$(a)$--$(b)$ applies: the
section is a closed subinterval of $[0,1]$ of strictly positive length, with
endpoints obtained from \eqref{eq:WilsonCIOnePop} on replacing $(n,\hp)$ by
$(n_{\ast},\hpi)$, and with relative interior a non-empty open subinterval of
$(0,1)$. Expanding those endpoints and clearing denominators returns, after
elementary simplification, the pair
\begin{gather*}
r_1=\frac{(n_1+n_2)z^2+4n_1n_2(1+\hp)-z\sqrt{(n_1+n_2)\bigl((n_1+n_2)z^2+4n_1n_2(1-\hp^2)\bigr)}}{2\bigl((n_1+n_2)z^2+4n_1n_2\bigr)},
\\
r_2=\frac{(n_1+n_2)z^2+4n_1n_2(1+\hp)+z\sqrt{(n_1+n_2)\bigl((n_1+n_2)z^2+4n_1n_2(1-\hp^2)\bigr)}}{2\bigl((n_1+n_2)z^2+4n_1n_2\bigr)},
\end{gather*}
so that $0\le r_1<r_2\le1$.

The same conclusion may be read directly from the coefficient form. Substituting
$p_2=1-p_1$ into \eqref{eq:ellipseprob} gives $q(p_1)\le0$ with
\begin{equation*}
q(p_1)=\left(A-B+C\right)p_1^{2}+\left(B-2C+D-E\right)p_1+\left(C+E+F\right),
\end{equation*}
and, writing $K=(n_1+n_2)z^{2}/(n_1n_2)$, one computes
$A-B+C=4+K$, $B-2C+D-E=-\bigl(4(1+\hp)+K\bigr)$ and $C+E+F=(1+\hp)^{2}$, so that
\begin{equation*}
q(p_1)=\left(4+K\right)p_1^2-\left(4(1+\hp)+K\right)p_1+(1+\hp)^2 .
\end{equation*}
Its leading coefficient is positive, its discriminant equals
\begin{equation*}
\left(4(1+\hp)+K\right)^{2}-4\left(4+K\right)(1+\hp)^{2}
=K\Bigl(K+4\left(1-\hp^{2}\right)\Bigr)>0,
\end{equation*}
its vertex $\bigl(4(1+\hp)+K\bigr)\big/\bigl(2(4+K)\bigr)$ lies in $[0,1]$, and it
satisfies the pleasing pair of identities
\begin{equation*}
q(0)=(1+\hp)^2\ge0,
\qquad
q(1)=(4+K)-\bigl(4(1+\hp)+K\bigr)+(1+\hp)^{2}=(1-\hp)^2\ge0 .
\end{equation*}
A convex parabola that is non-negative at both endpoints of $[0,1]$ and attains
its minimum inside $[0,1]$ has both of its real roots in $[0,1]$.

\emph{Part $(b)$.} Substituting $p_1=p_2=s$ into the coefficient form
\eqref{eq:ellipseprob} gives
\begin{equation*}
(A+B+C)s^2+(D+E)s+F
=\frac{(n_1+n_2)z^2}{n_1n_2}\,s^2-\frac{(n_1+n_2)z^2}{n_1n_2}\,s+\hp^2
=\frac{(n_1+n_2)z^2}{n_1n_2}\,s(s-1)+\hp^2,
\end{equation*}
where we used $A+B+C=z^2\left(\frac{1}{n_1}+\frac{1}{n_2}\right)$ and
$D+E=-z^2\left(\frac{1}{n_1}+\frac{1}{n_2}\right)$. Membership of $\Sz$ is
therefore equivalent to
\begin{equation*}
s(1-s)\ \ge\ \frac{\hp^2n_1n_2}{z^2(n_1+n_2)}=\frac{\hp^2n_{\ast}}{4z^2}.
\end{equation*}
The left-hand side is a concave parabola with maximum $1/4$ attained at $s=1/2$,
so the section is non-empty if and only if $\hp^2n_{\ast}/(4z^2)\le1/4$, that is,
if and only if $\left|\hp\right|\le z/\sqrt{n_{\ast}}$; and in that case it is the
closed interval centred at $1/2$ with half-length
$\sqrt{1/4-\hp^{2}n_{\ast}/(4z^{2})}\le1/2$, which is contained in $[0,1]$.
Finally $n_{\ast}\le n_1+n_2$, with equality if and only if $n_1=n_2$, since
$4n_1n_2\le(n_1+n_2)^2$ by the arithmetic--geometric mean inequality.
\end{proof}

\subsection{Proof of Property \ref{prty:convex_compact}}
\label{app:convex_compact}
\begin{proof}
Write the quadratic form of \eqref{eq:ellipseprob} as
$Q(x)=x^{\top}A_2x+b^{\top}x+F$ with
\begin{equation*}
A_2=\begin{pmatrix} A & B/2\\ B/2 & C\end{pmatrix}
=\begin{pmatrix} 1+z^2/n_1 & -1\\ -1 & 1+z^2/n_2\end{pmatrix},
\qquad
b=(D,E)^{\top},
\end{equation*}
so that $\Sz=\{x:Q(x)\le0\}$. The leading entry satisfies $A=1+z^{2}/n_1>0$, and
\begin{equation*}
\det A_2=AC-\frac{B^{2}}{4}
=\left(1+\frac{z^{2}}{n_1}\right)\left(1+\frac{z^{2}}{n_2}\right)-1
=\frac{z^{2}\left(n_1+n_2+z^{2}\right)}{n_1n_2}>0,
\end{equation*}
so $A_2$ is positive definite by Sylvester's criterion. Consequently $Q$ is a
strictly convex function, and $\Sz$, being a sublevel set of a convex function, is
convex; moreover $Q(x)\to+\infty$ as $\left\|x\right\|\to\infty$, so $\Sz$ is
bounded, and it is closed because $Q$ is continuous. Hence $\Sz$ is compact and
convex. Completing the square gives
$Q(x)=\left(x-x_c\right)^{\top}A_2\left(x-x_c\right)-k$ with
$x_c=-\tfrac12A_2^{-1}b$ the centre \eqref{eq:centre} and
$k=x_c^{\top}A_2x_c-F$. By Property~\ref{prty:nonempty_proper}$(b)$
the set $\Sx\subseteq\Sz$ has non-empty interior, so $k>0$; hence $\Sz$ has
non-empty interior and is symmetric about $x_c$, the defining inequality
depending on $x-x_c$ only through a quadratic form. (Equivalently
$k=-\Delta/\det A_2>0$ by Property~\ref{prty:determinant}, whose proof uses only
the positive-definiteness established in the preceding paragraph.)

Finally, $[0,1]^{2}$ is compact and convex, so $\Sx=\Sz\cap[0,1]^{2}$ is an
intersection of compact convex sets and is therefore compact and convex; being
convex, it is connected.
\end{proof}

\subsection{Proof of Property \ref{prty:nonempty_proper}}
\label{app:Sx_non-empty_propersubset}
\begin{proof}
\emph{Part $(a)$: non-emptiness.} The observed pair $(\hp_1,\hp_2)$ lies in
$[0,1]^2$ and satisfies
\begin{equation*}
\bigl(\hp-\hp_1+\hp_2\bigr)^2-z^2V(\hp_1,\hp_2)=-z^2V(\hp_1,\hp_2)\le0,
\end{equation*}
since $V\ge0$ on $[0,1]^2$; hence $(\hp_1,\hp_2)\in\Sx$.

\emph{Part $(b)$: non-empty interior.} The witness just exhibited is insufficient
for this purpose, because $V(\hp_1,\hp_2)=0$ whenever $\hp_1,\hp_2\in\{0,1\}$, in
which case the defining inequality holds with equality and $(\hp_1,\hp_2)$ is a
boundary point of $\Sz$. We argue instead along the opposite diagonal. By
Lemma~\ref{lem:diagonal_sections}$(a)$ the section of $\Sz$ by the line
$p_1+p_2=1$ is $W\bigl((1+\hp)/2,n_{\ast}\bigr)$, whose relative interior is a
non-empty open subinterval $\mathcal{J}\subset(0,1)$ by
Lemma~\ref{lem:rootlocation}$(b)$. For $p_1\in\mathcal{J}$ the point
$(p_1,1-p_1)$ satisfies the defining inequality of $\Sz$ \emph{strictly} and lies
in the open square $(0,1)^{2}$; since the defining function is continuous, a whole
neighbourhood of that point lies in $\Sz$, and shrinking the neighbourhood so that
it remains inside $(0,1)^{2}$ places it inside $\Sx$. Hence $\Sx$ has non-empty
interior, and it contains an open segment of the opposite diagonal.

\emph{Part $(c)$: proper inclusion.} By the symmetry \eqref{eq:mirror} we may
assume $\hp\ge0$. Observe first that any point at which the defining inequality of
$\Sz$ holds strictly is an interior point of $\Sz$; consequently, if such a point
lies on the boundary of the unit square, then $\Sz$ contains points outside
$[0,1]^2$ and the inclusion $\Sx\subseteq\Sz$ is strict.

Since $\hp\in[0,1]$, Lemma~\ref{lem:rootlocation}$(b)$ applied with
$\hpi=\hp$ and $\nu=\nmax$ furnishes a non-empty open interval
$\mathcal{I}\subset(0,1)$ with
\begin{equation}
\label{eq:strict_wilson}
(\hp-s)^2<\frac{z^2s(1-s)}{\nmax}
\qquad\text{for every }s\in\mathcal{I}.
\end{equation}
Fix any $s\in\mathcal{I}$ and consider the two points
\begin{equation*}
P_{\mathrm{b}}=(s,0)\in(0,1)\times\{0\},
\qquad
P_{\mathrm{r}}=(1,1-s)\in\{1\}\times(0,1),
\end{equation*}
lying on the bottom and right edges of the unit square respectively. Both have
coordinate difference equal to $s$, so the left-hand side of
\eqref{eq:Omega_variance_form} equals $(\hp-s)^2$ at each of them, while
\begin{equation*}
V(P_{\mathrm{b}})=\frac{s(1-s)}{n_1}\ge\frac{s(1-s)}{\nmax},
\qquad
V(P_{\mathrm{r}})=\frac{(1-s)s}{n_2}\ge\frac{s(1-s)}{\nmax},
\end{equation*}
both inequalities using $s(1-s)>0$ and $n_i\le\nmax$. Combining with
\eqref{eq:strict_wilson} gives $(\hp-s)^2<z^2V(P_{\mathrm{b}})$ and
$(\hp-s)^2<z^2V(P_{\mathrm{r}})$, so both points belong to the interior of $\Sz$.
Since they lie on the boundary of $[0,1]^2$, every neighbourhood of them meets the
complement of the unit square; hence $\Sz$ contains points with $p_2<0$ and points
with $p_1>1$, none of which belongs to $\Sx$. Therefore $\Sx\neq\Sz$, and since
$\Sx=\Sz\cap[0,1]^2\subseteq\Sz$ by construction, $\Sx\subset\Sz$.
\end{proof}

\subsection{Proof of Corollary \ref{cor:Sz_geometry}}
\label{app:Sz_geometry}
\begin{proof}
The first bullet is Lemma~\ref{lem:diagonal_sections}$(a)$, whose section lies
entirely in $[0,1]^2$ and hence in $\Sx$.

For the second and third bullets, retain the notation of the proof of
Property~\ref{prty:nonempty_proper} and assume first $\hp\ge0$.
Interiority of $P_{\mathrm{b}}=(s,0)$ supplies an $\varepsilon>0$ with
$(s,-\varepsilon')\in\Sz$ for all $0<\varepsilon'<\varepsilon$; since
$s\in(0,1)$, such points lie in $(0,1)\times(-\infty,0)$, which establishes the
second bullet in its first alternative. Interiority of $P_{\mathrm{r}}=(1,1-s)$
supplies analogously points $(1+\varepsilon',1-s)\in\Sz$ with $1-s\in(0,1)$, lying
in $(1,\infty)\times(0,1)$, which establishes the third bullet in its second
alternative. For $\hp\le0$ the symmetry \eqref{eq:mirror} maps these two
protrusions across the main diagonal, sending them into $(-\infty,0)\times(0,1)$
and $(0,1)\times(1,\infty)$ respectively, which are the remaining alternatives. In
either case the disjunctions asserted in the two bullets hold, with the
attributions stated there.

The fourth bullet requires only the observation that membership of $\Sz$ forces
$V(p_1,p_2)\ge0$, because the left-hand side of \eqref{eq:Omega_variance_form} is
a square. Now $t\mapsto t(1-t)$ is strictly negative for $t\notin[0,1]$ and zero
for $t\in\{0,1\}$. Hence if both $p_1\notin[0,1]$ and $p_2\notin[0,1]$ then
$V(p_1,p_2)<0$; and if exactly one coordinate lies outside $[0,1]$ while the other
lies in $\{0,1\}$, then again $V(p_1,p_2)<0$. Therefore every point of $\Sz$
either lies in $[0,1]^2$, or has exactly one coordinate outside $[0,1]$ and the
other strictly inside $(0,1)$, which is precisely the asserted containment
\begin{equation*}
\Sz \subseteq [0,1]^2 \cup
\big( (0,1)\times(-\infty,0) \big) \cup
\big( (-\infty,0)\times(0,1) \big) \cup
\big( (0,1)\times(1,\infty) \big) \cup
\big( (1,\infty)\times(0,1) \big).
\end{equation*}
The inclusion is automatically proper, since $\Sz$ is bounded by
Property~\ref{prty:convex_compact} while the right-hand side is not.
\end{proof}

\subsection{Proof of Property \ref{prty:determinant}}
\label{app:determinant}
\begin{proof}
Write the conic \eqref{eq:ellipseprob} in matrix form, with
\begin{equation*}
M=\begin{pmatrix}
A & B/2 & D/2\\
B/2 & C & E/2\\
D/2 & E/2 & F
\end{pmatrix},
\qquad
\Delta=\det M .
\end{equation*}
It is convenient to set $u=\dfrac{z^2}{2n_1}$ and $v=\dfrac{z^2}{2n_2}$, so that
\begin{equation*}
A=1+2u,\quad C=1+2v,\quad \frac{B}{2}=-1,\quad
\frac{D}{2}=-(\hp+u),\quad \frac{E}{2}=\hp-v,\quad F=\hp^2 .
\end{equation*}
Expanding along the first row,
\begin{align*}
\Delta&=A\left(CF-\tfrac{E^2}{4}\right)-\tfrac{B}{2}\left(\tfrac{B}{2}F-\tfrac{E}{2}\tfrac{D}{2}\right)+\tfrac{D}{2}\left(\tfrac{B}{2}\tfrac{E}{2}-C\tfrac{D}{2}\right)
\\
&=(1+2u)\,v\bigl(2\hp^2+2\hp-v\bigr)+\bigl(\hp(u-v)-uv\bigr)-(\hp+u)\bigl(u+v+2v\hp+2uv\bigr),
\end{align*}
where we used $CF-E^{2}/4=(1+2v)\hp^{2}-(\hp-v)^{2}=v\left(2\hp^{2}+2\hp-v\right)$,
$\tfrac{B}{2}F-\tfrac{E}{2}\tfrac{D}{2}=-\hp^{2}+(\hp-v)(\hp+u)=\hp(u-v)-uv$ and
$\tfrac{B}{2}\tfrac{E}{2}-C\tfrac{D}{2}=u+v+2v\hp+2uv$. Expanding each product,
the terms linear in $\hp$ cancel identically, the terms in $\hp^{2}$ collapse to
$4uv\hp^{2}$, and the constant terms sum to
$-(u+v)^{2}-2uv(u+v)$, leaving
\begin{equation*}
\Delta=4uv\hp^2-(u+v)\bigl(u+v+2uv\bigr).
\end{equation*}
Substituting $u+v=\dfrac{z^2(n_1+n_2)}{2n_1n_2}$, $uv=\dfrac{z^4}{4n_1n_2}$ and
$u+v+2uv=\dfrac{z^2(n_1+n_2+z^2)}{2n_1n_2}$ yields
\begin{equation*}
\Delta=\frac{z^4\hp^2}{n_1n_2}-\frac{z^4(n_1+n_2)(n_1+n_2+z^2)}{4n_1^2n_2^2}
=-\frac{z^4\bigl((n_1+n_2)(n_1+n_2+z^2)-4n_1n_2\hp^2\bigr)}{4n_1^2n_2^2},
\end{equation*}
which is \eqref{eq:Delta}. Since $\hp^2\le1$ and
$(n_1+n_2)^2-4n_1n_2=(n_1-n_2)^2$,
\begin{equation*}
\Delta\le-\frac{z^4\bigl((n_1+n_2)(n_1+n_2+z^2)-4n_1n_2\bigr)}{4n_1^2n_2^2}
=-\frac{z^4\bigl((n_1-n_2)^2+(n_1+n_2)z^2\bigr)}{4n_1^2n_2^2}<0,
\end{equation*}
the final inequality being strict because $z>0$ and $n_1+n_2>0$. Thus
$\Delta\ne0$ and the conic is non-degenerate, and $B^2-4AC<0$ was established in
Section~\ref{subsec:bivariate}, so it is of elliptic type. That it is a real
rather than an imaginary ellipse follows without appeal to any classification
criterion: by Property~\ref{prty:nonempty_proper}$(a)$ the solid
region $\Sz$ is non-empty, while the quadratic form is positive definite so that
$Q(x)\to+\infty$ at infinity; by continuity the level set $\{Q=0\}$ is therefore
non-empty. Hence $\Sz$ is a real, non-degenerate ellipse together with its
interior.
\end{proof}

\subsection{Proof of Properties \ref{prty:centre}, \ref{prty:axes} and \ref{prty:rotation}, and of Corollary \ref{cor:rotation_range}}
\label{app:conic_reductions}
\begin{proof}
All three properties follow from the standard reduction of a non-degenerate
central conic, applied to the matrix $M$ of Appendix~\ref{app:determinant} and to
its leading $2\times2$ block $A_2$ of Appendix~\ref{app:convex_compact}. Two
auxiliary quantities are needed. First,
$\det A_2=z^{2}N/P>0$, consistent with the discriminant
$B^2-4AC=-4\det A_2$ of Section~\ref{subsec:bivariate}. Second, the eigenvalues of
$A_2$ are
\begin{equation*}
\lambda_{\pm}=\frac{A+C}{2}\pm\sqrt{\frac{(A-C)^2}{4}+\frac{B^{2}}{4}}
=\frac{2n_1n_2+z^2(n_1+n_2)\pm\sqrt{4n_1^2n_2^2+(n_1-n_2)^2z^4}}{2n_1n_2},
\end{equation*}
where we used $B^{2}/4=1$,
$A+C=\bigl(2n_1n_2+z^2(n_1+n_2)\bigr)/(n_1n_2)$ and
$A-C=z^2(n_2-n_1)/(n_1n_2)$.

\emph{Centre.} The centre solves
$\nabla\bigl(Ap_1^2+Bp_1p_2+Cp_2^2+Dp_1+Ep_2+F\bigr)=0$, that is the linear system
$2Ap_1+Bp_2+D=0$, $Bp_1+2Cp_2+E=0$, whence by Cramer's rule
\begin{equation*}
p_1^c=\frac{BE-2CD}{4AC-B^2},
\qquad
p_2^c=\frac{BD-2AE}{4AC-B^2}.
\end{equation*}
Substituting the coefficients and using $4AC-B^2=4z^2N/P$,
\begin{align*}
BE-2CD&=2z^2\left(\frac{1}{n_1}+\frac{1}{n_2}\right)+\frac{4\hp z^2}{n_2}+\frac{2z^4}{n_1n_2}
=\frac{2z^2N+4\hp z^2 n_1}{n_1n_2},
\\
BD-2AE&=2z^2\left(\frac{1}{n_1}+\frac{1}{n_2}\right)+\frac{2z^4}{n_1n_2}-\frac{4\hp z^2}{n_1}
=\frac{2z^2N-4\hp z^2 n_2}{n_1n_2},
\end{align*}
and dividing by $4AC-B^2$ gives
$p_1^c=\tfrac12+n_1\hp/N$ and $p_2^c=\tfrac12-n_2\hp/N$, which is
\eqref{eq:centre}. In particular
\begin{equation*}
n_2p_1^c+n_1p_2^c=\frac{n_2}{2}+\frac{n_1n_2\hp}{N}+\frac{n_1}{2}-\frac{n_1n_2\hp}{N}=\frac{n_1+n_2}{2},
\end{equation*}
so the centre lies on the line $\ell$ of \eqref{eq:centreline} for every value of
$\hp$, as asserted in Remark~\ref{rem:centreline}.

\emph{Axis lengths.} For a real non-degenerate central conic the semi-axis
associated with the eigenvalue $\lambda$ of $A_2$ equals
$\sqrt{-\Delta/\left(\lambda\det A_2\right)}$, the major semi-axis corresponding
to the smaller eigenvalue $\lambda_{-}$ and the minor to $\lambda_{+}$; the axis
lengths are twice these. Writing
$X=(n_1+n_2)^2-4n_1n_2\hp^2+(n_1+n_2)z^2$, so that $-\Delta=z^4X/(4P^{2})$ by
\eqref{eq:Delta}, and
$D_{\pm}=2n_1n_2+z^2(n_1+n_2)\pm\sqrt{4n_1^2n_2^2+(n_1-n_2)^2z^4}$, so that
$\lambda_{\pm}=D_{\pm}/(2P)$, we obtain
\begin{equation*}
\mathrm{Major}=2\sqrt{\frac{-\Delta}{\lambda_{-}\det A_2}}
=2\sqrt{\frac{z^4X}{4P^{2}}\cdot\frac{2P}{D_{-}}\cdot\frac{P}{z^{2}N}}
=2\sqrt{\frac{z^2X}{2ND_{-}}} .
\end{equation*}
Rationalising with the identity
\begin{equation*}
D_{-}D_{+}=\bigl(2P+z^2S\bigr)^2-\bigl(4P^{2}+(n_1-n_2)^2z^4\bigr)
=4Pz^{2}S+z^{4}\bigl(S^{2}-(n_1-n_2)^{2}\bigr)=4Pz^2N
\end{equation*}
gives
\begin{equation*}
\mathrm{Major}=2\sqrt{\frac{z^{2}XD_{+}}{2N\cdot4Pz^{2}N}}=\frac{\sqrt{X\,D_{+}}}{\sqrt{2P}\,N},
\end{equation*}
which is the stated expression; interchanging the roles of $\lambda_{-}$ and
$\lambda_{+}$ yields $\mathrm{Minor}=\sqrt{X\,D_{-}}\big/\bigl(\sqrt{2P}\,N\bigr)$.
Setting $n_2=n_1$ gives $D_{+}=4n_1^2+2n_1z^2$ and $D_{-}=2n_1z^2$, and the two
expressions reduce to those displayed after Property~\ref{prty:axes}; in that case
$\mathrm{Major}^2=\bigl(4n_1(1-\hp^2)+2z^2\bigr)/(2n_1+z^2)\le2$, with equality if
and only if $\hp=0$, and $\mathrm{Minor}/\mathrm{Major}=z/\sqrt{2n_1+z^2}$. The
bound by $\sqrt2$ uses $n_1=n_2$ and does not extend to unequal allocation.

\emph{Rotation angle.} The principal axes of $A_2$ make an angle $\theta$ with the
horizontal satisfying $\tan(2\theta)=B/(A-C)$, that is,
\begin{equation*}
\tan(2\theta)=\frac{-2}{z^2(n_2-n_1)/(n_1n_2)}=\frac{2n_1n_2}{(n_1-n_2)z^2}.
\end{equation*}
The major axis is the eigendirection of $\lambda_{-}$; solving
$(A-\lambda_{-})x+(B/2)y=0$ with $B/2=-1$ gives $y=(A-\lambda_{-})x$, so
$\tan\theta=A-\lambda_{-}$, and
\begin{equation*}
A-\lambda_{-}=\frac{A-C}{2}+\sqrt{\frac{(A-C)^2}{4}+1}>0,
\end{equation*}
whence $\theta\in(0,\pi/2)$ in every case. Setting $d=z^2(n_1-n_2)/(n_1n_2)$, so
that $A-C=-d$, this slope equals $-d/2+\sqrt{d^2/4+1}$, the positive root of
$t^2+dt-1=0$. That root is consistent with the value of $\tan(2\theta)$ obtained
above, since $t^{2}+dt-1=0$ gives $1-t^{2}=dt$ and hence
$\tan(2\theta)=2t/(1-t^{2})=2/d$. When $n_1>n_2$ we have $d>0$, hence
$\tan(2\theta)>0$; as $2\theta\in(0,\pi)$ this forces $2\theta\in(0,\pi/2)$, so
the principal branch of the arctangent applies and $\theta\in(0,\pi/4)$. When
$n_2>n_1$ we have $d<0$, hence $\tan(2\theta)<0$ and $2\theta\in(\pi/2,\pi)$, so
$\pi$ must be added to the principal value and $\theta\in(\pi/4,\pi/2)$. When
$n_1=n_2$ we have $A=C$, so $2\theta=\pi/2$ and $\theta=\pi/4$. This proves both
Property~\ref{prty:rotation} and Corollary~\ref{cor:rotation_range}.
\end{proof}

\subsection{Proof of Proposition \ref{prop:SzCenterDomain} and of Corollary \ref{cor:centre_in_square}}
\label{app:SzCenterDomain}
\begin{proof}
Set
\begin{equation*}
x=p_1^c-\tfrac12=\frac{n_1\hp}{N},
\qquad
y=p_2^c-\tfrac12=-\frac{n_2\hp}{N},
\end{equation*}
as given by \eqref{eq:centre}, and put $w_1=n_1/N$, $w_2=n_2/N$, so that
$x=w_1\hp$ and $y=-w_2\hp$ with
\begin{equation}
\label{eq:weights}
w_1>0,\qquad w_2>0,\qquad w_1+w_2=\frac{n_1+n_2}{N}<1,
\end{equation}
the last inequality being strict because $z^2>0$.

\emph{Necessity.} First, $xy=-w_1w_2\hp^2\le0$, with equality if and only if
$\hp=0$, in which case $x=y=0$ and the centre is
$\left(\tfrac12,\tfrac12\right)$; thus
$\bigl(p_1^c-\tfrac12\bigr)\bigl(p_2^c-\tfrac12\bigr)<0$ whenever $\hp\ne0$.
Second, since $x$ and $y$ then have strictly opposite signs,
\begin{equation*}
\left|x-y\right|=\left|x\right|+\left|y\right|=(w_1+w_2)\left|\hp\right|<1,
\end{equation*}
using \eqref{eq:weights} and $\left|\hp\right|\le1$; equivalently, and more
informatively,
\begin{equation*}
p_1^c-p_2^c=x-y=(w_1+w_2)\hp=\hp\,\frac{n_1+n_2}{n_1+n_2+z^2},
\end{equation*}
which is the shrinkage identity \eqref{eq:centre_difference}. Hence
$(p_1^c,p_2^c)\in T\cup\{(\tfrac12,\tfrac12)\}$: the conditions $xy<0$ and
$\left|x-y\right|<1$ describe exactly the interiors of the two triangles named in
the statement, since in centred coordinates $\{x>0>y,\ x-y<1\}$ is the open
triangle with vertices $(0,0)$, $(1,0)$, $(0,-1)$ and $\{x<0<y,\ y-x<1\}$ its
reflection.

\emph{Sharpness over the positive reals.} Conversely, let $(x_0,y_0)$ satisfy
$x_0>0>y_0$ and $x_0-y_0<1$, and put $a=x_0>0$, $b=-y_0>0$, so that $a+b<1$.
Choose $\hp=1$, which is attainable, fix any $N>0$, and set
\begin{equation*}
n_1=aN,\qquad n_2=bN,\qquad z^2=N-n_1-n_2=N(1-a-b)>0 .
\end{equation*}
These are admissible: $n_1,n_2>0$, and $z=z_{\a/2}$ ranges over $(0,\infty)$ as
$\a$ ranges over $(0,1)$. With this choice $w_1=n_1/N=a$ and $w_2=b$, whence
$x=w_1\hp=x_0$ and $y=-w_2\hp=y_0$. The case $x_0<0<y_0$ follows by the symmetry
\eqref{eq:mirror}, and $\left(\tfrac12,\tfrac12\right)$ is attained at $\hp=0$.
Hence every point of $T\cup\{(\tfrac12,\tfrac12)\}$ is attained and the
containment cannot be improved.

\emph{Sharpness over the positive integers.} If $n_1$ and $n_2$ are required to be
integers, the construction above succeeds precisely when $N$ may be chosen so that
$aN$ and $bN$ are both positive integers, which is possible if and only if $a/b$
is rational: writing $a/b=p/q$ in lowest terms and $\lambda=a/p$, the choice
$N=k/\lambda$ for a positive integer $k$ gives $n_1=kp$ and $n_2=kq$. The set of
points of $T$ with rational coordinate ratio is dense in $T$, so the attainable
centres form a dense subset.

\emph{Proof of Corollary \ref{cor:centre_in_square}.} The centre of a non-degenerate ellipse lies in the region it bounds, so
$(p_1^c,p_2^c)\in\Sz$ always, and therefore $(p_1^c,p_2^c)\in\Sx$ if and only if
$(p_1^c,p_2^c)\in[0,1]^2$. It remains to
intersect $T\cup\{(\tfrac12,\tfrac12)\}$ with the unit square. Consider
$T^{+}=\left\{(X,Y):X>\tfrac12>Y,\ X-Y<1\right\}$, where $X=x+\tfrac12$ and
$Y=y+\tfrac12$ are the original coordinates. Its intersection with $[0,1]^2$
requires $X\in\left(\tfrac12,1\right]$ and $Y\in\left[0,\tfrac12\right)$; for such
points $X-Y\le1$ holds automatically, with equality only at $(1,0)$, which the
strict inequality $X-Y<1$ excludes. Hence
\begin{equation*}
T^{+}\cap[0,1]^2=\left(\left(\tfrac12,1\right]\times\left[0,\tfrac12\right)\right)\setminus\left\{(1,0)\right\},
\end{equation*}
and by reflection
\begin{equation*}
T^{-}\cap[0,1]^2=\left(\left[0,\tfrac12\right)\times\left(\tfrac12,1\right]\right)\setminus\left\{(0,1)\right\}.
\end{equation*}
Adjoining the point $\left(\tfrac12,\tfrac12\right)$, which lies in the square,
yields the region asserted in Corollary~\ref{cor:centre_in_square}. Since both
punctured rectangles lie in the upper-left and lower-right quadrants of the unit
square, the centre can never lie in the closed lower-left or upper-right quadrant
other than at the midpoint.
\end{proof}

In the remaining proofs we retain the abbreviations $S$, $P$ and $N$ introduced
at the head of this appendix, and we write
\begin{equation*}
\rho=\frac{z}{2}\sqrt{\frac{\Xi}{PS}},
\qquad
\tau=\hp+\rho,
\qquad
\sigma=\hp-\rho,
\end{equation*}
with $\Xi$ as in \eqref{eq:Xi}. Since $\Xi>0$ by \eqref{eq:Delta_bound} and
$z>0$, we have $\rho>0$ and $\sigma<\tau$. The root-location facts used below are
those of Lemma~\ref{lem:rootlocation}.
\subsection{Proof of Lemma \ref{lem:profile}}
\label{app:profile}
\begin{proof}
\emph{Closed form for $V^{\ast}$.} Fix $p\in[-1,1]$ and parametrise the feasible
segment by $p_1=s$, $p_2=s-p$, so that $s$ ranges over
$I_p=\left[\max(0,p),\min(1,1+p)\right]$, a non-empty compact interval. On it,
\begin{equation*}
f(s)=\frac{s(1-s)}{n_1}+\frac{(s-p)\left(1-s+p\right)}{n_2}
\end{equation*}
is a strictly concave quadratic with
\begin{equation*}
f'(s)=\left(\frac{1}{n_1}+\frac{1}{n_2}\right)\left(1-2s\right)+\frac{2p}{n_2},
\qquad\text{so}\qquad
s^{\ast}=\frac{1}{2}+\frac{n_1p}{n_1+n_2}
\end{equation*}
is the unconstrained maximiser, with corresponding second coordinate
$s^{\ast}-p=\tfrac12-n_2p/(n_1+n_2)$. Thus the maximising pair moves from the
centre of the square along the direction $(n_1,-n_2)$, exactly as does the centre
of the ellipse in \eqref{eq:centre}. Now $s^{\ast}\le1$ if and only if
$p\le S/(2n_1)$, and $s^{\ast}\ge p$ if and only if $p\le S/(2n_2)$; combining
these two with their mirror images for $p<0$, the unconstrained maximiser lies in
$I_p$ if and only if $\left|p\right|\le S/(2\nmax)=\hp_{\square}$. In that case
\begin{equation*}
V^{\ast}(p)=f(s^{\ast})
=\frac{1}{n_1}\left(\frac14-\frac{n_1^{2}p^{2}}{S^{2}}\right)
+\frac{1}{n_2}\left(\frac14-\frac{n_2^{2}p^{2}}{S^{2}}\right)
=\frac{S}{4P}-\frac{p^{2}}{S},
\end{equation*}
which is the first branch of \eqref{eq:Vstar}. If instead $p>\hp_{\square}$, then
by concavity the constrained maximum is attained at the endpoint of $I_p$ nearest
$s^{\ast}$. If $n_1=\nmax$ this endpoint is $s=1$, giving the pair $(1,1-p)$ and
$V^{\ast}(p)=p(1-p)/n_2=p(1-p)/\nmin$; if $n_2=\nmax$ it is $s=p$, giving the pair
$(p,0)$ and $V^{\ast}(p)=p(1-p)/n_1=p(1-p)/\nmin$. Either way the value is
$p(1-p)/\nmin$, and the maximiser lies on the edge of the unit square associated
with the larger sample size. The case $p<-\hp_{\square}$ follows from the symmetry
$V(p_1,p_2)=V(1-p_1,1-p_2)$, which maps the line $p_1-p_2=p$ onto
$p_1-p_2=-p$ and shows that $V^{\ast}$ is even.

\emph{Regularity.} $V$ is jointly concave and the constraint set in
\eqref{eq:Vstardef} is convex and varies affinely with $p$, so the partial maximum
$V^{\ast}$ is concave; positivity on $(-1,1)$ and $V^{\ast}(\pm1)=0$ are read off
\eqref{eq:Vstar}. Continuity and differentiability need only be checked at
$p=\pm\hp_{\square}$, and by evenness only at $p=\hp_{\square}$. The two branches
agree there, since
\begin{equation*}
\frac{S}{4P}-\frac{\hp_{\square}^{2}}{S}
=\frac{S}{4}\left(\frac{1}{P}-\frac{1}{\nmax^{2}}\right)
=\frac{S\left(\nmax-\nmin\right)}{4\nmax^{2}\nmin}
=\frac{\hp_{\square}\left(1-\hp_{\square}\right)}{\nmin},
\end{equation*}
using $1-\hp_{\square}=\left(\nmax-\nmin\right)/(2\nmax)$. Their derivatives also
agree there: the first branch gives $-2\hp_{\square}/S=-1/\nmax$, and the second
gives $\left(1-2\hp_{\square}\right)/\nmin=-1/\nmax$. Hence $V^{\ast}\in C^{1}$.

\emph{Representation \eqref{eq:EOprofile}.} The set $\Sx$ is non-empty by
Property~\ref{prty:nonempty_proper}, and it is compact and convex,
hence connected; the map $g(p_1,p_2)=p_1-p_2$ is continuous and linear, so
$g(\Sx)$ is a non-empty compact interval, necessarily
$\left[\min_{\Sx}g,\max_{\Sx}g\right]$. Therefore $p$ belongs to the \EO{}
interval if and only if there exists $(p_1,p_2)\in\Sx$ with $p_1-p_2=p$, that is,
if and only if there exists $(p_1,p_2)\in[0,1]^{2}$ with $p_1-p_2=p$ and
$\left(\hp-p\right)^{2}\le z^{2}V(p_1,p_2)$. Since the left-hand side does not
depend on which such pair is chosen, this holds if and only if
$\left(\hp-p\right)^{2}\le z^{2}V^{\ast}(p)$, which is \eqref{eq:EOprofile}.
Finally, $p\mapsto\left(\hp-p\right)^{2}-z^{2}V^{\ast}(p)$ is convex, being the
sum of a convex quadratic and $-z^{2}$ times a concave function, so its sublevel
set is indeed an interval, in agreement with the left-hand side.
\end{proof}

\subsection{Proof of Proposition \ref{prop:WilsonIdentification}}
\label{app:WilsonIdentification}
\begin{proof}
\emph{The pairs $\left[\Lone,\Uone\right]$ and $\left[\Ltwo,\Utwo\right]$.}
Substituting the second branch of \eqref{eq:Vstar} into
$\left(\hp-p\right)^{2}\le z^{2}V^{\ast}(p)$ gives, for $p>0$,
$\left(\hp-p\right)^{2}\le z^{2}p(1-p)/\nmin$, whose solution set is
$W\!\left(\hp,\nmin\right)$; its endpoints are, by
Lemma~\ref{lem:rootlocation}$(a)$, exactly $\Lone$ and $\Uone$ of
Definition~\ref{def:CIboundsLU}, and lie in
$[0,1]$. For $p<0$ the same substitution gives
$\left(\hp-p\right)^{2}\le z^{2}\left(-p\right)\left(1+p\right)/\nmin$; writing
$p=-t$ this reads $\left(-\hp-t\right)^{2}\le z^{2}t(1-t)/\nmin$, that is,
$t\in W\!\left(-\hp,\nmin\right)$, so the solution set is
$-W\!\left(-\hp,\nmin\right)\subseteq[-1,0]$, with endpoints
\begin{equation*}
-\,U^{W}\!\left(-\hp,\nmin\right)
=\frac{\nmin\hp-z^{2}/2}{\nmin+z^{2}}-\frac{z\sqrt{z^{2}-4\nmin\hp\left(1+\hp\right)}}{2\left(\nmin+z^{2}\right)}=\Ltwo,
\qquad
-\,L^{W}\!\left(-\hp,\nmin\right)=\Utwo .
\end{equation*}

The edge interpretation is a direct computation. On the edge $p_1=1$ one has
$V(1,p_2)=p_2(1-p_2)/n_2$; writing $t=1-p_2=p_1-p_2$ and using
$p_2(1-p_2)=t(1-t)$, membership of $\Sz$ reads
$\left(\hp-t\right)^{2}\le z^{2}t(1-t)/n_2$, so the range of the objective over
$\Sz\cap\{p_1=1\}$ is $W\!\left(\hp,n_2\right)$. On the edge $p_2=0$ one has
$V(p_1,0)=p_1(1-p_1)/n_1$ and $p_1-p_2=p_1$, giving the range
$W\!\left(\hp,n_1\right)$. By Lemma~\ref{lem:rootlocation}$(c)$, the union of the two
is $W\!\left(\hp,\nmin\right)=\left[\Lone,\Uone\right]$. The edges $p_1=0$ and
$p_2=1$ are treated identically and give
$-W\!\left(-\hp,n_2\right)$ and $-W\!\left(-\hp,n_1\right)$, whose union is
$-W\!\left(-\hp,\nmin\right)=\left[\Ltwo,\Utwo\right]$.

\emph{The pair $\left[\Lthree,\Uthree\right]$.} Substituting the first branch of
\eqref{eq:Vstar} gives
\begin{equation*}
\left(\hp-p\right)^{2}\le z^{2}\left(\frac{S}{4P}-\frac{p^{2}}{S}\right)
\iff
\frac{N}{S}\,p^{2}-2\hp\,p+\hp^{2}-\frac{z^{2}S}{4P}\le0 ,
\end{equation*}
a convex quadratic in $p$ whose roots are
\begin{equation*}
\frac{S}{N}\left[\hp\pm\sqrt{\hp^{2}\left(1-\frac{N}{S}\right)+\frac{Nz^{2}}{4P}}\right]
=\frac{S}{N}\left[\hp\pm\sqrt{\frac{z^{2}}{4PS}\left(SN-4P\hp^{2}\right)}\right]
=\frac{S}{N}\left(\hp\pm\rho\right),
\end{equation*}
where we used $1-N/S=-z^{2}/S$ and $\Xi=SN-4P\hp^{2}$. Since
$\frac{S}{N}\rho=\frac{z}{2N}\sqrt{S\Xi/P}$, these are exactly $\Lthree$ and
$\Uthree$; in particular
\begin{equation}
\label{eq:LU3tau}
\Lthree=\frac{S\sigma}{N},\qquad \Uthree=\frac{S\tau}{N}.
\end{equation}

That $\left[\Lthree,\Uthree\right]$ is also the range of $p_1-p_2$ over the whole
ellipse follows from the support-function computation. By Property~\ref{prty:convex_compact} the set $\Sz$ is
$\left\{x:\left(x-x_c\right)^{\top}A_2\left(x-x_c\right)\le k\right\}$ with
$x_c=(p_1^c,p_2^c)$ as in \eqref{eq:centre}, $A_2$ the leading $2\times2$ block of
Appendix~\ref{app:convex_compact} and $k=-\Delta/\det A_2>0$. With
$c=(1,-1)^{\top}$ and $\det A_2=z^{2}N/P$,
\begin{equation*}
A_2^{-1}c=\frac{z^{2}}{\det A_2}\begin{pmatrix}1/n_2\\ -1/n_1\end{pmatrix}
=\frac{1}{N}\begin{pmatrix}n_1\\ -n_2\end{pmatrix},
\qquad
c^{\top}A_2^{-1}c=\frac{S}{N},
\qquad
k=\frac{z^{2}\Xi}{4PN},
\end{equation*}
so that the extreme values of $c^{\top}x$ are
$c^{\top}x_c\pm\sqrt{k\,c^{\top}A_2^{-1}c}
=\hp\frac{S}{N}\pm\frac{z}{2N}\sqrt{\frac{S\Xi}{P}}$, using the shrinkage identity
\eqref{eq:centre_difference}. These coincide with $\Lthree$ and $\Uthree$, whose
midpoint is therefore $p_1^c-p_2^c$. The extremising points are
\begin{equation}
\label{eq:xpm}
x^{\pm}=x_c\pm\sqrt{\frac{k}{c^{\top}A_2^{-1}c}}\;A_2^{-1}c
=\left(\frac12+\frac{n_1\tau_{\pm}}{N},\ \frac12-\frac{n_2\tau_{\pm}}{N}\right),
\qquad \tau_{+}=\tau,\quad \tau_{-}=\sigma,
\end{equation}
so that they too lie along the direction $(n_1,-n_2)$, in agreement with
Lemma~\ref{lem:profile}.
\end{proof}

\subsection{Proof of Proposition \ref{prop:orderLUCI}}
\label{app:orderLUCI}
\begin{proof}
The inequalities $\Lone\le\Uone$, $\Ltwo\le\Utwo$ and $\Lthree\le\Uthree$ are
immediate from Definition~\ref{def:CIboundsLU}, the radicands being non-negative
by hypothesis.

\emph{The central inequalities $\Utwo\le0\le\Lone$.} By
Proposition~\ref{prop:WilsonIdentification},
$\left[\Lone,\Uone\right]=W\!\left(\hp,\nmin\right)\subseteq[0,1]$ and
$\left[\Ltwo,\Utwo\right]=-W\!\left(-\hp,\nmin\right)\subseteq[-1,0]$, and the
claim follows. It is worth recording an independent verification that does not
appeal to the edge representation. Writing
$f(\hp)=\sqrt{z^{2}+4\nmin\hp(1-\hp)}+\sqrt{z^{2}-4\nmin\hp(1+\hp)}$, one has
$\Lone-\Utwo=\left(2z^{2}-zf(\hp)\right)/\left(2(\nmin+z^{2})\right)$. Putting
$u=4\nmin\hp$ and $v=4\nmin\hp^{2}\ge0$, the two radicands are $z^{2}-v+u$ and
$z^{2}-v-u$, whence
\begin{equation*}
f(\hp)^{2}=2\left(z^{2}-v\right)+2\sqrt{\left(z^{2}-v\right)^{2}-u^{2}}
\le4\left(z^{2}-v\right)\le4z^{2},
\end{equation*}
so $f(\hp)\le2z$, with equality if and only if $u=v=0$, that is $\hp=0$. Hence
$\Lone-\Utwo\ge0$.

\emph{The outer inequalities $\Lthree\le\Ltwo$ and $\Uone\le\Uthree$.} By
Proposition~\ref{prop:WilsonIdentification}, $\Ltwo$ is the minimum of $p_1-p_2$
over the section of $\Sz$ by one of the edges $\{p_1=0\},\{p_2=1\}$, and the
hypothesis that the corresponding radicand is non-negative is precisely the
statement that this section is non-empty. A minimum taken over a non-empty subset
of $\Sz$ cannot fall below the minimum over $\Sz$, which is $\Lthree$; hence
$\Lthree\le\Ltwo$. The same argument applied to maxima, with the section by
$\{p_1=1\}$ or $\{p_2=0\}$, gives $\Uone\le\Uthree$. Chaining the inequalities
yields the stated ordering; if only one outer pair is real, the corresponding
sub-chain is obtained by deleting the missing terms.
\end{proof}

\subsection{Proof of Lemma \ref{lem:switching}}
\label{app:switching}
\begin{proof}
\emph{The switching rule \eqref{eq:switchingrule}.} Write
$\psi(p)=\left(\hp-p\right)^{2}-z^{2}V^{\ast}(p)$, so that by
Lemma~\ref{lem:profile} the \EO{} interval is $\{\psi\le0\}$ and $\psi$ is convex.
Let $\psi_0,\psi_{+},\psi_{-}$ denote the three quadratics obtained by using,
respectively, the first branch of $V^{\ast}$, the second branch on $p>0$ and the
second branch on $p<0$, each extended to all of $\mathbb{R}$; by
Proposition~\ref{prop:WilsonIdentification} their root pairs are
$\left(\Lthree,\Uthree\right)$, $\left(\Lone,\Uone\right)$ and
$\left(\Ltwo,\Utwo\right)$, and $\psi$ agrees with $\psi_0$ on
$\left[-\hp_{\square},\hp_{\square}\right]$, with $\psi_{+}$ on
$\left[\hp_{\square},1\right]$ and with $\psi_{-}$ on
$\left[-1,-\hp_{\square}\right]$.

Suppose first $\left|\Uthree\right|\le\hp_{\square}$. Then $\Uthree$ lies in the
middle branch, $\psi(\Uthree)=\psi_0(\Uthree)=0$, and for
$p\in\left(\Uthree,\hp_{\square}\right]$ we have $\psi=\psi_0>0$; convexity of
$\psi$ then forces $\psi>0$ on $\left(\Uthree,1\right]$, so $U=\Uthree$. Suppose
next $\Uthree>\hp_{\square}$. If also $\Lthree>\hp_{\square}$ then $\psi_0>0$ on
the whole middle branch, so $\{\psi\le0\}$ is contained in the right branch and
$U=\Uone$; if $\Lthree\le\hp_{\square}<\Uthree$ then
$\psi(\hp_{\square})=\psi_0(\hp_{\square})\le0$ and the largest root of $\psi$
lies in the right branch, again $U=\Uone$. Reality of $\Uone$ in both cases
follows because $\psi_{+}$ takes a non-positive value at or beyond
$\hp_{\square}$. Finally, if $\Uthree<-\hp_{\square}$ then
$\Lthree\le\Uthree<-\hp_{\square}$, so $\psi_0>0$ on the middle branch and
$\{\psi\le0\}$ lies wholly in the left branch, whence $U=\Utwo$. The three
statements for $L$ are proved identically, or deduced from the equivariance
$\hp\mapsto-\hp$.

\emph{Translation into conditions on $(\hp,z)$.} By \eqref{eq:LU3tau},
$\Uthree\le\hp_{\square}$ is equivalent to $\tau\le N/(2\nmax)$, and
$\Uthree\ge-\hp_{\square}$ to $\tau\ge-N/(2\nmax)$; symmetrically for $\Lthree$
and $\sigma$. Equivalently, by \eqref{eq:xpm}, these are exactly the conditions
$x^{+}\in[0,1]^{2}$ and $x^{-}\in[0,1]^{2}$, since $x^{\pm}\in[0,1]^{2}$ if and
only if $\left|\tau_{\pm}\right|\le N/(2n_1)$ and
$\left|\tau_{\pm}\right|\le N/(2n_2)$, that is
$\left|\tau_{\pm}\right|\le N/(2\nmax)$.

Set $c_1=\frac{N}{2\nmax}-\hp$ and $c_2=\frac{N}{2\nmax}+\hp$. A direct
computation, obtained by clearing denominators and using $\Xi=SN-4P\hp^{2}$ and
$S+z^{2}=N$, gives the identity
\begin{equation}
\label{eq:keyidentity}
4PS\left(c_1^{2}-\rho^{2}\right)
=\frac{N}{\nmax}\Bigl[\nmin\left(S-2\nmax\hp\right)^{2}-S\left(\nmax-\nmin\right)z^{2}\Bigr],
\end{equation}
and the same identity with $\hp$ replaced by $-\hp$ for $c_2$. Consequently, when
$\nmax>\nmin$,
\begin{equation*}
\rho^{2}\le c_1^{2}\iff z\le T_{-}:=\frac{\sqrt{\nmin}\left|S-2\nmax\hp\right|}{\sqrt{\nmax^{2}-\nmin^{2}}},
\qquad
\rho^{2}\le c_2^{2}\iff z\le T_{+}:=\frac{\sqrt{\nmin}\left|S+2\nmax\hp\right|}{\sqrt{\nmax^{2}-\nmin^{2}}},
\end{equation*}
while when $\nmax=\nmin$ the right-hand side of \eqref{eq:keyidentity} is
$\nmin N\left(S-2\nmax\hp\right)^{2}/\nmax\ge0$, so $\rho^{2}\le c_1^{2}$ always
holds, and symmetrically for $c_2$; this is the content of the conventions
\eqref{eq:Rconventions}, and comparison with \eqref{eq:R1}--\eqref{eq:R2} shows
that $\left\{T_{-},T_{+}\right\}=\left\{\Rone,\Rtwo\right\}$, with $T_{-}=\Rone$
and $T_{+}=\Rtwo$ when $\hp\ge0$ and the assignment reversed when $\hp<0$.

We now dispose of the sign conditions. Since $\rho>0$, the inequality
$\tau\le\frac{N}{2\nmax}$ holds if and only if $c_1\ge0$ and $z\le T_{-}$, and the
inequality $\tau\ge-\frac{N}{2\nmax}$ holds if and only if $c_2\ge0$ or
$z\ge T_{+}$. Three observations complete the translation.

$(\alpha)$ If $\hp\le\hp_{\square}$ then $c_1\ge\frac{S+z^{2}}{2\nmax}-\hp_{\square}=\frac{z^{2}}{2\nmax}>0$,
and if $\hp\ge-\hp_{\square}$ then likewise $c_2>0$. Hence for
$\left|\hp\right|\le\hp_{\square}$ both sign conditions are automatic, and
$x^{+}\in[0,1]^{2}$ if and only if $z\le T_{-}$, while $x^{-}\in[0,1]^{2}$ if and
only if $z\le T_{+}$.

$(\beta)$ If $\hp>\hp_{\square}$, then $x^{+}\notin[0,1]^{2}$ always. Indeed, put
$w=2\nmax\hp-S\in\left(0,\nmax-\nmin\right]$; then $c_1=\left(z^{2}-w\right)/(2\nmax)$
and $T_{-}=\kappa w$ with $\kappa=\sqrt{\nmin/\left(\nmax^{2}-\nmin^{2}\right)}$, so
feasibility would require both $z^{2}\ge w$ and $z\le\kappa w$, hence
$\sqrt{w}\le\kappa w$, that is $w\ge\kappa^{-2}=\left(\nmax^{2}-\nmin^{2}\right)/\nmin
=\left(\nmax-\nmin\right)S/\nmin$. Since $S>\nmin$, this contradicts
$w\le\nmax-\nmin$. Symmetrically, $x^{-}\notin[0,1]^{2}$ whenever
$\hp<-\hp_{\square}$.

$(\gamma)$ If $\hp<-\hp_{\square}$, put $\widetilde w=-\left(S+2\nmax\hp\right)\in\left(0,\nmax-\nmin\right]$,
so $c_2=\left(z^{2}-\widetilde w\right)/(2\nmax)$ and $T_{+}=\kappa\widetilde w$. Since
$\widetilde w\le\nmax-\nmin\le\kappa^{-2}$, we have $\kappa\widetilde w\le\sqrt{\widetilde w}$,
so the disjunction ``$c_2\ge0$ or $z\ge T_{+}$'' reduces to $z\ge T_{+}$.
Consequently, for $\hp<-\hp_{\square}$, $x^{+}\in[0,1]^{2}$ if and only if
$T_{+}\le z\le T_{-}$, that is $\Rone\le z\le\Rtwo$. Symmetrically, for
$\hp>\hp_{\square}$, $x^{-}\in[0,1]^{2}$ if and only if $\Rone\le z\le\Rtwo$.

Collecting $(\alpha)$--$(\gamma)$ and recalling the identification of $T_{\pm}$
with $\Rone,\Rtwo$ according to the sign of $\hp$ gives the translation asserted in the lemma and used in Appendix~\ref{app:mainCI}.

\emph{Reality of the named bound.} In the case $U=\Uone$ we must check that
$z^{2}+4\nmin\hp\left(1-\hp\right)\ge0$; this is automatic for $\hp\ge0$. For
$\hp<0$ the case $U=\Uone$ arises only when $z>T_{-}=\Rtwo$, and with
$u=\left|\hp\right|$ the algebraic identity
\begin{equation}
\label{eq:perfectsquare}
\left(2\nmax u+S\right)^{2}-4u\left(1+u\right)\left(\nmax^{2}-\nmin^{2}\right)
=\left(2\nmin u+S\right)^{2}
\end{equation}
gives
$\Rtwo^{2}-4\nmin u\left(1+u\right)=\nmin\left(2\nmin u+S\right)^{2}/\left(\nmax^{2}-\nmin^{2}\right)>0$,
so $z^{2}>\Rtwo^{2}>4\nmin u(1+u)$ and the radicand is strictly positive. In the
case $U=\Utwo$, which by the above requires $\hp<-\hp_{\square}$, the radicand
$z^{2}-4\nmin\hp\left(1+\hp\right)\ge z^{2}>0$ since $\hp\in[-1,0]$. The
statements for $L$ follow by symmetry, and $\Lthree,\Uthree$ are always real.
\end{proof}

\subsection{Proof of Theorem \ref{thm:mainCI}}
\label{app:mainCI}
\begin{proof}
\emph{Identification of the six cases.} By Lemma~\ref{lem:switching} the interval
is determined by the position of $\Uthree$ and $\Lthree$ relative to
$\pm\hp_{\square}$, and by parts $(\alpha)$--$(\gamma)$ of its proof these
positions translate as follows, where we use $T_{-}=\Rone$, $T_{+}=\Rtwo$ for
$\hp\ge0$ and $T_{-}=\Rtwo$, $T_{+}=\Rone$ for $\hp<0$:
\begin{center}
\begin{tabular}{lll}
\emph{range of $\hp$} & \emph{$U=\Uthree$ iff} & \emph{$L=\Lthree$ iff}\\[0.4ex]
$\hp>\hp_{\square}$ & never & $\Rone\le z\le\Rtwo$\\
$0\le\hp\le\hp_{\square}$ & $z\le\Rone$ & $z\le\Rtwo$\\
$-\hp_{\square}\le\hp<0$ & $z\le\Rtwo$ & $z\le\Rone$\\
$\hp<-\hp_{\square}$ & $\Rone\le z\le\Rtwo$ & never
\end{tabular}
\end{center}
Outside these ranges the corresponding bound is $\Uone$ or $\Utwo$, respectively
$\Lone$ or $\Ltwo$, according to \eqref{eq:switchingrule}. Two further facts are
needed to name the outer bound correctly.

First, for $\hp<-\hp_{\square}$ and $z\le\Rone$ we have $U=\Utwo$ rather than
$\Uone$, and consistently the pair $\left[\Lone,\Uone\right]$ is empty. Indeed,
with $u=\left|\hp\right|>\hp_{\square}$, the companion of
\eqref{eq:perfectsquare} is
\begin{equation*}
\left(2\nmax u-S\right)^{2}-4u\left(1+u\right)\left(\nmax^{2}-\nmin^{2}\right)
=4\nmin^{2}u^{2}-4\left(2\nmax-\nmin\right)Su+S^{2}=:\phi(u),
\end{equation*}
a convex quadratic in $u$ satisfying
\begin{equation*}
\phi\!\left(\hp_{\square}\right)=\frac{S^{2}}{\nmax^{2}}\left(\nmin+3\nmax\right)\left(\nmin-\nmax\right)\le0,
\qquad
\phi(1)=\left(7\nmax+9\nmin\right)\left(\nmin-\nmax\right)\le0 .
\end{equation*}
By convexity $\phi\le0$ on $\left[\hp_{\square},1\right]$, whence
$\Rone^{2}\le4\nmin u\left(1+u\right)$ and therefore
$z^{2}\le4\nmin\left|\hp\right|\left(1+\left|\hp\right|\right)$, so that the
radicand of $\Lone,\Uone$ is non-positive. Symmetrically, for
$\hp>\hp_{\square}$ and $z\le\Rone$ the pair $\left[\Ltwo,\Utwo\right]$ is empty
and $L=\Lone$.

Second, the two outer pairs are never simultaneously empty, since emptiness of
$\left[\Lone,\Uone\right]$ requires $\hp<0$ and emptiness of
$\left[\Ltwo,\Utwo\right]$ requires $\hp>0$; and by \eqref{eq:perfectsquare} both
are non-empty whenever $z\ge\Rtwo$.

Assembling the table entries from the display above now gives exactly the six rows
of Table~\ref{tab:CIresults}. For $\hp>\hp_{\square}$ the three regimes
$0<z\le\Rone$, $\Rone<z<\Rtwo$ and $z\ge\Rtwo$ give
$\left(\Lone,\Uone\right)$, $\left(\Lthree,\Uone\right)$ and
$\left(\Ltwo,\Uone\right)$, that is cases $(1)$, $(5)$ and $(2)$. For
$0\le\hp\le\hp_{\square}$ the regimes $0<z<\Rone$, $\Rone\le z<\Rtwo$ and
$z\ge\Rtwo$ give cases $(6)$, $(5)$ and $(2)$. The ranges $\hp<0$ follow by the
equivariance $\hp\mapsto-\hp$, giving cases $(3)$, $(4)$ and $(2)$ respectively.

\emph{Disjointness and exhaustiveness.} Fix $\hp$. If
$\left|\hp\right|>\hp_{\square}$ then $\Rone>0$ and the three regimes
$\left(0,\Rone\right]$, $\left(\Rone,\Rtwo\right)$, $\left[\Rtwo,\infty\right)$
partition $(0,\infty)$. If $0<\left|\hp\right|<\hp_{\square}$ the partition is
$\left(0,\Rone\right)$, $\left[\Rone,\Rtwo\right)$, $\left[\Rtwo,\infty\right)$,
again exhaustive and disjoint. If $\hp=\pm\hp_{\square}$ then $\Rone=0$, the first
regime is empty, and cases $(5)$ or $(4)$ together with $(2)$ cover $(0,\infty)$;
note that at $\hp=\hp_{\square}$ case $(1)$ requires $\hp>\hp_{\square}$ and case
$(6)$ requires $\hp<\hp_{\square}$, so neither competes. If $\hp=0$ then
$\Rone=\Rtwo$, the middle regime is empty, and cases $(6)$ and $(2)$ cover
$(0,\infty)$. Finally, in the balanced design $\nmax=\nmin$ one has
$\hp_{\square}=1$, so cases $(1)$ and $(3)$ are vacuous, case $(6)$ covers
$-1<\hp<1$ with $\Rone=\infty$, and the conventions \eqref{eq:Rconventions} place
$\hp=1$ in case $(5)$ and $\hp=-1$ in case $(4)$, both with $\Rone=0$ and
$\Rtwo=\infty$. That this assignment is the correct one is verified directly: at
$\hp=1$ and $n_1=n_2=n$ one has $\Xi=2nz^{2}$, hence $\rho=z^{2}/(2n)$ and
$\sigma=1-z^{2}/(2n)$, so by \eqref{eq:xpm}
$x^{-}=\left(\tfrac12+\tfrac{n-z^{2}/2}{2n+z^{2}},\ \tfrac12-\tfrac{n-z^{2}/2}{2n+z^{2}}\right)\in[0,1]^{2}$,
giving $L=\Lthree=\left(2n-z^{2}\right)/\left(2n+z^{2}\right)$, whereas
$\Lone=n/\left(n+z^{2}\right)$; the two differ, and $\Lthree$ is the correct
bound. At the same configuration $\tau=1+z^{2}/(2n)$ yields $x^{+}=(1,0)$ and
$\Uthree=1=\Uone$, so the upper bound is unambiguous.

\emph{Consistency with the Karush--Kuhn--Tucker conditions.} Although the argument
above is self-contained, it is worth recording that it agrees with the classical
first-order analysis \citep{Kuhn1951,Karush1939,Kjeldsen2000}, which the objective
and constraints render both necessary and sufficient here: the objective is
linear, the feasible set $\Sx$ is the intersection of the sublevel set of a
positive-definite quadratic form with a polytope and hence convex, and it has
non-empty interior, so Slater's condition holds. For the minimisation of
$p_1-p_2$ the stationarity and complementary-slackness conditions read
\begin{align*}
1+\mu_1\left(2Ap_1+Bp_2+D\right)-\mu_2+\mu_3&=0,\\
-1+\mu_1\left(Bp_1+2Cp_2+E\right)-\mu_4+\mu_5&=0,\\
\mu_1\left(Ap_1^{2}+Bp_1p_2+Cp_2^{2}+Dp_1+Ep_2+F\right)&=0,\\
\mu_2\left(-p_1\right)=\mu_3\left(p_1-1\right)=\mu_4\left(-p_2\right)=\mu_5\left(p_2-1\right)&=0,
\end{align*}
with $A,\dots,F$ as in \eqref{eq:ellipseprob} and $\mu_i\ge0$. Dual feasibility
admits exactly five families of solutions, distinguished by which constraint of
the unit square is active: the four edge solutions, with active constraint $p_1\le1$, $p_1\ge0$, $p_2\ge0$
or $p_2\le1$, on which the objective attains the values
$L^{W}(\hp,n_2)$, $-U^{W}(-\hp,n_2)$, $L^{W}(\hp,n_1)$ and $-U^{W}(-\hp,n_1)$
respectively; by Lemma~\ref{lem:rootlocation}$(c)$ only the two edges governed by
$\nmin$ can furnish the optimum, and their values are $\Lone$ and $\Ltwo$,
attained at $\left(1,1-\Lone\right)$ or $\left(\Lone,0\right)$ and at
$\left(0,-\Ltwo\right)$ or $\left(1+\Ltwo,1\right)$ according as $n_2$ or $n_1$
equals $\nmin$; together with the interior solution $\mu_2=\dots=\mu_5=0$ and
optimiser $x^{-}$ of \eqref{eq:xpm}, for which
$\mu_1=\sqrt{PS}\big/\left(z\sqrt{\Xi}\right)$ is well defined for all inputs
because $\Xi>0$. Imposing primal feasibility on these five candidates selects, in
each region of the $(\hp,z)$ domain, exactly the bound identified above. The
maximisation is obtained either by the same procedure or by the equivariance
$\hp\mapsto-\hp$. This completes the proof.
\end{proof}

\subsection{Proof of Corollary \ref{cor:EOadmissible}}
\label{app:EOadmissible}
\begin{proof}
$(a)$ Taking $p=\hp$ in \eqref{eq:EOprofile} gives
$\left(\hp-\hp\right)^{2}=0\le z^{2}V^{\ast}(\hp)$, which holds because
$V^{\ast}\ge0$ on $[-1,1]$; hence $\hp$ belongs to the interval.

$(b)$ The representation \eqref{eq:EOprofile} restricts $p$ to $[-1,1]$ by
construction, and at $p=\pm1$ we have $V^{\ast}(\pm1)=0$ by
Lemma~\ref{lem:profile}, so the defining inequality becomes
$\left(\hp\mp1\right)^{2}\le0$, satisfied only when $\hp=\pm1$; in that case it is
satisfied with equality, so the endpoint is attained.

$(c)$ Each endpoint equals $g\left(p_1^{\ast},p_2^{\ast}\right)$ for an optimiser
$\left(p_1^{\ast},p_2^{\ast}\right)\in\Sx\subseteq[0,1]^{2}$, exhibited explicitly
in the proof of Theorem~\ref{thm:mainCI}: either the interior point $x^{\pm}$ of
\eqref{eq:xpm} or one of the four edge points listed there. Since $g(\Sx)$ is the
whole interval, by Lemma~\ref{lem:profile}, every interior value is likewise
attained.

$(d)$ Monotonicity in $z$ and the two limits are immediate from
\eqref{eq:EOprofile}: the right-hand side $z^{2}V^{\ast}(p)$ is increasing in $z$
for each fixed $p$, so the sublevel set grows with $z$; as $z\downarrow0$ it
shrinks to $\{p:\left(\hp-p\right)^{2}\le0\}=\{\hp\}$, and as $z\to\infty$ it
increases to $\{p:V^{\ast}(p)>0\}\cup\{\hp\}$, whose closure is $[-1,1]$ since
$V^{\ast}>0$ on $(-1,1)$. For continuity, note that
$(p,\hp,z)\mapsto\left(\hp-p\right)^{2}-z^{2}V^{\ast}(p)$ is jointly continuous
and, for fixed $(\hp,z)$, is a convex function of $p$ that is strictly negative at
$p=\hp$ whenever $\left|\hp\right|<1$; the endpoints of the sublevel set of a
convex function with a strict interior point depend continuously on the parameters,
whence $L$ and $U$ are continuous. In particular the two candidate bounds agree on
the common boundary of any two adjacent regions of Table~\ref{tab:CIresults}, so
the strict and weak inequalities separating those regions may be exchanged without
altering any computed interval. Finally, the convention \eqref{eq:Rconventions} is exactly what preserves
continuity at $\left|\hp\right|=1$ in the balanced design: letting $\hp\uparrow1$
with $n_1=n_2$ keeps the configuration in case $(6)$, so the limiting bounds are
$\left(\Lthree,\Uthree\right)$ evaluated at $\hp=1$, and the convention
reproduces them by placing $\hp=1$ in case $(5)$, where $\Uone=\Uthree$.
\end{proof}

\subsection{Proof of Lemma \ref{lem:ErrorInversion}}
\label{app:ErrorInversion}
\begin{proof}
By Lemma~\ref{lem:profile}, the \EO{} interval computed from $\hp$ is
\begin{equation*}
\Bigl\{q\in[-1,1]:\ \left(\hp-q\right)^{2}\le z^{2}V^{\ast}(q)\Bigr\},
\end{equation*}
so a fixed $p\in[-1,1]$ belongs to it precisely when
$\left(\hp-p\right)^{2}\le z^{2}V^{\ast}(p)$. The quantity $V^{\ast}(p)$ is
determined by $p$, $n_1$ and $n_2$ alone and is independent of the data; the
inequality is therefore a statement about $\hp$ only, and, being a quadratic
inequality in $\hp$ with unit leading coefficient, it holds exactly on the
interval \eqref{eq:acceptance}. Since $V^{\ast}(p)\ge0$, that interval is
non-empty and is symmetric about $p$.

The identification of its endpoints with Definition~\ref{def:ErrorBoundsLU} is a
substitution. If $\left|p\right|>\hp_{\square}$, then
$V^{\ast}(p)=\left|p\right|\left(1-\left|p\right|\right)/\nmin$ by
\eqref{eq:Vstar}, and $p\mp z\sqrt{V^{\ast}(p)}$ are $\Lsone$ and $\Usone$. If
$\left|p\right|\le\hp_{\square}$, then
\begin{equation*}
z\sqrt{V^{\ast}(p)}
=z\sqrt{\frac{\nmax+\nmin}{4\nmax\nmin}-\frac{p^{2}}{\nmax+\nmin}}
=\frac{z}{2}\sqrt{\frac{\left(\nmax+\nmin\right)^{2}-4\nmax\nmin p^{2}}{\nmax\nmin\left(\nmax+\nmin\right)}},
\end{equation*}
so that $p\mp z\sqrt{V^{\ast}(p)}$ are $\Lstwo$ and $\Ustwo$.

We remark, for the reader who wishes to reconcile this with a direct inversion of
Table~\ref{tab:CIresults}, that inverting the six cases separately reproduces the
same interval. In that route each case contributes a sub-interval of $\hp$-values,
obtained by inverting both the bound inequality and the case condition, and the
sub-intervals must then be united; the thresholds that arise are precisely the
landmark values of the partition of Figure~\ref{fig:CIsregionshpz}, namely
$\Rone\left(\pm1\right)=\sqrt{\nmin\left(\nmax-\nmin\right)/\left(\nmax+\nmin\right)}$,
$\Rone(0)=\Rtwo(0)=\sqrt{\nmin\left(\nmax+\nmin\right)/\left(\nmax-\nmin\right)}$
and
$\Rtwo\left(\pm1\right)=\left(3\nmax+\nmin\right)\sqrt{\nmin/\left(\nmax^{2}-\nmin^{2}\right)}$.
The union of the six sub-intervals is \eqref{eq:acceptance}; the case analysis is
therefore consistent with the lemma but superfluous, since the event
$\left\{p\in\mathrm{CI}\left(\hp\right)\right\}$ never needed to be decomposed
according to which case of Table~\ref{tab:CIresults} the observed $\hp$ falls in.
\end{proof}

\subsection{Proof of Proposition \ref{prop:ErrorBoundsOrder}}
\label{app:ErrorBoundsOrder}
\begin{proof}
The inequalities $\Lsone\le p\le\Usone$ and $\Lstwo\le p\le\Ustwo$ are immediate,
since both pairs are of the form $p\mp z\sqrt{\,\cdot\,}$ with a non-negative
radicand. It therefore suffices to prove
\begin{equation*}
\frac{\left|p\right|\left(1-\left|p\right|\right)}{\nmin}
\ \le\
\frac{\nmax+\nmin}{4\nmax\nmin}-\frac{p^{2}}{\nmax+\nmin}
\qquad\text{for all } p\in[-1,1],
\end{equation*}
that is, that the edge branch of \eqref{eq:Vstar} never exceeds the interior
branch. This is a structural fact rather than a computation: writing $S=\nmax+\nmin$,
the right-hand side is the unconstrained maximum of $V$ over the whole line
$\left\{\left(q_1,q_2\right):q_1-q_2=p\right\}$, attained at
$\left(\tfrac12+n_1p/S,\ \tfrac12-n_2p/S\right)$ by the proof of
Lemma~\ref{lem:profile}, whereas the left-hand side is the maximum of the same
function over the sub-segment of that line lying in $[0,1]^{2}$. A maximum over a
subset cannot exceed the maximum over the whole set.

For completeness we verify the inequality directly. Setting
$u=\left|p\right|$ and clearing denominators, the claim is
\begin{equation*}
4\nmax\,u\left(1-u\right)S\le S^{2}-4\nmax\nmin u^{2},
\end{equation*}
that is, $0\le S^{2}-4\nmax Su+4\nmax\left(S-\nmin\right)u^{2}
=S^{2}-4\nmax Su+4\nmax^{2}u^{2}=\left(S-2\nmax u\right)^{2}$,
using $S-\nmin=\nmax$. The expression is a perfect square, hence non-negative,
and vanishes exactly when $u=S/\left(2\nmax\right)=\hp_{\square}$. This proves
the ordering, and identifies the case of equality as
$\left|p\right|=\hp_{\square}$, the point at which the two branches of
\eqref{eq:Vstar} meet.
\end{proof}

\subsection{Proof of Theorem \ref{thm:mainError}}
\label{app:mainError}
\begin{proof}
\emph{The excess identity \eqref{eq:excessIdentity}.} Write $S=n_1+n_2$ and
$P=n_1n_2$, and set $u_i=p_i-\tfrac12$, so that $p=p_1-p_2=u_1-u_2$ and
$p_i\left(1-p_i\right)=\tfrac14-u_i^{2}$. Then
\begin{equation*}
V\left(p_1,p_2\right)
=\frac{\tfrac14-u_1^{2}}{n_1}+\frac{\tfrac14-u_2^{2}}{n_2}
=\frac{1}{P}\left[\frac{S}{4}-\left(n_2u_1^{2}+n_1u_2^{2}\right)\right],
\end{equation*}
while, on the interior branch,
\begin{equation*}
V^{\ast}(p)=\frac{S}{4P}-\frac{p^{2}}{S}
=\frac{1}{P}\left[\frac{S}{4}-\frac{P\left(u_1-u_2\right)^{2}}{S}\right].
\end{equation*}
Subtracting,
\begin{equation*}
P\,S\left[V^{\ast}(p)-V\left(p_1,p_2\right)\right]
=S\left(n_2u_1^{2}+n_1u_2^{2}\right)-P\left(u_1-u_2\right)^{2}.
\end{equation*}
Expanding the right-hand side with $S=n_1+n_2$ and $P=n_1n_2$,
\begin{align*}
S\left(n_2u_1^{2}+n_1u_2^{2}\right)-P\left(u_1-u_2\right)^{2}
&=n_1n_2u_1^{2}+n_2^{2}u_1^{2}+n_1^{2}u_2^{2}+n_1n_2u_2^{2}
-n_1n_2u_1^{2}+2n_1n_2u_1u_2-n_1n_2u_2^{2}\\
&=n_2^{2}u_1^{2}+2n_1n_2u_1u_2+n_1^{2}u_2^{2}
=\left(n_2u_1+n_1u_2\right)^{2}=\delta\left(p_1,p_2\right)^{2},
\end{align*}
which is \eqref{eq:excessIdentity}. Dividing by $PSV$ and adding one gives the
stated expression for $\mathcal{R}^{2}$. Since the right-hand side is a square,
$V^{\ast}(p)\ge V\left(p_1,p_2\right)$ and hence $\mathcal{R}\ge1$; on the outer
branch the same inequality holds by Lemma~\ref{lem:profile}, because
$\left(p_1,p_2\right)$ is one of the competitors in the maximisation
\eqref{eq:Vstardef} defining $V^{\ast}$.

\emph{The coverage formula \eqref{eq:coverageEO}.} By
Lemma~\ref{lem:ErrorInversion} the event
$\left\{\left(p_1,p_2\right)\in\Sr\right\}$ coincides with
$\bigl\{\left|\hp-p\right|\le z\sqrt{V^{\ast}(p)}\bigr\}$. Under the normal
approximation $\hp\sim N\left(p,V\right)$ with $V>0$, the standardised variable
$\left(\hp-p\right)/\sqrt{V}$ is standard normal, so
\begin{equation*}
\mathbb{P}\left(\left(p_1,p_2\right)\in\Sr\right)
=\mathbb{P}\!\left(\left|\frac{\hp-p}{\sqrt{V}}\right|\le z\sqrt{\frac{V^{\ast}(p)}{V}}\right)
=2\Phi\left(z\mathcal{R}\right)-1 .
\end{equation*}

\emph{The error formula \eqref{eq:errorformula}.} By
\eqref{eq:Omega_variance_form}, membership of $\Sz$ is the inequality
$\left(\hp-p_1+p_2\right)^{2}\le z^{2}V\left(p_1,p_2\right)$, that is
$\left|\hp-p\right|\le z\sqrt{V}$; and since $\left(p_1,p_2\right)\in[0,1]^{2}$ by
hypothesis, membership of $\Sz$ and of $\Sx=\Sz\cap[0,1]^{2}$ are the same event.
Hence
\begin{equation*}
\mathbb{P}\left(\left(p_1,p_2\right)\in\Sx\right)
=\mathbb{P}\!\left(\left|\frac{\hp-p}{\sqrt{V}}\right|\le z\right)=2\Phi(z)-1=1-\a ,
\end{equation*}
exactly, under the normal approximation. Subtracting from
\eqref{eq:coverageEO} gives \eqref{eq:errorformula}.
\end{proof}

\subsection{Proof of Corollary \ref{cor:ErrorProperties}}
\label{app:ErrorProperties}
\begin{proof}
$(a)$ Since $\mathcal{R}\ge1$ and $\Phi$ is increasing,
$\mathrm{Error}=2\left[\Phi\left(z\mathcal{R}\right)-\Phi(z)\right]\ge0$; and
since $\Phi<1$,
$\mathrm{Error}<2\left[1-\Phi(z)\right]=\a$. For
$\left(p_1,p_2\right)\in(0,1)^{2}$ we have $V>0$, so $\mathcal{R}$ is finite and
the second inequality is strict. As $\left(p_1,p_2\right)\to(0,0)$ we have
$V\to0$ while $p\to0$ and hence
$V^{\ast}(p)\to\left(\nmax+\nmin\right)/\left(4\nmax\nmin\right)>0$, so
$\mathcal{R}\to\infty$ and $\Phi\left(z\mathcal{R}\right)\to1$, giving
$\mathrm{Error}\to\a$; the corner $(1,1)$ is identical. At the corners $(0,1)$ and
$(1,0)$ both $V$ and $V^{\ast}$ vanish and the limit of $\mathcal{R}$ is
direction-dependent: approaching $(0,1)$ along
$\left(\varepsilon_1,1-\varepsilon_2\right)$ gives
$\mathcal{R}^{2}\to\left(\varepsilon_1+\varepsilon_2\right)\big/
\left(\varepsilon_1\nmin/n_1+\varepsilon_2\nmin/n_2\right)$, which equals $1$ for
every direction when $n_1=n_2$ and lies in $\left[1,\nmax/\nmin\right]$ otherwise.

$(b)$ On the interior branch, \eqref{eq:excessIdentity} gives
$\mathcal{R}=1$ if and only if $\delta\left(p_1,p_2\right)=0$, that is
$n_2p_1+n_1p_2=\left(n_1+n_2\right)/2$; dividing by
$\left(n_1+n_2\right)/2$ puts this in the form \eqref{eq:valleyline} with the
stated $\b_1,\b_2$, which satisfy $\b_1+\b_2=2$ and $\b_1/\b_2=n_2/n_1$. The line
passes through $\left(\tfrac12,\tfrac12\right)$ and has slope $-n_2/n_1$; when
$n_1=n_2$ it is $p_1+p_2=1$. Its intersection with the unit square is exactly the
set of least-favourable pairs, parametrised by
$\left(\tfrac12+n_1p/S,\ \tfrac12-n_2p/S\right)$ for
$\left|p\right|\le\hp_{\square}$, as in the proof of Lemma~\ref{lem:profile}. On
the outer branch $\left|p\right|>\hp_{\square}$ the maximiser in
\eqref{eq:Vstardef} lies on an edge of the unit square, so no interior pair
attains $V^{\ast}$ and the error is strictly positive there.

$(c)$ Under $\left(p_1,p_2\right)\mapsto\left(1-p_1,1-p_2\right)$ the difference
$p$ changes sign, and $V^{\ast}$ is even by Lemma~\ref{lem:profile}; also
$p_i\left(1-p_i\right)$ is invariant, so $V$ is unchanged. Hence $\mathcal{R}$,
and therefore the error, is invariant. Directly,
$\delta\left(1-p_1,1-p_2\right)=n_2\left(\tfrac12-p_1\right)+n_1\left(\tfrac12-p_2\right)=-\delta\left(p_1,p_2\right)$,
which enters \eqref{eq:excessIdentity} squared. The second symmetry is immediate
from the fact that $V$, $V^{\ast}$ and $\hp_{\square}$ are invariant under the
simultaneous interchange of the two populations.

$(d)$ Put $n_1=c_1N$, $n_2=c_2N$. Then $\delta=N\left[c_2\left(p_1-\tfrac12\right)+c_1\left(p_2-\tfrac12\right)\right]$,
so $\delta^{2}=N^{2}\delta_0^{2}$ with $\delta_0$ free of $N$; also
$n_1n_2\left(n_1+n_2\right)=c_1c_2\left(c_1+c_2\right)N^{3}$ and $V=v_0/N$ with
$v_0=p_1\left(1-p_1\right)/c_1+p_2\left(1-p_2\right)/c_2$ free of $N$. Hence
\begin{equation*}
\mathcal{R}^{2}=1+\frac{N^{2}\delta_0^{2}}{c_1c_2\left(c_1+c_2\right)N^{3}\cdot v_0/N}
=1+\frac{\delta_0^{2}}{c_1c_2\left(c_1+c_2\right)v_0},
\end{equation*}
which is independent of $N$; the same conclusion holds on the outer branch, where
$V^{\ast}$ and $V$ are both proportional to $N^{-1}$. Consequently the error
\eqref{eq:errorformula} is likewise independent of $N$ and does not vanish in the
large-sample limit unless $\delta_0=0$. In the balanced case $c_1=c_2$ with
$p_1=p_2=q$ one has $p=0$, $V^{\ast}=1/(2n)$ and $V=2q(1-q)/n$, whence
$\mathcal{R}^{2}=\left[4q(1-q)\right]^{-1}$.
\end{proof}

\subsection{Proof of Proposition \ref{prop:WaldBias}}
\label{app:WaldBias}
\begin{proof}
Let $X_i\sim\mathrm{Bin}\left(n_i,p_i\right)$ and $\hp_i=X_i/n_i$, so that
$\mathbb{E}\left[\hp_i\right]=p_i$ and
$\mathbb{E}\bigl[\hp_i^{2}\bigr]=\Var\left[\hp_i\right]+p_i^{2}
=p_i\left(1-p_i\right)/n_i+p_i^{2}$. Then
\begin{equation*}
\mathbb{E}\left[\hp_i\left(1-\hp_i\right)\right]
=p_i-\frac{p_i\left(1-p_i\right)}{n_i}-p_i^{2}
=p_i\left(1-p_i\right)\left(1-\frac{1}{n_i}\right).
\end{equation*}
Dividing by $n_i$ and summing over $i=1,2$, and using the independence of the two
samples,
\begin{equation*}
\mathbb{E}\left[\widehat{\Var}\left[\hp\right]\right]
=\sum_{i=1}^{2}\frac{p_i\left(1-p_i\right)}{n_i}\left(1-\frac{1}{n_i}\right)
=\Var\left[\hp\right]-\sum_{i=1}^{2}\frac{p_i\left(1-p_i\right)}{n_i^{2}},
\end{equation*}
which is \eqref{eq:Waldbias}; the subtracted term is strictly positive whenever
$p_i\in(0,1)$ for at least one $i$. When $n_1=n_2=n$ the common factor
$\left(1-1/n\right)$ may be taken outside the sum, giving
$\mathbb{E}\bigl[\widehat{\Var}\left[\hp\right]\bigr]=\left(1-1/n\right)\Var\left[\hp\right]$.
\end{proof}

\newpage
\section{Additional Figures}
\renewcommand{\thesubsection}{\thesection.\roman{subsection}}

\subsection{Comparison of the \EO{}, Wald and Newcombe intervals against \texorpdfstring{$\widehat{p}_2$}{hatp2}}
\label{app:CICp}
\noindent
\begin{figure}[H]
\centering
    \begin{subfigure}[t]{0.48\textwidth}
        \centering
        \includegraphics[width=\linewidth]{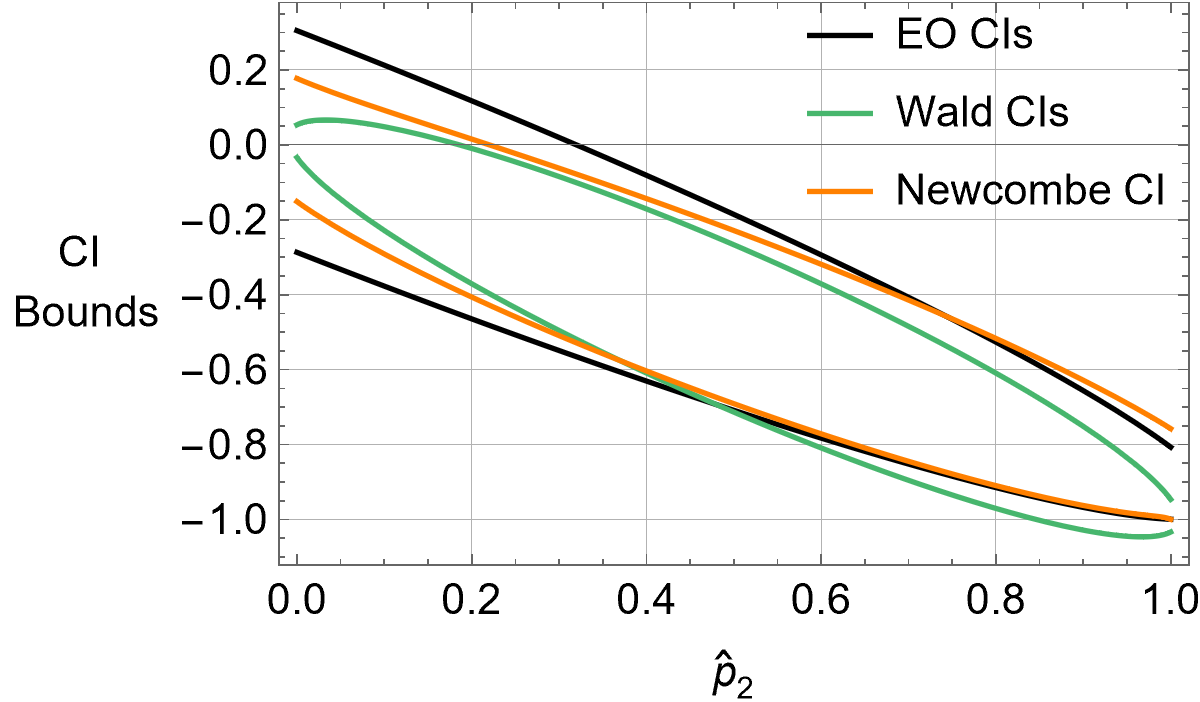}
        \caption{$\hp_1=0.01$}
        \label{subfig:CICp001}
    \end{subfigure}
    \hfill
    \begin{subfigure}[t]{0.48\textwidth}
        \centering
        \includegraphics[width=\linewidth]{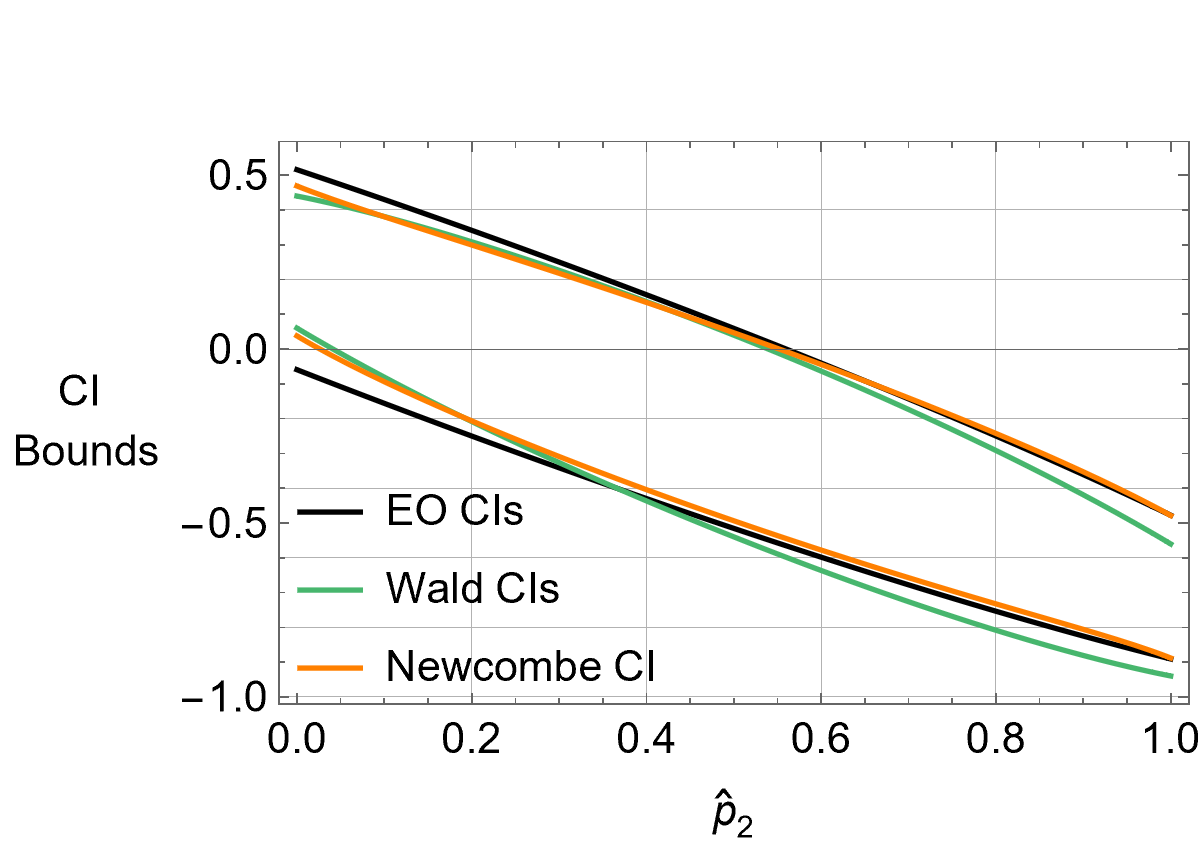}
        \caption{$\hp_1=0.25$}
        \label{subfig:CICp025}
    \end{subfigure}

    \vspace{0.8\baselineskip}

    \begin{subfigure}[t]{0.48\textwidth}
        \centering
        \includegraphics[width=\linewidth]{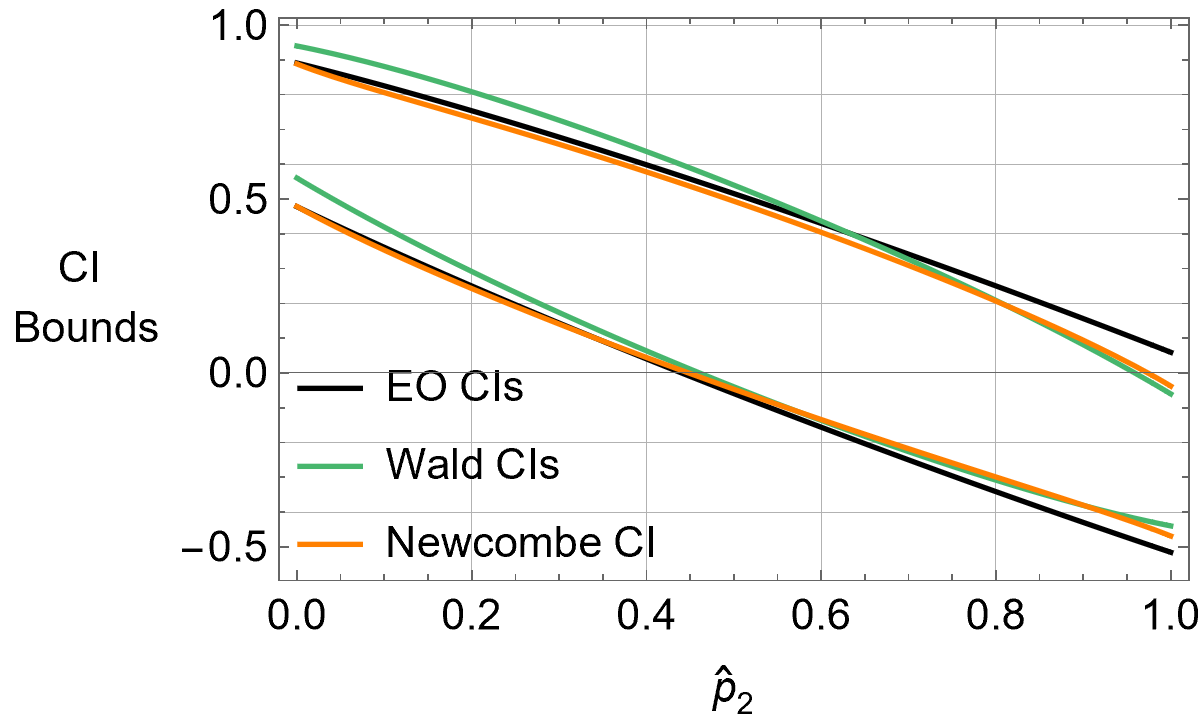}
        \caption{$\hp_1=0.75$}
        \label{subfig:CICp075}
    \end{subfigure}
    \hfill
    \begin{subfigure}[t]{0.48\textwidth}
        \centering
        \includegraphics[width=\linewidth]{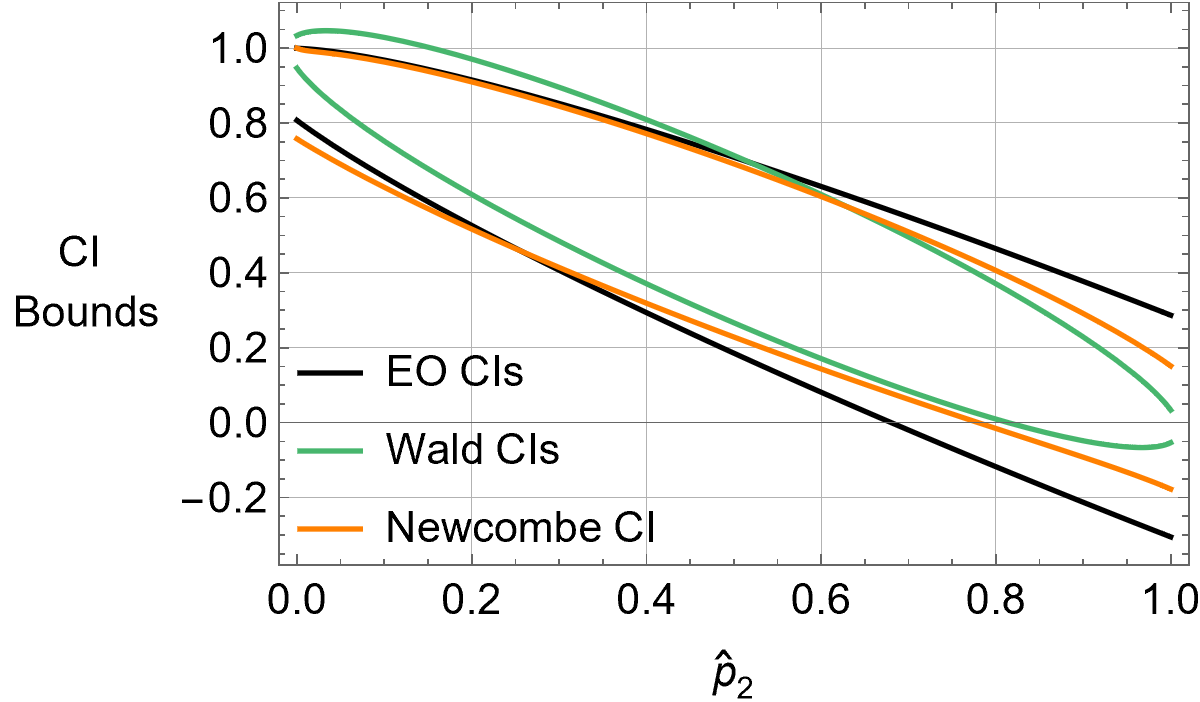}
        \caption{$\hp_1=0.99$}
        \label{subfig:CICp099}
    \end{subfigure}
\caption[Comparison of the \EO{}, Wald and Newcombe CIs against $\widehat{p}_2$]{Bounds of
the \EO{}, Wald and Newcombe intervals plotted against $\hp_2$ for four fixed
values of $\hp_1$, with $n_1=n_2=20$ and $\a=5\%$. These are the counterparts of
Figure~\ref{fig:CICq}, obtained by interchanging the roles of the two observed
proportions. Widths alone are not evidence of accuracy; the attained coverages of
the three procedures are compared in Section~\ref{sec:comparison} using
\eqref{eq:exactcoverage}.}
\label{fig:CICp}
\end{figure}

\newpage
\subsection{Coverage error when \texorpdfstring{$n_1>n_2$}{n1gn2}}
\label{app:Error-n1gn2}
\noindent
\begin{figure}[H]
\centering
    \begin{subfigure}[t]{0.48\textwidth}
        \centering
        \includegraphics[width=\linewidth]{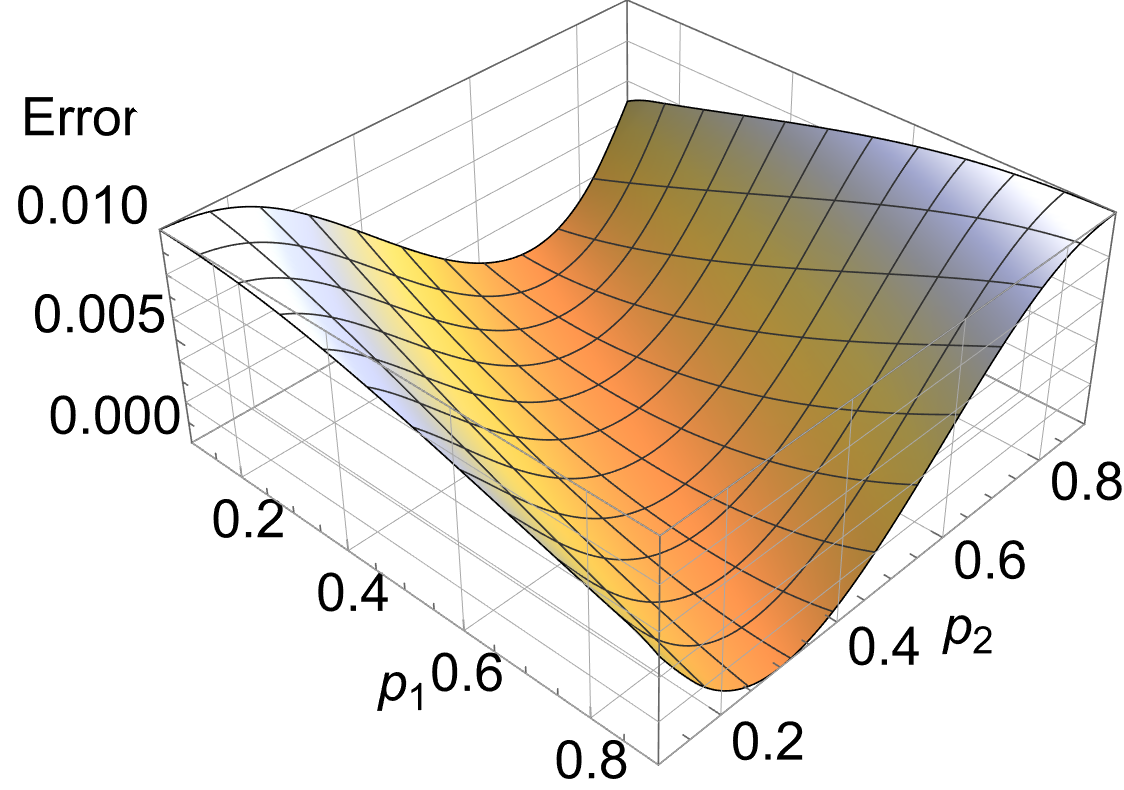}
        \caption{$\a=1\%$}
        \label{subfig:1_n1gn2}
    \end{subfigure}
    \hfill
    \begin{subfigure}[t]{0.48\textwidth}
        \centering
        \includegraphics[width=\linewidth]{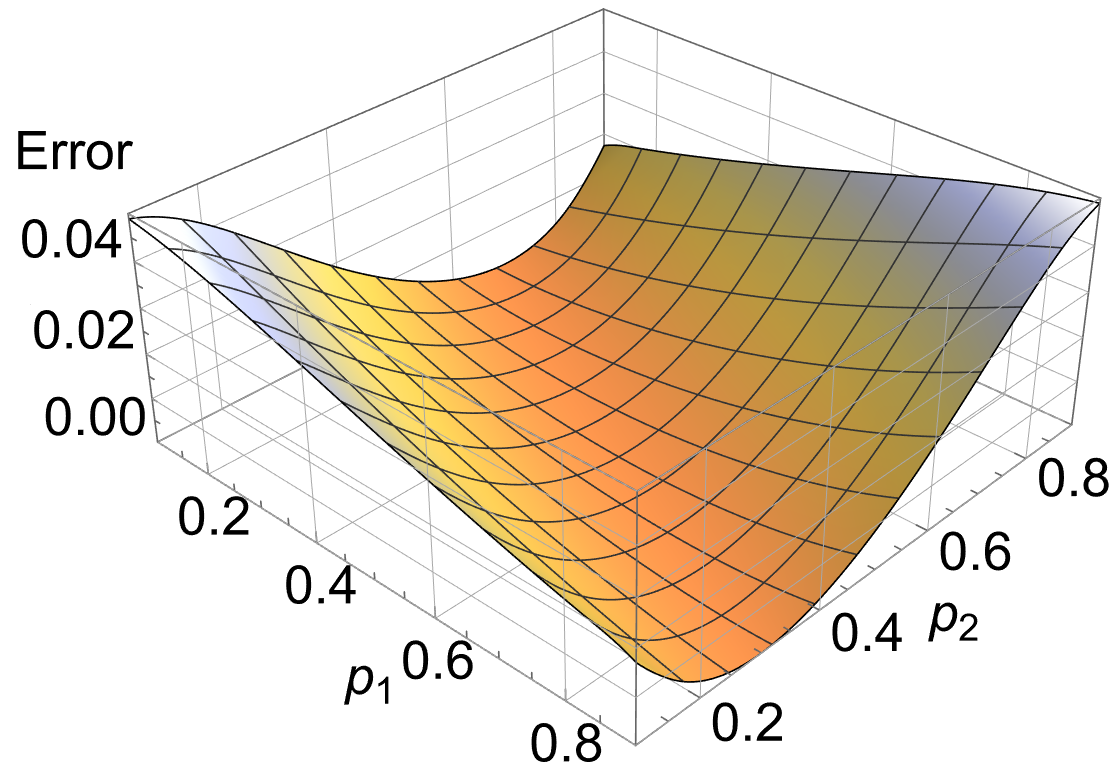}
        \caption{$\a=5\%$}
        \label{subfig:5_n1gn2}
    \end{subfigure}

    \vspace{0.8\baselineskip}

    \begin{subfigure}[t]{0.48\textwidth}
        \centering
        \includegraphics[width=\linewidth]{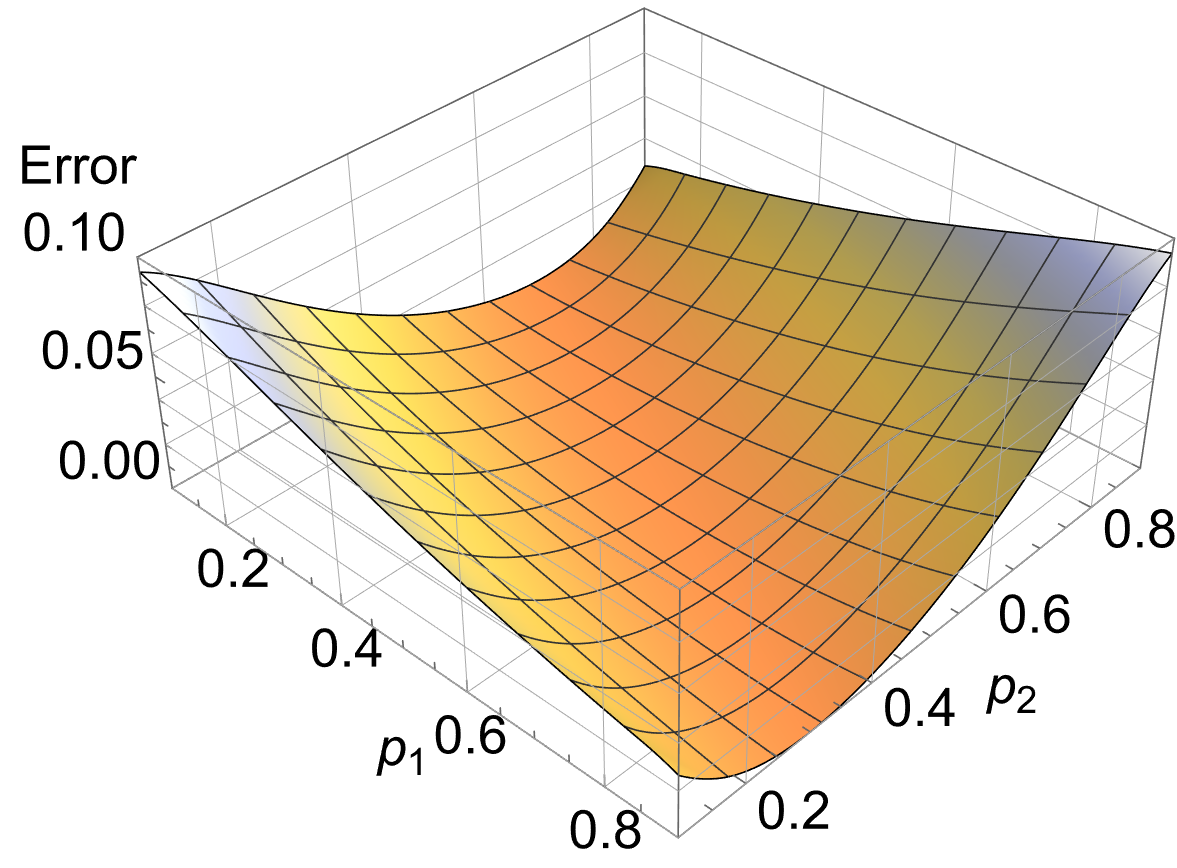}
        \caption{$\a=10\%$}
        \label{subfig:10_n1gn2}
    \end{subfigure}
\caption[Coverage error over the parameter domain for $n_1>n_2$]{Coverage error
\eqref{eq:errorformula} of the \EO{} interval over the parameter domain
$\left(p_1,p_2\right)$, for three significance levels, with $n_1=60$ and
$n_2=30$. The zero-error valley is the line $p_1+2p_2=\tfrac32$ of
\eqref{eq:valleyline}, of slope $-\tfrac12$; unlike the balanced case it runs
through the interior of the domain rather than corner to corner, so that the error
remains strictly positive at the corners $(p_1,p_2)$ with $p_1$ large and $p_2$
small.}
\label{fig:Error-n1gn2}
\end{figure}

\newpage
\subsection{Coverage error when \texorpdfstring{$\mathbf{n_1<n_2}$}{n1ln2}}
\label{app:Error-n2gn1}
\noindent
\begin{figure}[H]
\centering
    \begin{subfigure}[t]{0.48\textwidth}
        \centering
        \includegraphics[width=\linewidth]{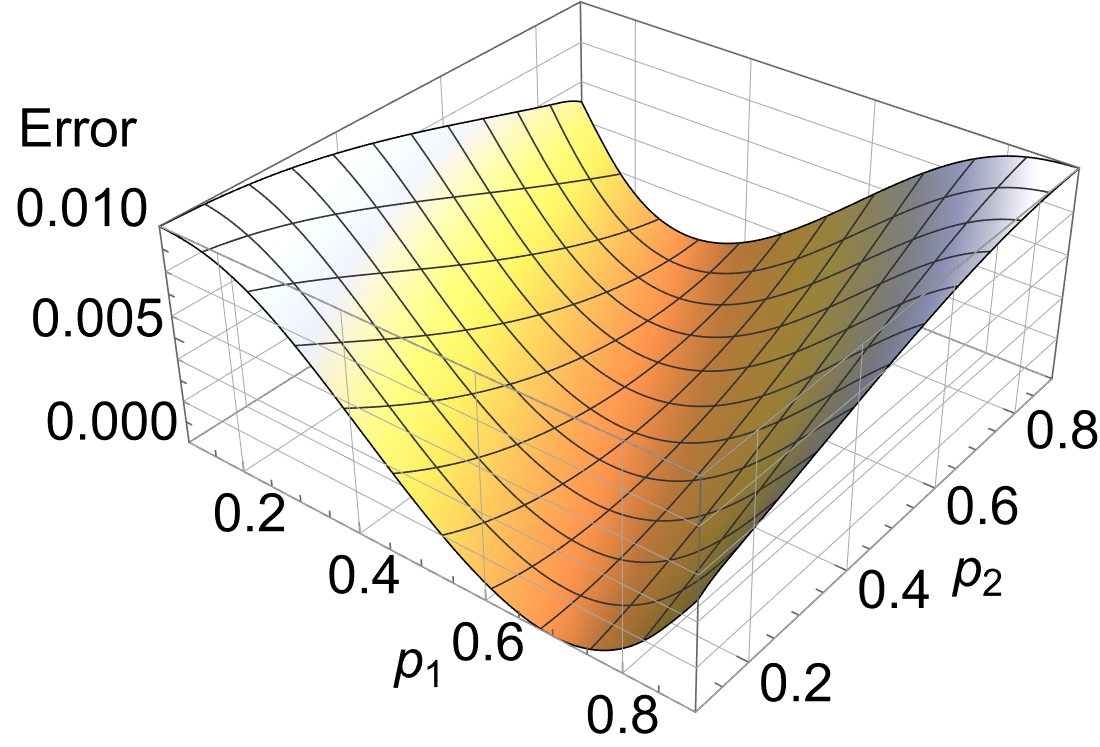}
        \caption{$\a=1\%$}
        \label{subfig:1_n2gn1}
    \end{subfigure}
    \hfill
    \begin{subfigure}[t]{0.48\textwidth}
        \centering
        \includegraphics[width=\linewidth]{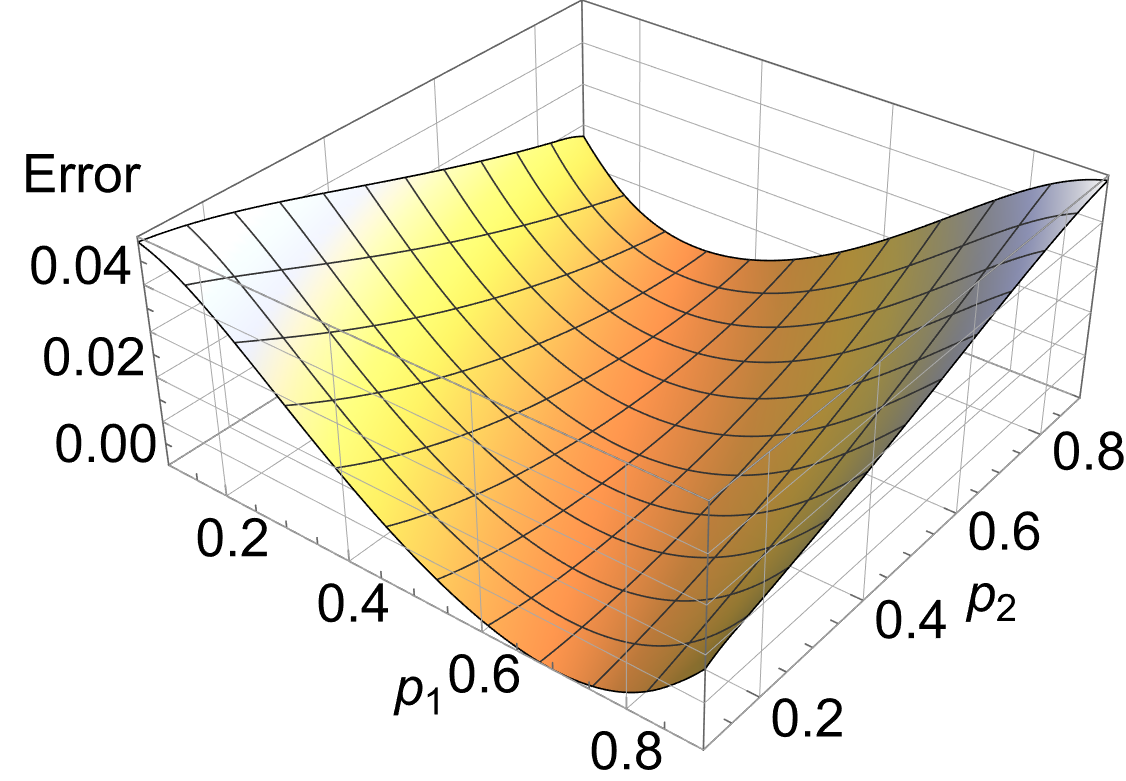}
        \caption{$\a=5\%$}
        \label{subfig:5_n2gn1}
    \end{subfigure}

    \vspace{0.8\baselineskip}

    \begin{subfigure}[t]{0.48\textwidth}
        \centering
        \includegraphics[width=\linewidth]{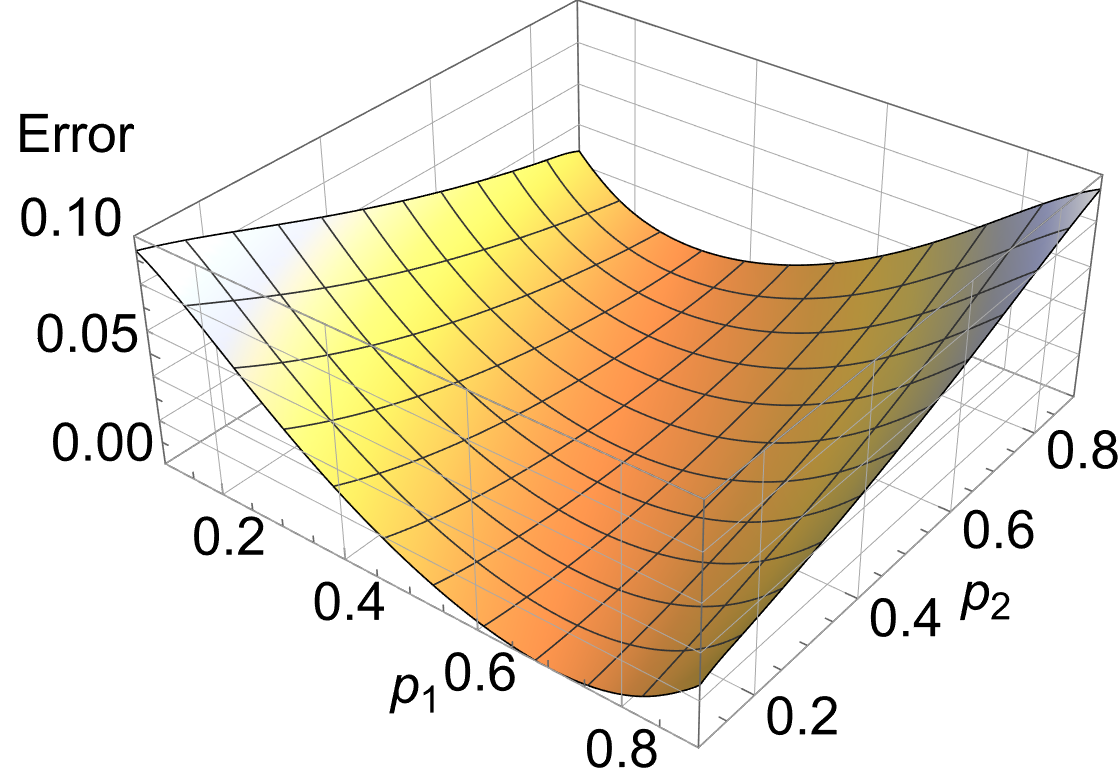}
        \caption{$\a=10\%$}
        \label{subfig:10_n2gn1}
    \end{subfigure}
\caption[Coverage error over the parameter domain for $n_1<n_2$]{Coverage error
\eqref{eq:errorformula} of the \EO{} interval over the parameter domain
$\left(p_1,p_2\right)$, for three significance levels, with $n_1=30$ and
$n_2=60$. The zero-error valley is the line $2p_1+p_2=\tfrac32$, of slope $-2$.
These surfaces are the mirror images of those in
Figure~\ref{fig:Error-n1gn2}, as required by the interchange symmetry of
Corollary~\ref{cor:ErrorProperties}$(c)$.}
\label{fig:Error-n2gn1}
\end{figure}
\newpage

\end{appendices}

\bibliography{References}

\end{document}